\pdfoutput=1
\documentclass[11pt]{article} 
\def\arXiv{1}  
\usepackage{enumerate}
\usepackage{amssymb}
\usepackage{amsbsy}
\usepackage{amsmath}
\usepackage{amsthm}
\usepackage{amsfonts}
\usepackage{latexsym}
\usepackage{graphicx}
\usepackage{mathtools}
\usepackage{xcolor}
\usepackage{enumitem}
\usepackage{thmtools}
\usepackage{thm-restate}
\usepackage{graphicx}
\usepackage{color}
\usepackage{bbm}
\usepackage{xifthen}
\usepackage{xspace}
\usepackage{mathtools}
\usepackage{titlesec}
\usepackage{setspace}
\usepackage{etoolbox}
\usepackage{xfrac}
\usepackage{esvect}
\usepackage{nicefrac}
\usepackage{cases}
\usepackage{empheq}
\usepackage{booktabs}
\usepackage{multirow}
\usepackage{xparse}
\usepackage{caption}
\usepackage{subcaption}
\usepackage{bbm}
\usepackage{placeins}

\newcommand{\notarxiv}[1]{foo}
\newcommand{\arxiv}[1]{ba}
\ifdefined \arXiv
\renewcommand{\arxiv}[1]{#1}%
\renewcommand{\notarxiv}[1]{\ignorespaces}%
\else%
\renewcommand{\arxiv}[1]{\ignorespaces}%
\renewcommand{\notarxiv}[1]{#1}%
\fi%

\arxiv{
	\usepackage[top=1in, right=1in, left=1in, bottom=1in]{geometry}
	\usepackage{url}
	\usepackage[hidelinks]{hyperref}
	\hypersetup{
		colorlinks=true,
		linkcolor=blue!70!black,
		citecolor=blue!70!black,
		urlcolor=blue!70!black}
	\usepackage[numbers, square]{natbib}
}

\notarxiv{
	\usepackage[subtle, all=normal, mathspacing=normal, floats=tight, 
	sections=tight]{savetrees}
	\usepackage[hidelinks=true]{hyperref}
}

\usepackage[linesnumbered,ruled,vlined]{algorithm2e}

\usepackage{cleveref}

\definecolor{darkblue}{rgb}{0,0,.75}
\newcommand{\mc}[1]{\mathcal{#1}}

\DeclarePairedDelimiter{\abs}{\lvert}{\rvert} %
\DeclarePairedDelimiter{\brk}{[}{]}
\DeclarePairedDelimiter{\crl}{\{}{\}}
\DeclarePairedDelimiter{\prn}{(}{)}

\DeclarePairedDelimiter{\norm}{\|}{\|}

\DeclareDocumentCommand\opnorm{ s o m }{%
	\IfBooleanTF{#1}{%
		\norm*{#3}_{\mathrm{op}}
	}{%
		\IfNoValueTF{#2}{%
			\norm{#3}_{\mathrm{op}}
		}{%
			\norm[#2]{#3}_{\mathrm{op}}
		}
	}
}

\newcommand{\inner}[2]{\left<#1,#2\right>}

\newcommand{\overeq}[1]{\overset{#1}{=}}
\newcommand{\overle}[1]{\overset{#1}{\le}}
\newcommand{\overge}[1]{\overset{#1}{\ge}}

\NewDocumentCommand\Ex{s O{} m }{%
	\mathbb{E}%
	\begingroup
	\IfBooleanTF{#1}
	{\ExInn*{#3}}
	{\ExInn[#2]{#3}}%
	\endgroup
}

\DeclarePairedDelimiterX\ExInn[1]{[}{]}{%
	\activatebar
	#1%
}

\RenewDocumentCommand\Pr{sO{}r()}{%
	\mathbb{P}%
	\begingroup
	\IfBooleanTF{#1}
	{\PrInn*{#3}}
	{\PrInn[#2]{#3}}%
	\endgroup
}

\DeclarePairedDelimiterX\PrInn[1](){%
	\activatebar
	#1%
}

\newcommand{\activatebar}{%
	\begingroup\lccode`~=`|
	\lowercase{\endgroup\def~}{\;\delimsize\vert\;}%
	\mathcode`|=\string"8000
}

\newcommand\numberthis{\addtocounter{equation}{1}\tag{\theequation}}

\newcommand{\defeq}{\coloneqq}

\newcommand{\half}{\frac{1}{2}}
\newcommand{\indic}[1]{\mathbbm{1}_{\!\left\{#1\right\}}} %
\newcommand{\R}{\mathbb{R}}
\newcommand{\Z}{\mathbb{Z}}
\makeatletter
\long\def\@makecaption#1#2{
  \vskip 0.8ex
  \setbox\@tempboxa\hbox{\small {\bf #1:} #2}
  \parindent 1.5em  %
  \dimen0=\hsize
  \advance\dimen0 by -3em
  \ifdim \wd\@tempboxa >\dimen0
  \hbox to \hsize{
    \parindent 0em
    \hfil 
    \parbox{\dimen0}{\def\baselinestretch{0.96}\small
      {\bf #1.} #2
    } 
    \hfil}
  \else \hbox to \hsize{\hfil \box\@tempboxa \hfil}
  \fi
}
\makeatother

\definecolor{innerboxcolor}{rgb}{.9,.95,1}
\definecolor{outerlinecolor}{rgb}{.6,0,.2}

\newcommand{\E}{\mathbb{E}} %
\providecommand{\argmin}{\mathop{\rm argmin}}

\providecommand{\abs}{\mathop{\rm abs}}

\newtheorem{theorem}{Theorem}
\newtheorem{lemma}{Lemma}

\newtheorem{proposition}{Proposition}
\newtheorem{definition}{Definition}

\newtheorem*{claim*}{Claim}

\newcounter{example}

\newenvironment{example*}[1][]{
  \ifthenelse{\isempty{#1}}{%
    \noindent \textbf{Example:}\hspace*{.05em}
  }{%
    \noindent \textbf{Example} ({#1})\textbf{:}\hspace*{.05em}
  }
}{%
  $\Diamond$ \bigskip
}

\def\keywordname{{\bfseries \emph Keywords}}%
\def\keywords#1{\par\addvspace\medskipamount{\rightskip=0pt plus1cm
    \def\and{\ifhmode\unskip\nobreak\fi\ $\cdot$
    }\noindent\keywordname\enspace\ignorespaces#1\par}}

\newcommand{\Otil}[1]{\widetilde{O}( #1 )}

\newcommand{\grad}{\nabla}
\newcommand{\hess}{\nabla^2}
\newcommand{\del}{\partial}
\newcommand{\xopt}{x_\star}
\newcommand{\Var}{\mathrm{Var}}

\SetCommentSty{mycommfont}
\SetKwInput{KwInput}{Input}                %
\SetKwInput{KwOutput}{Output}              %
\SetKwInput{KwReturn}{Return}              %
\SetKwInput{KwParameters}{Parameters}
\SetKwInput{KwInput}{Input} 
\SetKwInput{KwParameter}{Parameters}                %
\SetKwProg{Fn}{function}{}{}
\SetKwInput{KwOutput}{Output}              %
\SetKwInput{KwReturn}{Return}              %
\SetKwComment{Comment}{$\triangleright$\ }{}
\SetKwRepeat{Do}{do}{while}
\SetKwBlock{Repeat}{Repeat}{}
\SetCommentSty{mycommfont}

\makeatletter
\newcommand\Block[2]{%
	#1%
	\algocf@group{#2}%
}
\makeatother

\newcommand{\oracle}[1][]{\mathcal{O}_{\textup{\textsf{#1}}}}
\newcommand{\proj}{\mathsf{Proj}}

\newcommand{\xset}{\mathcal{X}}

\newcommand{\xhat}{\hat{x}}
\newcommand{\hx}{\xhat}
\newcommand{\hg}{\hat{g}}

\newcommand{\bx}{\bar{x}}

\newcommand{\QT}{\mathcal{Q}_{T}}
\newcommand{\QTS}{\mathcal{Q}_{\widehat{T}}}

\newcommand{\SupT}[1][T]{\mathcal{S}^{\ge}_{#1}}
\newcommand{\SequalT}{\mathcal{S}^{=}_{T}}
\newcommand{\SsupT}[1][T]{\mathcal{S}^{>}_{#1}}
\newcommand{\SdownT}[1][T]{\mathcal{S}^{\le}_{#1}}
\newcommand{\SsdownT}{\mathcal{S}^{<}_T}

\newcommand{\lamminus}{\lambda_{1/2}}
\newcommand{\xminus}{x_{1/2}}
\newcommand{\tlambda}{\tilde{\lambda}}
\newcommand{\tw}{\tilde{w}}

\newcommand{\blambda}{\lambda'}
\newcommand{\MSoracle}{\oracle}

\newcommand{\ba}{a'}
\newcommand{\bA}{A'}
\newcommand{\tx}{\tilde{x}}

\newcommand{\blambdainit}{\blambda_{0}}

\newcommand{\T}{\top}

\newcommand{\lazyflag}{\textsc{lazy}}
\newcommand{\mscheck}{\textsc{CheckMS}}
\newcommand{\minresnewton}{\textsc{ConjRes}}
\newcommand{\lamvalid}{\lambda_{\mathrm{vld}}}
\newcommand{\laminvalid}{\lambda_{\mathrm{invld}}}
\newcommand{\failedbefore}{\textsc{FailedCheck}}

\newcommand{\adaOracle}{\mathcal{O}_{\textsf{\textup{aMSN}}}}
\newcommand{\foadaOracle}{\mathcal{O}_{\textsf{\textup{aMSN-fo}}}}

\newcommand{\filt}[1][t]{\mathcal{F}_{#1}}

\newcommand{\Holder}{H\"older\xspace}

\newcommand{\oracleP}[1][p,\nu]{\mathcal{O}_{#1\textsf{-reg}}}
\newcommand{\oracleBall}[1][r]{\mathcal{O}_{#1\textsf{-ball}}}
\newcommand{\oracleBaCoN}[1][r]{\mathcal{O}_{#1\textsf{-BaCoN}}}

\usepackage[textsize=scriptsize,textwidth=2cm]{todonotes}
\newcommand{\yairside}[1]{\todo[color=blue!10]{Yair: #1}}
\newcommand{\yair}[1]{{\bf \color{blue} Yair: #1}}
\newcommand{\yujia}[1]{{\bf \color{red} Yujia: #1}}
\newcommand{\sidford}[1]{{\bf \color{purple} Sidford: #1}}
\newcommand{\arun}[1]{{\bf \color{orange} Arun: #1}}
\newcommand{\danielle}[1]{{\bf \color{brown} Danielle: #1}}

 \renewcommand{\yairside}[1]{\ignorespaces}
 \renewcommand{\yair}[1]{\ignorespaces}
 \renewcommand{\yujia}[1]{\ignorespaces}
 \renewcommand{\sidford}[1]{\ignorespaces}
 \renewcommand{\arun}[1]{\ignorespaces}
 \renewcommand{\danielle}[1]{\ignorespaces} 
\usepackage[utf8]{inputenc} %
\usepackage[T1]{fontenc}    %
\usepackage{url}            %
\usepackage{booktabs}       %
\usepackage{amsfonts}       %
\usepackage{nicefrac}       %
\usepackage{microtype}      %
\usepackage{xcolor}         %

\title{Optimal and Adaptive Monteiro-Svaiter Acceleration} %

\arxiv{
\author{%
	Yair Carmon\thanks{Tel Aviv University, \texttt{ycarmon@tauex.tau.ac.il} and \texttt{hausler@mail.tau.ac.il}.} 
	~~
	Danielle Hausler\footnotemark[1] ~~ 
	Arun Jambulapati\thanks{Stanford University, \texttt{\{jmblpati,yujiajin,sidford\}@stanford.edu}.} ~~
	Yujia Jin\footnotemark[2] ~~
	Aaron Sidford\footnotemark[2]
}
}

\notarxiv{
	\author{%
		Yair Carmon\footnotemark[1]\thanks{Tel Aviv University, \texttt{ycarmon@tauex.tau.ac.il}, \texttt{hausler@mail.tau.ac.il}}\And
		Danielle Hausler\footnotemark[1]\And
		Arun Jambulapati\thanks{Stanford University, \texttt{\{jmblpati,yujiajin,sidford\}@stanford.edu}}\And
		Yujia Jin\footnotemark[2]\And
		Aaron Sidford\footnotemark[2]
	}
}

\arxiv{\date{}}

\begin{document}

\maketitle

\begin{abstract}
We develop a variant of the Monteiro-Svaiter (MS) acceleration framework that removes the need to solve an expensive implicit equation at every iteration. Consequently, for any $p\ge 2$ we improve the complexity of convex optimization with Lipschitz $p$th derivative by a logarithmic factor,  matching a lower bound. We also introduce an MS subproblem solver that requires no knowledge of problem parameters, and implement it as either a second- or first-order method via exact linear system solution or MinRes, respectively. On logistic regression our method outperforms previous second-order acceleration schemes, but under-performs Newton's method; simply iterating our first-order adaptive subproblem solver performs comparably to L-BFGS.
\end{abstract}

\section{Introduction}

We consider the problem of minimizing a convex function $f:\xset\to\R$ 
over closed convex set $\xset\subseteq \R^d$, given access to an oracle $\oracle:\xset\to\xset$ that minimizes a local model of $f$ around a given query point. A key motivating example of such an oracle is the cubic-regularized Newton step
\begin{equation}\label{eq:cr-oracle}
	\oracle[cr](y) = \argmin_{x\in\xset}\crl*{ f(y) + \grad f(y)^{\T} (x-y) + \half (x-y)^{\T} \hess f(y) (x-y) + \frac{M}{6}\norm{x-y}^3},
\end{equation}
i.e., minimizing the second-order Taylor approximation of $f$ around $y$ plus a cubic regularization term. However, our results apply to additional oracles including a simple gradient step, regularized higher-order Taylor expansions~\cite{baes2009estimate,gasnikov2019optimal, bubeck2019optimal, jiang2019optimal, bullins2020highly, nesterov2021implementable, grapiglia2020tensor,song2021unified, nesterov2021superfast, kamzolov2020near}, ball-constrained optimization~\cite{carmon2020acceleration}, and new adaptive oracles that we develop.

Seminal work by \citet{MonteiroS13a} (MS) shows how to accelerate the basic oracle iteration $x_{t+1} = \oracle(x_t)$. Their algorithm is based on the fact that many oracles, including $\oracle[cr]$, implicitly approximate proximal points. That is, for every $y$ and  $x=\oracle(y)$, there exists $\lambda_{x,y} > 0$ such that $x \approx \argmin_{x'\in\xset} \crl[\big]{ f(x') + \frac{1}{2}\lambda_{x,y}\norm{x'-y}^2}$, with the approximation error  controlled by a specific condition they define. MS prove that, under this  condition, the accelerated  proximal point method~\cite{guler1992new,salzo2012inexact} (with dynamic regularization parameter) maintains its rate of convergence. Applying their framework to $\oracle[cr]$ and assuming $\hess f$ is Lipschitz, they achieve error bounds that decay as $O(t^{-7/2}\log t)$ after $t$ oracle calls, improving the $O(t^{-2})$ rate of the basic $\oracle[cr]$ iteration~\cite{nesterov2006cubic} and the  $O(t^{-3})$ rate of an earlier accelerated method~\cite{nesterov2008accelerating}. Subsequent works apply variations of the MS framework to different oracles, obtaining improved theoretical guarantees for functions with continuous higher-order derivatives~\cite{gasnikov2019optimal, bubeck2019optimal, jiang2019optimal,song2021unified,alves2021variants}, parallel  optimization~\cite{bubeck2019complexity}, logistic and $\ell_\infty$ regression~\cite{bullins2020highly,carmon2020acceleration}, minimizing functions with H\"older continuous higher derivatives \cite{song2021unified}, and distributionally-robust optimization~\cite{carmon2021thinking,carmon2022distributionally}. 

However, all of these algorithms based on the MS framework share a common drawback: the iterate $y_t$ used to produce $x_{t+1}=\oracle(y_t)$ depends on the proximal parameter $\lambda_{t+1} = \lambda_{x_{t+1},y_t}$, which itself depends on both $x_{t+1}$ and $y_t$. This circular dependence necessitates solving an implicit equation for $\lambda_{t+1}$; MS (and many subsequent results based upon it) propose bisection procedures for doing so using a number of oracle calls logarithmic in the problem parameters.
From a theoretical perspective, the additional bisection complexity introduces a logarithmic gap between the upper bounds due to MS-based algorithms and the best known lower bounds \cite{arjevani2019oracle,grapiglia2020tensor} in a number of settings.  %

From a practical perspective, the use of bisection in the MS framework is undesirable as it potentially discards the optimization progress made by oracle calls during each bisection. In his textbook, \citet[\S4.3.3]{nesterov2018lectures} argues that the logarithmic cost of bisection likely renders the MS scheme for accelerating $\oracle[cr]$ inferior in practice to algorithms whose error decays at the asymptotically worse rate $O(t^{-3})$ but do not require bisection; he notes that removing the bisection from the MS algorithm is an ``open and challenging question in Optimization Theory.''
\citet{carmon2021thinking} also point out bisection as one of the main hurdles in making their theoretical scheme practical, while \citet{song2021unified} note this limitation and propose a heuristic alternative to bisection. 
  (See \Cref{apdx:related} for extended discussion of related work, including a concurrent and independent result by \citet{kovalev2022first}.)

\subsection{Our contributions}

We settle this open question, providing a variant of MS acceleration that does not require bisection (\Cref{sec:framework}). When combined with certain existing MS oracles (\Cref{ssec:solvers-prior}), our algorithm obtains complexity bounds that are optimal up to constant factors, improving over prior art by a logarithmic factor (see \Cref{table:summary}). In addition, our algorithm has no  parameters sensitive to tuning. 

We then go a step further and  (in \Cref{ssec:solvers-amsn}) develop an adaptive alternative to $\oracle[cr]$~(\Cref{eq:cr-oracle}). Our oracle does not require tuning the parameter $M$, which in theory should be proportional to the (difficult to estimate) Lipschitz constant of $\hess f$. Using our oracle, we obtain the optimal Hessian evaluation complexity $O(t^{-(4+3\nu)/2})$ for functions with
order-$\nu$ \Holder Hessian (Lipschitz Hessian is the $\nu=1$ special case), without requiring any knowledge of the \Holder constant and order $\nu$. 
Our oracle is also efficient: while existing complexity bounds for computing $\oracle[cr]$ require a logarithmic number of linear system solutions per call, our oracle requires a double-logarithmic number. Moreover, when used with our acceleration method, the number of linear system solves per iteration is essentially constant.

We also provide a first-order implementation of our adaptive oracle (\Cref{ssec:solvers-amsn-fo}). It approximately solves linear systems via first-order operations (Hessian-vector products) using MinRes/Conjugate Residuals~\cite{stiefel1955relaxationsmethoden,fong2012cg} with a simple, adaptive, stopping criterion lifted directly from our analysis. Our oracle attains the optimal first-order evaluation complexity for smooth functions up to an \emph{additive} logarithmic term, without knowledge of the gradient Lipschitz constant or any parameter tuning. Moreover, it maintains an optimal outer iteration complexity for \Holder Hessian of any order.

Finally, we report empirical results (\Cref{sec:experiments}).\footnote{The code for our experiments is available at \url{https://github.com/danielle-hausler/ms-optimal}.} On logistic regression problems, combining our optimal acceleration scheme with our adaptive oracle outperforms previously proposed accelerated second-order methods. However, we also show that (while somewhat helpful for $\oracle[cr]$ with a conservative choice of $H$), adding momentum to well-tuned or adaptive  second-order methods is \emph{harmful} in logistic regression: simply iterating our oracle---or, better yet, applying Newton's method---dramatically outperforms all ``accelerated'' algorithms. This important fact seems to have gone unobserved in the literature on accelerated second-order methods, despite logistic regression appearing in many related experiments~\cite{song2021unified,doikov2020inexact,mischenko2021regularized,jiang2020unified}. Simply iterating our adaptive oracle outperforms the classical accelerated gradient descent, and performs comparably to L-BFGS. 

\subsection{Limitations and outlook}\label{ssec:limitation}

While our algorithms resolve an enduring theoretical open problem in convex optimization, and are free of sensitive parameters that typically hinder theoretically-optimal methods, practical performance remains a limitation. On logistic regression, Newton's method is remarkably fast, and our acceleration scheme does not seem to help our adaptive oracle. We do not fully understand why this is so, but we suspect that it has to do with additional structure in logistic regression, which Newton's method can automatically exploit but momentum cannot. We believe that future research should identify the structure that makes Newton's method so efficient, and modifying momentum schemes to leverage it. 

Scalability is another important limitation. While our first-order oracle significantly improves scalability over the second-order oracle from which it is built, it still relies on exact gradient and Hessian-vector products. Therefore, it will have difficulty scaling up to very large datasets. Nevertheless, we hope that further scalability improvements may be possible by building an oracle that utilizes cheap stochastic gradient estimates instead of exact gradients, bringing with it the exciting prospect of a new and powerful adaptive stochastic gradient method. The alternative, probabilistic approximation condition we propose in \Cref{app:gen-framework} might be helpful in this regard.

\arxiv{A final limitation is that our theory and experiments center around convex optimization problems, whereas optimization in modern machine learning are mainly non-convex. However, many of the central techniques in modern machine learning, including momentum~\cite{nesterov1983method} and adaptive gradient methods~\cite{duchi2011adaptive} were initially proposed and analyzed in the context of convex optimization. We believe that our techniques might also prove  useful beyond the convex optimization landscape.}
 \arxiv{%
\arxiv{\subsection{Additional related work}\label{apdx:related}}
\notarxiv{\section{Additional related work}\label{apdx:related}}

\paragraph{Bisection-free methods for variational inequalities.} 

Monteiro and Svaiter also proposed second-order methods for solving variational inequalities for monotone operators with continuous derivatives~\cite{monteiro2012iteration}, and subsequent work provided improved rates for variational inequalities with continuous higher-order derivatives via tensor methods~\cite{bullins2020higher,jiang2022generalized}. These works also feature an implicit equation over a scalar regularization/step-size parameter, that necessitates a bisection and increases complexity by a logarithmic factor.  In recent papers, \citet{lin2022perseus} and~\citet{adil2022line} remove that logarithmic factor by developing  bisection-free methods for variational inequalities. However, applying these methods directly to convex optimization with Lipschitz $p$th derivatives yields a rate of $O(t^{-(p+1)/2})$ rather than the  optimal $O(t^{-(3p+1)/2})$ rate of our method. Moreover, these works remove bisections using techniques fundamentally different from ours. In particular, they do not apply a damping scheme on the $A_t$ sequence, nor do they apply a multiplicative update for the regularization parameter.

\paragraph{Adaptive Newton and tensor methods.} 
A number of works consider adaptive variants of the cubic-regularized Newton method and its tensor counterparts. \citet{cartis2011adaptive,gould2012updating} propose adaptive variants of cubic regularization for non-convex optimization. For convex optimization, 
\citet{mischenko2021regularized} provides a simple adaptive scheme converging at rate $O(t^{-2})$, followed by an improvement in its guarantee to $O(t^{-3})$ \cite{doikov2021gradient}. For tensor methods, \citet{jiang2020unified} proposes an adaptive regularization scheme for convex functions with Lipschitz-continuous $p^{th}$ derivatives which achieves the rate $O(t^{-p-1})$. In addition, \citet{grapiglia2020tensor} gives analogous results under $\nu$-H\"older continuity of the derivatives. 

Even for the case second-order methods with Lipschitz-continuous Hessian, an adaptive scheme with optimal rate $O(t^{-3.5})$ remained open prior to this work. Beyond removing the bisection from the MS framework, our key algorithmic techniques include directly considering quadratically-regularized Newton step (similar to~\cite{mischenko2021regularized,doikov2021gradient} and different from~\cite{grapiglia2020tensor,jiang2020unified}), and using the original MS approximation condition for selecting an appropriate regularization parameter, which is new in the context of adaptive methods. These techniques  allow us to adapt to both the constant and order of Hessian \Holder continuity simultaneously. 

\paragraph{Comparison to~\cite{kovalev2022first}.} 
In concurrent and independent theoretical work, \citet{kovalev2022first} propose an algorithm that also attains the optimal $p$th derivative evaluation complexity for convex optimization with Lipschitz $p$th derivatives. While also inspired by MS acceleration, the algorithm of~\cite{kovalev2022first} is quite different from ours: they replace the implicitly defined regularization parameters inherent to MS oracles by approximating proximal points with explicit, predetermined regularization parameters. To obtain these proximal points they apply a tensor-extragradient method, and stop it when the MS condition is met. By careful analysis, they show that---even though an individual outer acceleration step requires multiple derivative evaluations---the overall complexity of their method is optimal. In contrast, our method makes a single oracle call (high-order derivative evaluation) per step. To obtain optimal complexity, our method relies on a dynamic sequence of ``guessed'' regularization parameters and a momentum damping schemes that handles cases where these guesses overshoot. The two works offer complementary viewpoints of the algorithmic innovation required for removing bisection from the MS framework. 

In comparison to~\citet{kovalev2022first}, we believe that our algorithm offers advantages in terms of generality and adaptivity. From a generality perspective, our algorithm applies to every setting where the original MS framework applies. In addition to functions with Lipschitz $p$th derivatives, that includes functions with \Holder derivatives~\cite{song2021unified}, ball minimization oracles~\cite{carmon2020acceleration}, and a second-order oracles for functions with Lipschitz third derivative~\cite{nesterov2021superfast,kamzolov2020near}. While extending~\cite{kovalev2022first} to these settings may be possible, it would require additional work in formulating and analyzing an appropriate subproblem solver. Regarding adaptivity, like the original MS framework, our algorithm is agnostic to both the order of the Lipschitz derivative and its corresponding Lipschitz constant.
In contrast, \citet{kovalev2022first} require the derivative order for determining the regularization parameters, and the Lipschitz constant for the subproblem solver. 

\notarxiv{\newpage} %

}

\section{Removing bisection from the Monteiro-Svaiter framework}\label{sec:framework}

\begin{figure}[ht]
	\begin{minipage}[t]{0.51\linewidth}%
		\centering
		\begin{algorithm}[H]
			\DontPrintSemicolon
			\caption{Optimal MS Acceleration}\label{alg:optms}
			\KwInput{Initial $x_0$, function $f$, oracle $\oracle$}
			\KwParameters{Initial $\blambdainit$, multiplicative adjustment factor $\alpha > 1$}
			
			Set $v_0 = x_0$, $A_0=0$\;
			$\tx_1, \lambda_1 = \oracle(x_0;\blambdainit)$~~,~~$\blambda_1 = \lambda_1$\;
			\For{$t=0,1,\ldots,$}{
				$\ba_{t+1} = \frac{1}{2\blambda_{t+1}}\prn*{1 + \sqrt{1 + 4\blambda_{t+1} A_t}}$\label{line:a-quandratic}\;
				$\bA_{t+1} = A_t + \ba_{t+1}$\;\label{line:A-update}
				$y_{t} = \frac{A_{t}}{\bA_{t+1}} x_t + \frac{\ba_{t+1}}{\bA_{t+1}} v_t$\;\label{line:y-update}
				\lIf{$t>0$}{
					$\tx_{t+1}, \lambda_{t+1} = \oracle(y_t; \blambda_{t+1})$}
				\If{$\lambda_{t+1} \le \blambda_{t+1}$}{
					$a_{t+1}=\ba_{t+1}$, $A_{t+1} = \bA_{t+1}$\;\label{line:update-begin-1}
					$x_{t+1} = \tx_{t+1}$\label{line:x-update-1} \;
					{\color{green!60!black}$\blambda_{t+2} = \frac{1}{\alpha}\blambda_{t+1}$}\label{line:lam-update-1} 
				}
				\Else{
					$\gamma_{t+1} =\frac{\blambda_{t+1}}{\lambda_{t+1}}$\;\label{line:gamma-define}
					$a_{t+1}=\gamma_{t+1}\ba_{t+1}$, ~$A_{t+1} = A_t + a_{t+1}$ \;\label{line:a-update-2} 
					$x_{t+1} = \frac{(1-\gamma_{t+1})A_{t}}{A_{t+1}} x_t
					+ \frac{\gamma_{t+1} \bA_{t+1}}{A_{t+1}} \tx_{t+1}$\;\label{line:x-update-2} 
					{\color{green!60!black}$\blambda_{t+2} = \alpha\blambda_{t+1}$}\label{line:lam-update-2}
				}

				$v_{t+1} = v_t - a_{t+1}\grad f(\tx_{t+1})$\label{line:v-update}\;
			}
		\end{algorithm}
	\end{minipage}\hfill
	\begin{minipage}[t]{0.47\linewidth}%
		\centering
		\setcounter{algocf}{-1}
		\begin{algorithm}[H]
			\DontPrintSemicolon
			\caption{MS Acceleration}\label{alg:ms}
			\KwInput{Initial $x_0$, function $f$, oracle $\oracle$}
			\KwParameters{Bisection limits $\lambda^\ell$, $\lambda^h$, and tolerance $\rho > 1$}
			Set $v_0 = x_0$, $A_0=0$\;
			\For{$t=0,1,\ldots,$}{
				\SetKwBlock{Bisection}{find $\lambda_{t+1}$ such that}{arg3}
					$\lambda^\ell_{t+1}, \lambda^h_{t+1} = \lambda^\ell, \lambda^h$\;
					$\blambda_{t+1} = \frac{\lambda^\ell_{t+1} +  \lambda^h_{t+1}}{2}$\label{line:bisection-start}
					\;
						$\ba_{t+1} = \frac{1}{2\blambda_{t+1}}\prn*{1 + \sqrt{1 + 4\blambda_{t+1} A_t}}$\;\label{line:update-standard-start}
						$\bA_{t+1} = A_t + \ba_{t+1}$\;
						$y_{t} = \frac{A_{t}}{\bA_{t+1}} x_t + \frac{\ba_{t+1}}{\bA_{t+1}} v_t$\;
						$\tx_{t+1}, \lambda_{t+1} = \oracle(y_t; \blambda_{t+1})$\label{line:update-standard-end}

					\If{$\lambda_{t+1} \in [\frac{1}{\rho}\blambda_{t+1}, \blambda_{t+1}]$}{
						$a_{t+1}=\ba_{t+1}$, ~$A_{t+1} = \bA_{t+1}$\;\label{line:update-MS-start}
						$x_{t+1} = \tx_{t+1}$\;\label{line:update-MS-end}
					}
					\ElseIf{$\lambda_{t+1} < \frac{1}{\rho}\blambda_{t+1}$}
					{
						$\lambda^h_{t+1} = \blambda_{t+1}$\;
						{\color{red!75!black}Go to \cref*{line:bisection-start}}\;\label{line:restart-1}
					}
					\Else{
						$\lambda^l_{t+1} = \blambda_{t+1}$\;
						{\color{red!75!black} Go to \cref*{line:bisection-start}}\;\label{line:restart-2}
					}
					$v_{t+1} = v_t - a_{t+1}\grad f(\tx_{t+1})$\;\label{line:bisection-end}
				}
			\end{algorithm}
		\end{minipage}
	\end{figure}

In this section we present our acceleration algorithm (\Cref{alg:optms}) which removes bisection from the MS method (shown in stylized form as \Cref{alg:ms}) and thereby attains optimal rates of convergence. For simplicity, in this section and the next we focus on unconstrained optimization ($\xset = \R^d$) and assume that $f$ is continuously differentiable, so that $\grad f$ exists. In \Cref{app:gen-framework} we extend our framework to general closed and convex domains and non-differentiable convex objectives. 

The key object in both the original MS algorithm and our new variant is an oracle $\oracle$ that approximately minimizes a local model of $f$ at a query point $y$. In particular, $\oracle$ satisfies the following approximation error bound, adapted from \citet[eq.~(3.3)]{MonteiroS13a} ($\lambda$ in~\cite{MonteiroS13a} is $1/\lambda$ in our notation).
\begin{definition}[MS oracle]\label{def:ms-condition}
	An oracle $\MSoracle: \R^d\times \R_+ \to \R^d \times \R_+$ is a \emph{$\sigma$-MS oracle} for function $f:\notarxiv{\R^d}\arxiv{\R^d}\to\R$ if for every $y\in\R^d$  and $\blambda>0$, the points $(x,\lambda)=\MSoracle(y;\blambda)$ satisfy
	\begin{equation}\label{eq:ms-condition}
		\norm*{ x - \prn*{y-\tfrac{1}{\lambda}\grad f(x)} } \le \sigma \norm{x-y}.
	\end{equation}
\end{definition}
\Cref{def:ms-condition} endows the oracle with an additional output $\lambda$ and an additional input $\blambda$. The value of $\lambda$ has the following simple interpretation: any point $x$ satisfying~\eqref{eq:ms-condition} approximately minimizes $F(x') = f(x') + \frac{\lambda}{2}\norm{x' -y}^2$ in the sense that $\norm{\grad F(x)} \le \lambda \sigma \norm{x-y}$. In particular, computing an exact proximal point $x_\lambda  = \argmin_{x'} F(x')$ and outputting $(x,\lambda)$ implements a $0$-MS oracle. 
The input $\blambda$ is optional: oracle implementations in prior work do not require it, but our new adaptive oracles (described in the next section) use it for improved efficiency. In \Cref{app:gen-framework} we provide a slightly more general approximation condition for MS oracles that handles non-smooth objectives and bounded domains, as well as a different, stochastic condition similar to that of~\cite{asi2021stochastic,carmon2022distributionally}.

Let us discuss the key differences between our algorithm (\Cref{alg:optms}) and the stylized MS algorithm (\Cref{alg:ms}).
At every iteration, \Cref{alg:ms} searches for a value $\blambda_{t+1}$ such that $\tx_{t+1}, \lambda_{t+1} = \oracle(y_t; \blambda_{t+1})$ satisfies $\lambda_{t+1} \approx \blambda_{t+1}$ (note that $y_t$ depends on $\blambda_{t+1}$). This is done via a bisection procedure 
iteratively shrinking an interval 
that contains a successful choice of $\blambda_{t+1}$.\footnote{
	\Cref{alg:ms} simplifies the bisection routine of \citet{MonteiroS13a} and implicitly assumes that an initial interval $[\lambda^{\ell}, \lambda^h]$ always contains a valid solution. One can guarantee such an interval exists by selecting very small $\lambda^{\ell}$ and very large $\lambda^{h}$. Alternatively, one may construct a valid initial interval via a bracketing procedure, as we do in the empirical comparison. 
	Either way, the cost  is logarithmic in problem parameters.
}  This bisection process is inefficient in the sense that every time we reach \cref{line:restart-1,line:restart-2} (highlighted in red) all of the optimization progress made by the last oracle call is discarded.

In contrast, even though our algorithm queries $\oracle$ in the same way (with $y_t$ computed based on a guess $\blambda_{t+1}$), it makes use of the oracle output even if $\lambda_{t+1}$ is very far from $\blambda_{t+1}$, thus never discarding progress made by the oracle. Instead of performing a bisection, we  compare $\blambda_{t+1}$ and $\lambda_{t+1}$ to guide our next guess $\blambda_{t+2}$. When $\blambda$ overshoots $\lambda$, we decrease it by a factor $\alpha$ (\cref{line:lam-update-1}, highlighted in green) and set $x_t$ and $A_t$ as in \Cref{alg:ms}. When it undershoots, we multiply it by $\alpha$ (\cref{line:lam-update-2}). In this case, we perform an additional key algorithmic modification which we call the \emph{momentum damping mechanism}: we scale down the growth of the parameter $A_{t+1}$ and replace the next iterate with a convex combination of $x_t$ and $\tx_{t+1}$ to ensure that our overly optimistic guess for $\lambda_{t+1}$ does not destabilize the algorithm.\footnote{It is also possible to set $x_{t+1} = \argmin_{x\in\{\tx_{t+1}, x_t\}} f(x)$ instead of the convex combination in \cref{line:x-update-2} and maintain our theoretical guarantees. } In \Cref{app:experiments-momentum-damping-mechanism} we demonstrate empirically that this mechanism is important for stabilizing \Cref{alg:optms}.  

Different MS oracles attain different rates of convergence when accelerated via the MS framework. In the following definition, we distill a key property that determines this rate.
\begin{definition}[Movement bound]\label{def:movement-bound}
	For $s \ge 1$, $c,\lambda>0$, and $x,y\in\R^d$ we say that  $(x,y,\lambda)$ satisfy a  \emph{$(s,c)$-movement bound} if 
	\begin{equation}\label{eq:movement-bound}
		\norm{x-y} \ge \begin{cases}
			(\lambda / c^{s})^{1/(s-1)} & s<\infty\\
			1/c & s = \infty\,,
		\end{cases}
	\end{equation}
	where a $(1,c)$-movement bound simply means that $\lambda \le c$. 
	\end{definition}
In the next section, we will show how to build MS oracles that, given query $y$, output $(x,\lambda)$ such that $(x,y,\lambda)$ always satisfy a $(s,c)$-movement bound, for certain $s$ and $c$ depending on the oracle type and function structure (e.g., level of smoothness). For example, when $f$ has $H$-Lipschitz Hessian, the cubic-regularized Newton step with $M=2H$ is a $\half$-MS oracle that guarantees a $(2,\sqrt{H})$-movement bound. 
 With the necessary definitions in hand, we are ready to state our main result: the iteration (and MS oracle query) complexity of \Cref{alg:optms}. %

\begin{theorem}\label{thm:main}
	Let $f:\R^d\to \R$ be convex and differentiable, and consider \Cref{alg:optms} with parameters $\alpha>1$, $\blambda>0$, and a $\sigma$-MS oracle (\Cref{def:ms-condition}) for $f$ with  $\sigma\in [0,0.99)$. Let $s\ge 1$ and $c>0$, and suppose that for all $t$ such that $\lambda_{t}>\blambda_{t}$ or $t=1$, the iterates $(\tx_{t}, y_{t-1}, \lambda_{t})$ satisfy a $(s,c)$-movement bound (\Cref{def:movement-bound}). There exist $C_{\alpha,s} = O\prn*{
	\frac{s}{\min\crl*{s,\ln \alpha}} \alpha^{\frac{s+1}{3s+1}} }$ and $K_{\alpha} = O\prn*{\frac{1}{\ln \alpha}\alpha^{1/3} }$ such that\footnote{For a fixed $s\ge 1$, the value of $\alpha$ minimizing our complexity bound is $\alpha^\star = e^{\frac{3s+1}{s+1}}$. In practice, performance is not sensitive to the choice of $\alpha$ (see \Cref{app:experiments-parameter-sensitivity}).}
	for any $\xopt \in \R^d$ and $\epsilon>0$, we have $f(x_T) - f(\xopt) \le \epsilon$ when
	\begin{equation*}
		T \ge \begin{cases}
			C_{\alpha,s} \cdot \prn*{\frac{c^{s} \norm{x_0 - \xopt}^{s+1}}{\epsilon}}^{\frac{2}{3s+1}}  & s< \infty \\[8pt]
			K_\alpha \cdot \prn*{c \norm{x_0 - \xopt}}^{\frac{2}{3}}\log \frac{\lambda_1\norm{x_0 - \xopt}^2}{\epsilon} & s=\infty. 
		\end{cases}
	\end{equation*}
\end{theorem}

\paragraph{Proof sketch.}
The remainder of this section is an overview of the proof of \Cref{thm:main}, which we provide in full in~\Cref{apdx:framework}. To simplify this proof sketch, we treat $\alpha$, $c$, and $1/(1-\sigma)$ as $O(1)$, and focus on $s<\infty$. To highlight the novel aspects of the proof, let us first briefly recall the analysis of \Cref{alg:ms}~\cite{MonteiroS13a,gasnikov2019optimal, bubeck2019optimal, jiang2019optimal,carmon2020acceleration}. For every $t \le T$ let 
\begin{equation*}
	E_t \defeq f(x_t) - f(\xopt)~~,~~D_t \defeq \half\norm{v_t - \xopt}^2
	~~\mbox{and}~~M_t = \half \norm{\tx_{t} - y_{t-1}}^2.
\end{equation*}
The key facts about the standard MS iterations are
\begin{equation}\label{eq:standard-ms-facts}
	E_T \le \frac{D_0}{A_T}~~,~~\sum_{t\in[T]} \lambda_{t} A_{t} M_{t} \le O(D_0)
	~~\mbox{and}~~
	\sqrt{A_T} \ge \Omega(1)\sum_{t \in [T]} \frac{1}{\sqrt{\blambda_t}}.
\end{equation}
The first fact implies that the optimality gap at iteration $T$ is inversely proportional to $A_T$, while the latter two facts imply that $A_T$ grows rapidly. 
More specifically, substituting the movement bound $M_t \ge \Omega\prn*{(\lambda_t)^{2/(s-1)}}$ and $\blambda_t \ge \Omega(\lambda_t)$ (thanks to the bisection) yields $\sum_{t\in[T]} {\blambda_{t}}^{\frac{s+1}{s-1}} A_{t} = O(D_0)$. 
Combining this with the third fact in~\eqref{eq:standard-ms-facts} and using the reverse \Holder inequality allows one to conclude that, for $k=\frac{s+1}{3s+1}$ and $k' = \frac{s-1}{3s+1}$, we have
$	A_T^{k} \ge \Omega(D_0^{-k'}) \sum_{t\in [T]} A_i^{k'}$,
which, upon further algebraic manipulation, yields $A_T \ge \Omega(T^{(3s+1)/2}D_0^{-(s-1)/2})$. Plugging this back to to the first fact in~\eqref{eq:standard-ms-facts} gives the claimed convergence rate.

Having described the standard MS analysis, we move on to our algorithm.
Our \emph{first challenge} is re-establishing the facts~\eqref{eq:standard-ms-facts}. The difficult case is $\lambda_t > \blambda_t$, where the standard cancellation that occurs in the MS analysis may fail. This is where the momentum damping mechanism (\cref{line:a-update-2,line:x-update-2} of our algorithm) comes into play, allowing us to show that (See \Cref{prop:main-potential} in the appendix)
\begin{equation}\label{eq:opt-ms-facts}
	E_T \le \frac{D_0}{A_T}~~,~~\sum_{t\in \SsupT\cup\{1\}} \blambda_{t} A_{t} M_{t} \le O(D_0)
	~~\mbox{and}~~
	\sqrt{A_T} \ge \Omega(1)\sum_{t \in \SdownT} \frac{1}{\sqrt{\blambda_t}},
\end{equation}
where  $\SupT\defeq\{t\in[T] \mid \lambda_{t}\geq\blambda_{t}\}$ and 
 $\SsupT$, $\SdownT$ , $\SsdownT$ and $\SequalT$ are analogously defined.

Comparing~\eqref{eq:standard-ms-facts} and~\eqref{eq:opt-ms-facts}, the price of removing the bisection becomes evident: at each iteration (except the first) only one of the terms forcing the growth of $A_t$ receives a contribution. The \emph{second challenge} of our proof is establishing a lower bound on $\sqrt{A_T}$ in terms of the $1/\sqrt{\blambda_t}$ values for $t\in\SsupT\cup\{1\}$, where the movement bound holds for $M_t$. This is where the multiplicative $\blambda$ update rule (\cref{line:lam-update-1,line:lam-update-2} of the algorithm) comes into play: it allows us to ``credit'' the contribution of every ``down iterate'' (in $\SdownT$) to an adjacent ``up'' iterate ($\SsupT \cup \{1\}$) and furthermore argue that the contribution gets an exponential bonus based on the distance  between the two. Consequently, we are able to identify a set $\QT \subseteq \SsupT \cup \{1\}$ of iterates, and a sequence $\{r_t\}$ such that (see \Cref{lem:increase})
$	\sqrt{A_{T}} \geq \Omega(1)\sum_{t\in \QT}\sqrt{\frac{\alpha^{r_{t}-1}}{\blambda_{t}}}~~\mbox{and}~~
	\sum_{t\in[T]}r_t = \frac{T-1}{2}.
$

Repeating the reverse \Holder argument of prior work, we obtain the recursive bound
\begin{equation}
\label{eq:sketch-holder-us}
	A_T^{k} \ge \Omega(D_0^{-k'}) \sum_{t\in \QT} A_t^{k'} \alpha^{k r_t} \ge 
	\Omega(D_0^{-k'}) \sum_{t\in \QT} A_t^{k'} r_t
\end{equation}
with  $k=\frac{s+1}{3s+1}$ and $k' = \frac{s-1}{3s+1}$  as before. 
The \emph{final challenge} of our proof is to show that such recursion implies sufficient growth of $A_t$. This is where careful algebra comes into play; we show that \eqref{eq:sketch-holder-us} implies that $A_T \ge \Omega\prn*{( \sum_{t\in[T]} r_t )^{(3s+1)/2} D_0^{-(s-1)/2}}$ (see \Cref{lem:growth-rate-high,lem:growth-rate-ball}) which establishes our result since $\sum_{t\in[T]}r_t = \frac{T-1}{2}$. 

\section{MS oracle implementations}\label{sec:solvers}

In this section we describe several oracles that satisfy both \Cref{def:ms-condition} (the MS condition) and \Cref{def:movement-bound} (movement bounds) and may therefore be used by \Cref{alg:optms}. \Cref{ssec:solvers-prior} briefly reviews oracles that have appeared in prior work, while \Cref{ssec:solvers-amsn} and \Cref{ssec:solvers-amsn-fo} describe our new adaptive oracle implementations. We summarize the key oracle properties and resulting complexity bounds in \Cref{table:summary}.

\subsection{Oracles from prior work}\label{ssec:solvers-prior}

Here we consider several previously-studied oracles of the form $(x,\lambda) = \oracle(y)$, where we omit the second argument $\blambda$ since prior work does not leverage it to improve implementation efficiency.   

\paragraph{Gradient descent step \cite[e.g.,][]{nesterov2018lectures}.} As a gentle start, consider the oracle $\oracle[gd](y) = (y - \eta \grad f(y), \frac{1}{\eta})$, i.e., an oracle that returns $x$ by taking standard gradient step with size $\eta$ and $\lambda=1/\eta$. Obviously, the oracle always satisfies a $(1,\eta^{-1})$-movement bound. Moreover, if we assume that $\grad f$ is $L$-Lipschitz, then $\norm{x-(y-\frac{1}{\lambda}\grad f(x))} = \eta \norm{\grad f(x)-\grad f(y)} \le \eta L \norm{x-y}$. Therefore, when $\eta^{-1} \ge L/\sigma$ the oracle is a $\sigma$-MS oracle. 

\begin{table}[t]
   \notarxiv{\hspace{-1.5em}}
	\centering
	\renewcommand{\arraystretch}{1.75}
	\begin{tabular}{@{}p{3.65cm}p{1.3cm}ll@{}}
		\toprule
		Assumption                                     & Oracle           & Complexity with \Cref{alg:optms}                                                       & Lower bound                                            \\ \midrule
		$\grad^p f$ is $(1,\nu)$-\Holder$^*$               & $\oracleP$       & $O\prn*{\epsilon^{-\frac{2}{3(p+\nu)-2}}}$ evals of $\grad^p f$                        & $\Omega\prn*{\epsilon^{-\frac{2}{3(p+\nu)-2}}}$ \cite{arjevani2019oracle,grapiglia2020tensor}  \\
		$\grad^3 f$ is $1$-Ljpschitz               & $\oracle[$3$-reg-so]$       & $O\prn*{\epsilon^{-\frac{1}{5}}}$ Hessian evals                        & $\Omega\prn*{\epsilon^{-\frac{1}{5}}}$ \cite{arjevani2019oracle,grapiglia2020tensor}  \\
		N/A                                               & $\oracleBall$    & $O\prn*{r^{-\frac{2}{3}} \log \frac{1}{\epsilon} }$ oracle calls                               & $\Omega\prn*{r^{-\frac{2}{3}} }$ \cite{carmon2020acceleration} \\
		Stable Hessian                        & $\oracleBaCoN$ & $O\prn*{r^{-\frac{2}{3}} \log \frac{1}{\epsilon} }$ Hessian evals                              & \multicolumn{1}{c}{-}                                                        \\ \midrule
		\multirow[t]{2}{=}{$\hess f$ is $(1,\nu)$-\Holder$^\dagger$}                 & \multirow[t]{2}{=}{$\adaOracle$ (Alg.~\ref{alg:adaptive-msn-step})}    & $O\prn*{\epsilon^{-\frac{2}{4+3\nu}}}$ Hessian evals                                   & $\Omega\prn*{\epsilon^{-\frac{2}{4+3\nu}}}$ \cite{arjevani2019oracle,grapiglia2020tensor}      \\
		&                  & $O\prn*{\epsilon^{-\frac{2}{4+3\nu}}} + \Otil{1}$ linear systems                 & \multicolumn{1}{c}{-}                                                        \\
		\multirow[t]{2}{=}{$\grad f$ is $\ell$-Lipschitz and $\hess f$ is $(1,\nu)$-\Holder$^\dagger$} & \multirow[t]{2}{=}{$\foadaOracle$ (Alg.~\ref{alg:adaptive-msn-cg-step})}   & $O\prn*{(\frac{\epsilon}{\ell})^{-\half}} + \Otil{1}$ first-order evals                        & $\Omega\prn*{(\frac{\epsilon}{\ell})^{-\half}}$ \cite{nesterov2018lectures}          \\
		&                  & $O\prn*{\min\crl*{(\frac{\epsilon}{\ell})^{-\half}, \epsilon^{-\frac{2}{4+3\nu}}}}$ iterations &      \multicolumn{1}{c}{-}                                                           \\ \bottomrule
	\end{tabular}
	\caption{\label{table:summary}Complexity bounds for finding $x$ such that $f(x)-f(\xopt)\le\epsilon$ assuming $\norm{x-\xopt}\le 1$, attained by MS oracles from the literature (top 4 rows, described in \Cref{ssec:solvers-prior}) and oracles we develop (bottom two rows). In all cases we improve on prior work by a logarithmic factor. $^*$We require $p+\nu\ge2$. $^\dagger$Our adaptive oracles do not require knowledge of continuity constants or even the \Holder order $\nu\in[0,1]$.}
\end{table}

\paragraph{Taylor descent step~\cite{baes2009estimate,nesterov2021implementable,gasnikov2019optimal, bubeck2019optimal, jiang2019optimal, song2021unified}.}
Generalizing both $\oracle[gd]$ and the cubic-regularized Newton step oracle $\oracle[cr]$, we define for every integer $p\ge1$ and $\nu\in[0,1]$ the oracle $\oracleP$, that, for parameter $C$ and input $y$ returns $(x,\lambda) = \oracleP(y)$ where
\begin{equation}\label{eq:oracle-p-reg}
	x = \argmin_{x'\in\R^d} \crl*{ \tilde{f}_p(x';y) + \frac{M}{p!(p+\nu)}\norm{x'-y}^{p+\nu}}
	~~\mbox{,}~~\lambda=\frac{M}{p!}\norm{x-y}^{p+\nu-2}
\end{equation}
and $\tilde{f}_p(x;y) \defeq \sum_{i=0}^p \frac{1}{i!}\grad^i f(y)[(x-y)^{\otimes i}]$ is the Taylor expansion of $f$ around $y$ evaluated at $x$. 
Oracles $\oracle[gd]$ and $\oracle[cr]$ correspond to the special cases $\oracleP[1,1]$ (with $\eta = M^{-1}$) and $\oracleP[2,1]$, respectively. 
Clearly, by definition, the oracle always satisfies a $(p+\nu-1, (M/p!)^{1/(p+\nu-1)})$-movement bound. Moreover, it is easy to show that 
\begin{equation*}
	\norm*{x-\prn*{y-\frac{1}{\lambda}\grad f(x)}} = \frac{1}{\lambda}\norm{ \grad f(x) - \grad \tilde{f}_p(x;y) } = \frac{p!}{M} \frac{\norm{\grad f(x) - \grad \tilde{f}_p(x;y)}}{\norm{x-y}^{p+\nu-2}}.
\end{equation*}

For any $p\ge1$ and $\nu\in[0,1]$ we say that
\begin{equation*}
	\mbox{$\grad^p f$ is $(H,\nu)$-\Holder if for all $x,y$ we have $\opnorm{\grad^p f(x) - \grad^p f(y)} \le H\norm{x-y}^\nu$}.
\end{equation*}
(An $(H,1)$-\Holder derivative is $H$-Lipschitz.)
If $\grad^p f$ is $(H,\nu)$-\Holder, Taylor's theorem gives $\norm{\grad f(x) - \grad \tilde{f}_p(x;y)} \le \frac{H}{p!}\norm{x-y}^{p+\nu-1}$~\cite[Lemma 2.5]{song2021unified}, and so $\norm{x-(y-\frac{1}{\lambda}\grad f(x))} \le \frac{H}{M} \norm{x-y}$. Therefore, when $M\ge H / \sigma$ the oracle is a $\sigma$-MS oracle.  

\paragraph{Exploiting third-order smoothness with a second order oracle~\cite{nesterov2021superfast,kamzolov2020near}.}
For $p > 2$, computing $\oracleP$ is typically intractable due to the need to compute the high-order derivative tensors $\grad^3 f(y), \grad^4 f(y), \ldots, \grad^p f(y)$. Nevertheless for $p=3$ \citet{nesterov2021superfast} designs an approximate solver for~\eqref{eq:oracle-p-reg}, which we denote $\oracle[$3$-reg-so]$, using only $\hess f(y)$ and a logarithmic number of gradient evaluations. When $\grad^3 f$ is $(L_3, 1)$-\Holder, \cite{nesterov2021superfast} shows that $\oracle[$3$-reg-so]$ is a valid MS-oracle satisfying a $(3, O(L_3))$-movement bound, on par with the movement bound of $\oracleP[3,1]$. 

\paragraph{Exact ball optimization oracle \cite{carmon2020acceleration}.} For a given query $y$, consider the exact minimizer of $f$ constrained to a ball of radius $r$ around $y$, i.e., consider an oracle $\oracleBall$ such that $(x,\lambda) = \oracleBall(y)$ satisfy
$	x \in \argmin_{x':\norm{x'-y}\le r} f(x')$ and $\lambda = \frac{\norm{\grad f(x)}}{\norm{x-y}}$.
One may easily verify that (unless $\lambda = \norm{\grad f(x)}=0$) we have $x = y - \frac{1}{\lambda}\grad f(x)$, and therefore the oracle is a $0$-MS oracle. Moreover, when $f$ is convex, we have either $\norm{x-y}=r$ or $x$ is a global minimizer of $f$, and so we may assume without loss of generality that the oracle satisfies an $(\infty, 1/r)$ movement bound.

\paragraph{Ball-Constrained Newton (BaCoN) oracle \cite{carmon2020acceleration}.} Exactly implementing $\oracleBall$ is generally intractable. Nevertheless, \citet[Alg.~3]{carmon2020acceleration} describe a method $\oracleBaCoN$ based on solving a sequence of $\Otil{1}$ trust-region problems (ball-constrained Newton steps), which we call that, for functions that are $O(1)$-Hessian stable in a ball of radius $r$ (or $1/r$-quasi-self-concordant) and have a finite condition number, outputs $(x,\lambda)$ satisfying the $\half$-MS oracle condition and an $(\infty, O(1/r))$-movement bound. Implementing  $\oracleBaCoN$ requires only a single Hessian evaluation and a number of linear system solutions that is polylogarithmic in problem parameters.
Subsequent works implementing ball oracles~\cite{carmon2021thinking,asi2021stochastic,carmon2022distributionally} satisfy an approximation guarantee different than the MS condition, similar to the one we describe in \Cref{app:gen-framework}.

\subsection{An adaptive Monteiro-Svaiter-Newton oracle}\label{ssec:solvers-amsn}

The oracle implementations in \Cref{ssec:solvers-prior} satisfy movement bounds by design and the MS condition~\eqref{eq:ms-condition} by assumption. For example, the cubic-regularized Newton step oracle $\oracle[cr]$ is guaranteed to satisfy 
the MS condition only when the regularization parameter $M$ is sufficiently larger than the Lipschitz constant of $\hess f$. This suggests that $M$ must be carefully tuned to ensure good performance. Prior work attempt to dynamically adjust $M$ in order to meet certain approximation conditions~\cite{cartis2011adaptive,gould2012updating,grapiglia2020tensor,jiang2020unified}.  However, even computing a single cubic-regularized Newton step entails searching for $\lambda$ that satisfies $\norm{ [\hess f(y) + \lambda I]^{-1} \grad f(y) } = \frac{M\lambda}{2}$. Therefore, such a search over $M$ is essentially a (potentially) redundant double search over $\lambda$. 

We propose a more direct and more adaptive MS oracle recipe: \emph{search for the smallest $\lambda$ for which the regularized Newton step $x=y-[\hess f(y) + \lambda I]^{-1} \grad f(y)$ satisfies the MS condition~\eqref{eq:ms-condition}}.\footnote{The prior works  \cite{mischenko2021regularized,doikov2021gradient} also directly consider  quadratically-regularized Newton steps, but employ approximation conditions other than~\eqref{eq:ms-condition} to select the parameter $\lambda$.} This yields valid MS oracle by construction, independently of any assumption. Moreover, it is simple to argue
that when  $\hess f$ is $(H,\nu)$-\Holder continuous for some $\nu\in[0,1]$, such oracle would guarantee the same movement bound as $\oracleP[2,\nu]$ with the best choice parameters $M$ and $\nu$ (see \Cref{apdx:solvers-ideal})---even though our recipe requires neither of these parameters!

\begin{figure}[h!]
	\begin{minipage}[t]{0.50\linewidth}%
		\setcounter{algocf}{1}
		\begin{algorithm}[H]
			\DontPrintSemicolon
			\caption{$\adaOracle$}\label{alg:adaptive-msn-step}
			\KwInput{Query $y\in\R^d$, $\blambda >0$. Flag $\lazyflag$.}
			\KwParameters{MS factor $\sigma \in (0,1)$.}
			
			\If{$\mscheck(\blambda;y,\sigma)$\label{line:amsn-first-check}}{
				\lIf{$\lazyflag$}{\Return{$y - [\hess f(y) + \blambda I]^{-1} \grad f(y)$, $\blambda$}}\Else{
				$\lamvalid \gets \blambda$~~,~~
				$ k \gets 0 $\;
				\While{$\mscheck(\lamvalid / 2^{2^{k}};y,\sigma)$\label{line:amsn-decrease-loop}}{
					$\lamvalid \gets \lamvalid / 2^{2^{k}}$\;
					$k\gets k +1$\;
				}
				$k^\star \gets k$~~,~~$\laminvalid \gets \lamvalid / 2^{2^{k^\star}}$
			}}
			\Else{
				$\laminvalid \gets \blambda$~~,~~$ k \gets 0 $\;
				\While{\textbf{\textup{not}} $\mscheck(\laminvalid 2^{2^{k}};y,\sigma)$\label{line:amsn-increase-loop}}{
					$\laminvalid \gets \laminvalid  2^{2^{k}}$\;
					$k\gets k +1$\;
				}
				$k^\star \gets k$~~,~~$\lamvalid \gets \laminvalid  2^{2^{k^\star}}$
			}

			\While(\label{line:amsn-bisection-loop}){$\laminvalid < \lamvalid / 2$}{
				$\lambda \gets \sqrt{\laminvalid \lamvalid}$\;
				\lIf{$\mscheck(\lambda;y,\sigma)$}{
					$\lamvalid \gets \lambda$
				}
				\lElse{
					$\laminvalid \gets \lambda$
				}
			}
			
			\Return{$y - [\hess f(y) + \lamvalid I]^{-1} \grad f(y)$, $\lamvalid$}
			
			\vspace{\baselineskip}
			
			\Fn{$\mscheck(\lambda;y,\sigma)$}{
				$x = y - [\hess f(y) + \lambda I]^{-1} \grad f(y)$\;
				\lIf{$\norm*{x - (y-\frac{1}{\lambda}\grad f(x))} \le \sigma \norm{x-y}$}{\Return{\textup{True}}}\lElse{\Return{\textup{False}}}
			}
		\end{algorithm}
	\end{minipage}\hfill
	\begin{minipage}[t]{0.48\linewidth}%
		\begin{algorithm}[H]
			\DontPrintSemicolon
			\caption{$\foadaOracle$}\label{alg:adaptive-msn-cg-step}
			\KwInput{$y\in\R^d$, $\blambda >0$. Flag $\lazyflag$.}
			\KwParameters{MS factor $\sigma \in (0,1)$.}
			
			$\lambda \gets \blambda$~~,~~$\failedbefore\gets$ False\;
			\Repeat{
				$A \gets \hess f(y) + \lambda I$~,~$b \gets - \grad f(y)$\;
				\Comment{Apply MinRes/Conjugate Residuals~\cite{fong2012cg} %
					until obtaining $w$ s.t.\ $\norm{A w-b} \le \frac{\lambda\sigma}{2}\norm{w}$}
			
				$x\gets y + \minresnewton(A,b,\lambda\sigma)$\;

				\If{
					$\norm*{x - (y-\frac{1}{\lambda}\grad f(x))} \le \sigma \norm{x-y}$
					\label{line:newton-cr-ms-check}
				}{
					\If{$\lazyflag$ \textup{or} $\failedbefore$}{
						\Return{$x$, $\lambda$}}
					\lElse{$\lambda \gets \lambda /2$%
					}%
				}\Else{
					$\failedbefore\gets $ True\;
				    $\lambda \gets 2\lambda$~~~%
				    \;
				}
			}
		
			\vspace{\baselineskip}
		
		\Fn{$\minresnewton(A,b,\lambda\sigma)$}{
			$w_0 \gets 0$\;
			$p_0 \gets r_0\gets Aw_0 -b$\Comment*{$r_i = Aw_i - b$}
			$s_0\gets q_0 \gets Ar_0$\Comment*{$q_i = Ap_i$}
			$i\gets 0$\;
			\While{$\norm{r_i} > \frac{\lambda\sigma}{2}\norm{w_i}$\label{line:amsnfo-while}}
			{
				$w_{i+1} \gets w_{i} - \frac{\inner{r_i}{s_i}}{\norm{q_i}^2} p_i$\;
				$r_{i+1} \gets r_{i} - \frac{\inner{r_i}{s_i}}{\norm{q_i}^2} q_i$\;
				$s_{i+1} \gets A r_{i+1}$\;
				$p_{i+1} \gets \frac{\inner{r_{i+1}}{s_{i+1}}}{\inner{r_i}{s_i}} p_i + r_{i+1}$\;
				$q_{i+1} \gets \frac{\inner{r_{i+1}}{s_{i+1}}}{\inner{r_i}{s_i}} q_i + s_{i+1}$\;
				$i\gets i+1$\;
			}

			\Return{$w_i$}
		}
		\end{algorithm}
	\end{minipage}
\vspace{-16pt}
\end{figure} 
Exactly fulfilling this recipe, i.e., finding the ideal minimal $\lambda^\star$ that satisfies the MS condition, is difficult. Fortunately, to adaptively guarantee movement bounds, it suffices to find a value $\lambda$ such the corresponding regularized Newton step satisfies the MS condition, while the step corresponding to $\lambda/2$ does not; 
\Cref{alg:adaptive-msn-step} finds  precisely such a $\lambda$. 

Let us describe the operation of \Cref{alg:adaptive-msn-step}. If the input $\blambda$ is invalid (i.e., its corresponding regularized Newton step does not satisfy the MS condition so that $\mscheck(\blambda;y,\sigma)$ evaluates to False), we set $\laminvalid\gets \blambda$ and test a double-exponentially increasing series of $\lambda$'s, until reaching a valid $\lamvalid$ (\cref{line:amsn-increase-loop}). If $\blambda$ is valid and the $\lazyflag$ flag is set, we return it immediately. Otherwise (if $\lazyflag$ is not set)
we set $\lamvalid=\blambda$ and decrease it at a double-exponential rate until finding an invalid $\laminvalid$ (\cref{line:amsn-decrease-loop}). In either case (so long as $\lazyflag$ is not set) we obtain an (invalid,valid) pair $(\laminvalid, \lamvalid)$ such that $\lamvalid / \laminvalid = 2^{2^{k^\star}}$ at the cost of $2+k^\star$ linear system solutions. We then perform precisely $k^\star$ log-scale bisection steps in order to shrink $\lamvalid / \laminvalid$ down to 2 while maintaining the invariant that $\lamvalid$ is valid and $\laminvalid$ is invalid   (\cref{line:amsn-bisection-loop}). 

The following theorem bounds the complexity  of \Cref{alg:adaptive-msn-step} in terms of linear-system solution number, and establishes a movement bound for its output assuming that $\hess f$ is locally \Holder around the query point. We defer the proof of the theorem and its following corollary to \Cref{apdx:solvers-amsn}.

\begin{restatable}{theorem}{restateThmAmsn}\label{thm:amsn}
	\Cref{alg:adaptive-msn-step} with parameter $\sigma$ is a $\sigma$-MS oracle $\adaOracle$. For any $y\in\R^d$, computing $(x,\lambda)=\adaOracle(y)$ requires at most $2 +2 \log_2\prn*{1+\abs*{\log_2 \frac{\lambda}{\blambda}}}$ linear systems solutions. If $\lazyflag$ is False
	or $\lambda > \blambda$, and if $\hess f$ is $(H,\nu)$-\Holder in a ball of radius $2\norm{x-y}$ around $y$, then $(x,y,\lambda)$ satisfy a
	 $\prn*{1+\nu, (2H/\sigma)^{1/(1+\nu)}/\sigma}$-movement bound.
\end{restatable}

To understand the $\lazyflag$ option of \Cref{alg:adaptive-msn-step}, note that when $\blambda$ is valid we will necessarily output $\lambda \le \blambda$. In such case \Cref{thm:main} does not require a movement bound (except for the first iteration). Therefore, we might as well save on computation and return $\blambda$. The following \Cref{coro:adaptive-msn-step} gives the overall complexity bound for the combination of \Cref{alg:optms} and $\adaOracle$, leveraging ``lazy'' oracle calls to show that the number of linear system solves per iteration is essentially constant. %

\begin{restatable}{corollary}{restateCorrAmsn}\label{coro:adaptive-msn-step}
	Consider \Cref{alg:optms} with initial point $x_0$, parameters $\alpha$ satisfying $1.1\le \alpha = O(1)$  and $\blambdainit$, and $\sigma$-MS oracle $\adaOracle$ (with $\lazyflag=$ True in all but the first iteration) with $\sigma \in (0.01,0.99)$. For any $H,\epsilon>0$, $\nu\in[0,1]$ and any $\xopt\in\R^d$, if $f$ is convex with $(H,\nu)$-\Holder Hessian,
	the algorithm produces an iterate $x_T$ such that $f(x_T) \le f(\xopt) + \epsilon$ using 
		$T = O\prn*{ \prn*{{H \norm{x_0-\xopt}^{2+\nu}}/{\epsilon}}^{2/(4+3\nu)}}$
	Hessian evaluations and $O\prn*{T + \log\log\max\crl*{\frac{H R^\nu}{\blambdainit},
		\frac{\blambdainit R^2}{\epsilon} }}$
	linear system solutions, where $R$ is the distance between $x_0$ and $\argmin_{x'}f(x')$.
\end{restatable}

Note that as long as $\blambdainit$ is in the range $\brk[\big]{2^{-2^T}HR^\nu , 2^{2^T}\epsilon R^{-2} }$, the double logarithmic term in our bound on linear system solution number is $O(T)$. Therefore, the overall bound is $O(T)$ for an extremely large range of $\blambdainit$ values.

\subsection{First-order implementation via MinRes/Conjugate Residuals}\label{ssec:solvers-amsn-fo}

We now present a first-order implementation of our adaptive oracle,  $\foadaOracle$ (\Cref{alg:adaptive-msn-cg-step}), which replaces exact linear system solutions with approximations obtained via Hessian-vector products and the MinRes/Conjugate Residuals method~\cite{stiefel1955relaxationsmethoden,fong2012cg}. Similar to \Cref{alg:adaptive-msn-step}, the algorithm searches for $\lambda$ such that $x_\lambda \approx y -[\hess f(y) + \lambda I]^{-1} \grad f (y)$ satisfies the MS condition, but $x_{\lambda/2}$ does not. Departing from the double-exponential scheme of \Cref{alg:adaptive-msn-step}, here we adopt the following doubling scheme that allows us to control the cost of the $x_\lambda$ approximation. If $\blambda$ is such that $x_{\blambda}$ does not satisfy the MS condition, we repeatedly test $\lambda=2\blambda, 4\blambda, 8\blambda, \ldots$ and return the first one for which $x_\lambda$ satisfies the MS condition. If $x_{\blambda}$ satisfies the MS condition and the algorithm is $\lazyflag$, we immediately return it. Otherwise, we repeatedly test $\lambda = \frac{1}{2}\blambda, \frac{1}{4}\blambda, \frac{1}{8}\blambda, \ldots$ until reaching $\lambda$ for which $x_\lambda$ does not satisfy the MS condition, and return $x_{2\lambda}$. 

The subroutine $\minresnewton$ of \Cref{alg:adaptive-msn-cg-step} takes as input a matrix $A$, a vector $b$, and accuracy parameter $\lambda\sigma$, and iteratively generates $\{w_i\}$ that approximate $A^{-1}b$. The construction of the MinRes/Conjugate Residuals method guarnatees that $w_i$ minimizes the norm of the residual $r_i=Aw_i - b$ in the Krylov subspace $\mathrm{span}\{b, Ab, \ldots, A^{i-1}b\}$. The key algorithmic decision here is when to stop the iterations: stop too early, and the approximation for the Newton step might not be accurate enough to guarantee a movement bound; stop too late, and incur a high Hessian-vector product complexity. We introduce a simple stopping condition (\cref{line:amsnfo-while}) that strikes a balance. On the one hand, we show that whenever the condition $\norm{r_i} \le \frac{\lambda\sigma}{2}\norm{w_i}$ holds, the resulting point $x$ can certify roughly the same movement bounds as exact Newton steps. On the other hand, by invoking the complexity bounds in~\cite{lee2021geometric} and using the the optimality of  $\norm{r_i}$, we guarantee that the stopping condition is met within a number of iterations proportional to $1/\sqrt{\lambda}$. The structure of our doubling scheme for $\lambda$ then allows us to relate the overall first-order complexity to the lowest value of $\lambda$ queried, obtaining the following guarantees. See proofs in \Cref{apdx:solvers-amsn-cg}. 

\begin{restatable}{theorem}{restateThmAmsnFO}\label{thm:amsn-fo}
	\Cref{alg:adaptive-msn-cg-step} with parameter $\sigma$ is a $\sigma$-MS oracle $\foadaOracle$. For any $y\in\R^d$, computing $(x,\lambda)=\foadaOracle(y)$ requires at most $O\prn*{\sqrt{1+\frac{\opnorm{\hess f(y)}}{\sigma \min\{\lambda,\blambda\}}}}$ Hessian-vector product and $O(\abs*{\log \frac{\lambda}{\blambda}})$ gradient computations.
	 If $\lazyflag$ is False or $\lambda > \blambda$, and if $\hess f$ is $(H,\nu)$-\Holder, then $(x,y,\lambda)$ satisfy a
	$\prn*{1+\nu, (6H/\sigma)^{1/(1+\nu)}}$-movement bound.
\end{restatable}

\begin{restatable}{corollary}{restateCorrAmsnFO}
	Consider \Cref{alg:optms} with initial point $x_0$, parameters $\alpha$ satisfying $1.1\le \alpha = O(1)$ and $\blambdainit$, and $\sigma$-MS oracle $\foadaOracle$ with $\lazyflag$ set to True in all but the first iteration and $\sigma \in (0.01,0.99)$.
	For any $L,H,\epsilon>0$, $\nu\in[0,1]$ and any $\xopt\in\R^d$, if $f$ is convex with $(H,\nu)$-\Holder Hessian and $L$-Lipschitz gradient,
	the algorithm produces an iterate $x_T$ such that $f(x_T) \le f(\xopt) + \epsilon$ within 
		$T = O\prn*{\prn*{\frac{H \norm{x_0-\xopt}^{2+\nu}}{\epsilon}}^{2/(4+3\nu)}}$
	iterations and at most  $O\prn*{\prn*{\frac{L\norm{x_0-\xopt}^2}{\epsilon}}^{1/2}  + \sqrt{\frac{L}{\blambdainit}}+ \log\frac{\blambdainit}{L}}$
	gradient and Hessian-vector product evaluations.
\end{restatable}

Note that the $L$-Lipschitz gradient assumption implies an $(L,0)$-\Holder Hessian assumption, giving the iteration complexity bound we state in \Cref{table:summary}. Moreover, note that our algorithm has the optimal $O( \sqrt{L\norm{x_0-\xopt}^2/\epsilon} )$ complexity for any $\blambdainit$ in the range $\Omega(\epsilon / \norm{x_0-\xopt}^2 )$ to $L \exp\crl*{O( \sqrt{L\norm{x_0-\xopt}^2/\epsilon} )}$. By choosing a large $\blambdainit$ (say $10^6$) we may guarantee that only the logarithmic term is added to the optimal first-order evaluation complexity.

\section{Experiments}\label{sec:experiments}

\begin{figure}
	\includegraphics[width=1\textwidth]{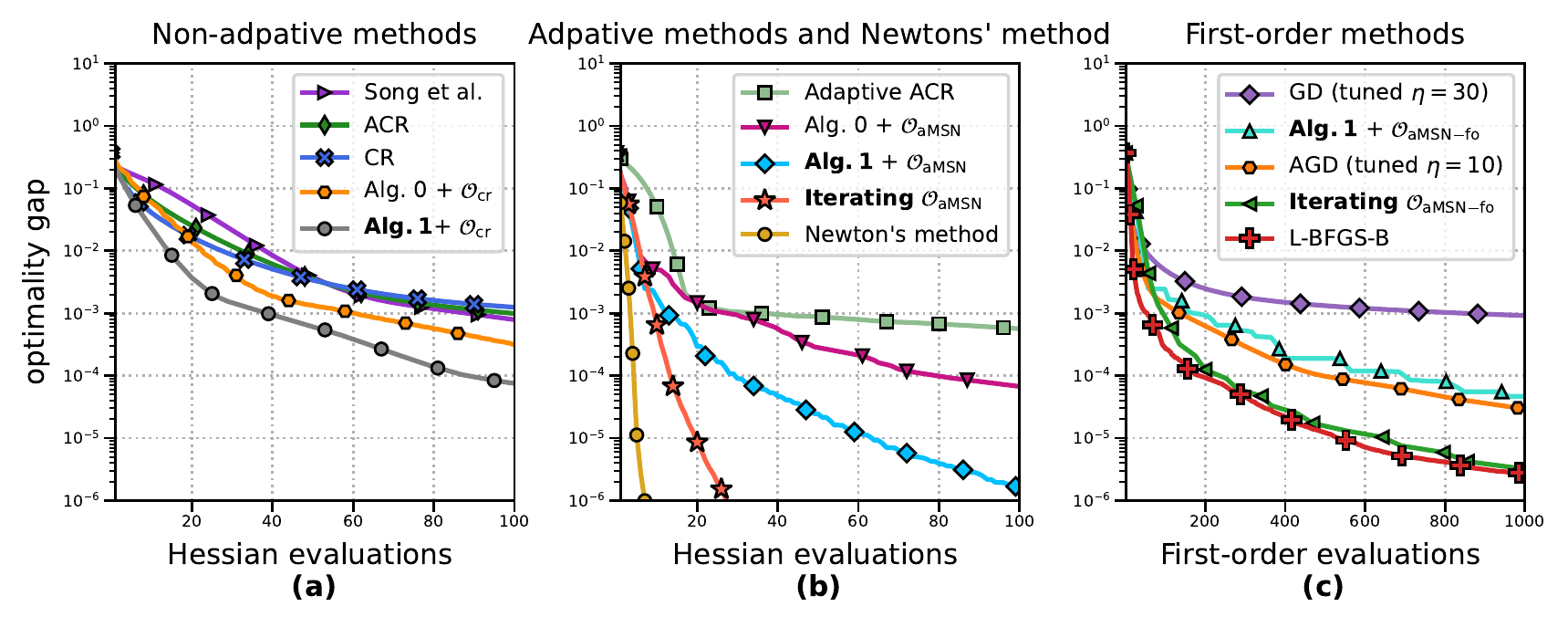}
	\caption{Empirical results for logistic regression on the ``a9a'' dataset. See \Cref{sec:experiments} for description, and \Cref{app:experiments-extra-experiments} for additional datasets. Boldface legend entries denote methods we contribute.
	}\label{fig:methods-comp}
\notarxiv{\vspace{-16pt}}
\end{figure}

We conduct three sets of experiments. First, we consider $\oracle[cr]$ with a fixed parameter $M$ and compare previous acceleration schemes to \Cref{alg:optms}. Second, we combine \Cref{alg:optms} with our adaptive $\adaOracle$  and test it against previous adaptive accelerated (second-order) methods and Newton's method.
Finally, we compare \Cref{alg:optms} with our first-order adaptive oracle $\foadaOracle$ to other first-order methods. We provide full implementation details in \Cref{app:experiments-technical-details}. 
\Cref{fig:methods-comp} summarizes our results for logistic regression on the  `a9a' dataset~\cite{chang2011libsvm}; see \Cref{app:experiments-extra-experiments} for similar results on three additional  datasets.
These experiments were conducted with \emph{no tuning} of \Cref{alg:optms}: the parameters $\sigma$ and $\alpha$ were simply set to $\half$ and $2$, respectively. An additional experiment, reported in \Cref{app:experiments-parameter-sensitivity}, shows that the algorithm is indeed insensitive to that choice.

\paragraph{Non-adaptive methods.}
We use the non-adaptive oracle $\oracle[cr]$~\eqref{eq:cr-oracle}, and take $M$ to be $0.2\bar{H}$ where, for feature vectors $\phi_1,\ldots,\phi_n$, $\bar{H} = \opnorm{\frac{1}{n}\sum_{i=1}^{n}\phi_i \phi_i^T}\max_{i \in [n]}\norm*{\phi_i}$ is an upper bound on $6\sqrt{3}\approx 10$ times the  Lipschitz constant of the logistic regression Hessian~\cite[see, e.g.,][]{song2021unified}. Fixing the MS oracle allows for a controlled comparison of different acceleration schemes: \Cref{fig:methods-comp}(a) shows that standard MS acceleration with a  carefully-implemented bisection outperforms standard cubic regularization (CR) and its accelerated counterpart  (ACR)~\cite[Alg.~4.8]{nesterov2008accelerating}, and removing the bisection via \Cref{alg:optms} yields the best results. We also implemented the heuristic suggested by \citet{song2021unified}, where instead of a bisection in \Cref{alg:ms} we select a sequence $\blambda_t$ such that $A_{t} =  \frac{1}{M\norm*{x_0 - \xopt}}\prn*{{t}/{3}}^{7/2}$. 
In~\Cref{app:experiments-tuning-cr} we tune the $M$ parameter for each method separately, finding that the optimal $M$ for CR is near $0$, so that $\oracle[cr]$ is nearly a Newton step (and not a valid MS oracle).

\paragraph{Adaptive methods and Newton's method.}
We compare the following adaptive accelerated second-order methods (which do not require an estimate of the Hessian Lipschitz constant): Adaptive ACR \cite[Algorithm 4]{grapiglia2020tensor} (which adaptively sets $M$ in $\oracle[cr]$), standard MS acceleration (\Cref{alg:ms}) with $\adaOracle$ (\Cref{alg:adaptive-msn-step}, with $\lazyflag=$ False) and \Cref{alg:optms} with $\adaOracle$ (with $\lazyflag=$ True in all but the first iteration).  \Cref{fig:methods-comp}(b) shows that the latter converges significantly faster than the other adaptive acceleration schemes. However, the classical ``unaccelerated'' Newton iteration $x_{t+1} = -(\hess f(x_t))^{-1} \grad f(x_t)$ strongly outperforms all ``accelerated'' methods, indicating that momentum mechanisms might actually be slowing down convergence in logistic regression problems. To test this, we consider the following simple iteration of (the non-lazy variant) of our oracle: $x_{t+1}, \lambda_{t+1} = \adaOracle(x_t; \lambda_t/2)$; it significantly improves over \Cref{alg:optms}.

These results beg the question: is momentum ever useful for second-order methods? In \Cref{app:experiments-worst-case} we test different schemes on the lower bound construction~\cite{arjevani2019oracle,grapiglia2020tensor}. We find momentum is helpful for $\oracle[cr]$, but not for the adaptive oracle $\adaOracle$. What makes Newton's method perform so well on logistic regression, and whether simply iterating $\adaOracle$ is worst-case optimal, are important questions for future work.

\paragraph{First-order methods.}
We compare our first-order adaptive $\foadaOracle$ (\Cref{alg:adaptive-msn-cg-step}) to the following baselines: gradient descent and accelerated gradient descent~\cite{nesterov1983method} with a tuned step size $\eta$, and L-BFGS-B from SciPy~\cite{byrd1995limited,zhu1997algorithm,scipy2020}. In light of the above comparison with Newton's method, we also test the following simple iteration of (the lazy variant) of our oracle: $x_{t+1},\lambda_{t+1} = \foadaOracle(x_t; \lambda_t/2)$.  \Cref{fig:methods-comp}(c) shows that forgoing (second-order) momentum is better for the first-order oracle, too: \Cref{alg:optms} performs comparably to tuned AGD (without tuning a single parameter), and the equally adaptive $\foadaOracle$ iteration performs comparably to with L-BFGS-B.

\section*{Acknowledgments}
We thank the anonymous reviewers for helpful questions and suggestions, leading to important writing clarifications, a simplification of the proof of \Cref{lem:growth-rate-high}, and an improved complexity bound for second-order methods under third-order smoothness (due to the observation that our framework is compatible with the oracle of~\cite{nesterov2021superfast}). YC and DH were supported in part by the Israeli Science Foundation (ISF) grant no.\ 2486/21 and the Len Blavatnik and the Blavatnik Family foundation. YJ was supported in part by a Stanford Graduate Fellowship and the Dantzig-Lieberman Fellowship. 
AS was supported in part by a Microsoft Research Faculty Fellowship, NSF CAREER Award CCF-1844855, NSF Grant CCF-1955039, a PayPal research award, and a Sloan Research Fellowship.

\bibliographystyle{abbrvnat}

\notarxiv{
\section*{Checklist}

\begin{enumerate}

	\item For all authors...
	\begin{enumerate}
		\item Do the main claims made in the abstract and introduction accurately reflect the paper's contributions and scope?
		\answerYes{}
		\item Did you describe the limitations of your work?
		\answerYes{}
		\item Did you discuss any potential negative societal impacts of your work?
		\answerNA{}
		\item Have you read the ethics review guidelines and ensured that your paper conforms to them?
		\answerYes{}
	\end{enumerate}

	\item If you are including theoretical results...
	\begin{enumerate}
		\item Did you state the full set of assumptions of all theoretical results?
		\answerYes{}
		\item Did you include complete proofs of all theoretical results?
		\answerYes{see supplementary material}
	\end{enumerate}

	\item If you ran experiments...
	\begin{enumerate}
		\item Did you include the code, data, and instructions needed to reproduce the main experimental results (either in the supplemental material or as a URL)?
		\answerYes{in the supplementary material}
		\item Did you specify all the training details (e.g., data splits, hyperparameters, how they were chosen)?
		\answerYes{full implementation details in~\Cref{app:experiments-technical-details}}
		\item Did you report error bars (e.g., with respect to the random seed after running experiments multiple times)?
		\answerNo{We only experimented with deterministic algorithms}
		\item Did you include the total amount of compute and the type of resources used (e.g., type of GPUs, internal cluster, or cloud provider)?
		\answerNo{We only ran small scale experiments on CPU}
	\end{enumerate}

	\item If you are using existing assets (e.g., code, data, models) or curating/releasing new assets...
	\begin{enumerate}
		\item If your work uses existing assets, did you cite the creators?
		\answerYes{}
		\item Did you mention the license of the assets?
		\answerYes{see~\Cref{app:experiments-technical-details}}
		\item Did you include any new assets either in the supplemental material or as a URL?
		\answerYes{in the supplemental material}
		\item Did you discuss whether and how consent was obtained from people whose data you're using/curating?
		\answerNA{}
		\item Did you discuss whether the data you are using/curating contains personally identifiable information or offensive content?
		\answerNA{}
	\end{enumerate}

	\item If you used crowdsourcing or conducted research with human subjects...
	\begin{enumerate}
		\item Did you include the full text of instructions given to participants and screenshots, if applicable?
		\answerNA{}
		\item Did you describe any potential participant risks, with links to Institutional Review Board (IRB) approvals, if applicable?
		\answerNA{}
		\item Did you include the estimated hourly wage paid to participants and the total amount spent on participant compensation?
		\answerNA{}
	\end{enumerate}

\end{enumerate}

 }	

\newpage
\appendix

\notarxiv{%
\arxiv{\subsection{Additional related work}\label{apdx:related}}
\notarxiv{\section{Additional related work}\label{apdx:related}}

\paragraph{Bisection-free methods for variational inequalities.} 

Monteiro and Svaiter also proposed second-order methods for solving variational inequalities for monotone operators with continuous derivatives~\cite{monteiro2012iteration}, and subsequent work provided improved rates for variational inequalities with continuous higher-order derivatives via tensor methods~\cite{bullins2020higher,jiang2022generalized}. These works also feature an implicit equation over a scalar regularization/step-size parameter, that necessitates a bisection and increases complexity by a logarithmic factor.  In recent papers, \citet{lin2022perseus} and~\citet{adil2022line} remove that logarithmic factor by developing  bisection-free methods for variational inequalities. However, applying these methods directly to convex optimization with Lipschitz $p$th derivatives yields a rate of $O(t^{-(p+1)/2})$ rather than the  optimal $O(t^{-(3p+1)/2})$ rate of our method. Moreover, these works remove bisections using techniques fundamentally different from ours. In particular, they do not apply a damping scheme on the $A_t$ sequence, nor do they apply a multiplicative update for the regularization parameter.

\paragraph{Adaptive Newton and tensor methods.} 
A number of works consider adaptive variants of the cubic-regularized Newton method and its tensor counterparts. \citet{cartis2011adaptive,gould2012updating} propose adaptive variants of cubic regularization for non-convex optimization. For convex optimization, 
\citet{mischenko2021regularized} provides a simple adaptive scheme converging at rate $O(t^{-2})$, followed by an improvement in its guarantee to $O(t^{-3})$ \cite{doikov2021gradient}. For tensor methods, \citet{jiang2020unified} proposes an adaptive regularization scheme for convex functions with Lipschitz-continuous $p^{th}$ derivatives which achieves the rate $O(t^{-p-1})$. In addition, \citet{grapiglia2020tensor} gives analogous results under $\nu$-H\"older continuity of the derivatives. 

Even for the case second-order methods with Lipschitz-continuous Hessian, an adaptive scheme with optimal rate $O(t^{-3.5})$ remained open prior to this work. Beyond removing the bisection from the MS framework, our key algorithmic techniques include directly considering quadratically-regularized Newton step (similar to~\cite{mischenko2021regularized,doikov2021gradient} and different from~\cite{grapiglia2020tensor,jiang2020unified}), and using the original MS approximation condition for selecting an appropriate regularization parameter, which is new in the context of adaptive methods. These techniques  allow us to adapt to both the constant and order of Hessian \Holder continuity simultaneously. 

\paragraph{Comparison to~\cite{kovalev2022first}.} 
In concurrent and independent theoretical work, \citet{kovalev2022first} propose an algorithm that also attains the optimal $p$th derivative evaluation complexity for convex optimization with Lipschitz $p$th derivatives. While also inspired by MS acceleration, the algorithm of~\cite{kovalev2022first} is quite different from ours: they replace the implicitly defined regularization parameters inherent to MS oracles by approximating proximal points with explicit, predetermined regularization parameters. To obtain these proximal points they apply a tensor-extragradient method, and stop it when the MS condition is met. By careful analysis, they show that---even though an individual outer acceleration step requires multiple derivative evaluations---the overall complexity of their method is optimal. In contrast, our method makes a single oracle call (high-order derivative evaluation) per step. To obtain optimal complexity, our method relies on a dynamic sequence of ``guessed'' regularization parameters and a momentum damping schemes that handles cases where these guesses overshoot. The two works offer complementary viewpoints of the algorithmic innovation required for removing bisection from the MS framework. 

In comparison to~\citet{kovalev2022first}, we believe that our algorithm offers advantages in terms of generality and adaptivity. From a generality perspective, our algorithm applies to every setting where the original MS framework applies. In addition to functions with Lipschitz $p$th derivatives, that includes functions with \Holder derivatives~\cite{song2021unified}, ball minimization oracles~\cite{carmon2020acceleration}, and a second-order oracles for functions with Lipschitz third derivative~\cite{nesterov2021superfast,kamzolov2020near}. While extending~\cite{kovalev2022first} to these settings may be possible, it would require additional work in formulating and analyzing an appropriate subproblem solver. Regarding adaptivity, like the original MS framework, our algorithm is agnostic to both the order of the Lipschitz derivative and its corresponding Lipschitz constant.
In contrast, \citet{kovalev2022first} require the derivative order for determining the regularization parameters, and the Lipschitz constant for the subproblem solver. 

\notarxiv{\newpage} %

}
\section{Proof of \Cref{thm:main}}\label{apdx:framework}

In this section, we provide a complete proof for~\Cref{thm:main}. We begin in \Cref{sec:main_thm:potential} with establishing a potential decrease result analogous to the standard analyses of accelerated proximal methods (\Cref{prop:main-potential}). In \Cref{sec:main_thm:at_lower} we then provide a lower bound of $A_T$ in terms of the values of $\blambda_t$ at a subset of ``up'' iterations where $\blambda_t > \lambda_t$ (\Cref{lem:increase}). We apply these results along with the reverse \Holder inequality to obtain \Cref{thm:main} in \Cref{sec:main_thm:proof}. The last part of the analysis relies on two technical lemmas on the growth rates of sequences satisfying certain recursive inequalities (\Cref{lem:growth-rate-high,lem:growth-rate-ball}), which we prove at the end of the section in \Cref{sec:main_thm:helper}. Beyond its utility for proving \Cref{thm:main}, Lemma 1 includes additional lower bounds on $A_T$ in terms of $\lambda'_t$ which we use in the analysis of adaptive oracle implementations in \Cref{apdx:solvers}.

We use the following notation for the set of ``down'' iterations:
\begin{equation*}
	\SdownT\defeq\{t\in[T] \mid \lambda_{t}\leq\blambda_{t}\}
\end{equation*}
and analogously define 
$\SsupT$ (``up'' iterations), $\SupT$ , $\SsdownT$ and $\SequalT$.

\subsection{Potential decrease}
\label{sec:main_thm:potential}

\begin{proposition}\label{prop:main-potential}
	Under the assumptions of~\Cref{thm:main}, let $E_t\defeq f(x_t)-f(\xopt)$, $D_t \defeq \frac{1}{2}\norm{v_t - \xopt}^2$, and $N_{t+1}\defeq \frac{1}{2}\norm{\tx_{t+1}-y_t}^2$ for all $t\ge 0$. Then, for all $t\ge0$
	\begin{equation}\label{eq:main-potential}
	A_{t+1}E_{t+1}+D_{t+1}+(1-\sigma^2)A_{t+1}'\min(\lambda_{t+1},\blambda_{t+1})N_{t+1}\le A_tE_t+D_t.
	\end{equation}
	Consequently, for all $T\ge1$, $\sqrt{A_T}\ge \frac{1}{2}\sum_{t\in\SdownT}1/\sqrt{\blambda_{t}}$,
	\begin{equation}\label{eq:main-potential-implied}
	E_{T}\le \frac{D_0}{A_T}~~\text{, and}~~(1-\sigma^2)\sum_{t\in \SupT}A_{t}\lambda_{t}'N_{t}\le D_0-A_TE_T\,.
	\end{equation}
\end{proposition}
	
\begin{proof}
By definition of $D_t$ and the definition of $v_{t+1}$ in~\cref{line:v-update}, we have 
\begin{align*}
D_{t+1}& =\frac{1}{2}\norm{v_{t+1} - \xopt}^2
=\frac{1}{2}\norm{(v_{t}-a_{t+1}\grad f(\tx_{t+1})) - \xopt}^{2}
\\ &
=D_{t}+a_{t+1}\inner{\grad f(\tx_{t+1})}{\xopt - v_{t}}+\frac{a_{t+1}^{2}}{2}\norm{\grad f(\tx_{t+1})}^{2}\,.\numberthis \label{eq:potential-sandwitch-top}
\end{align*}
Also, by the definition $y_t$ in~\cref{line:y-update} and $A'_{t+1} \defeq A_t+a'_{t+1}$ in~\cref{line:A-update}, we have
\[
\ba_{t+1}v_{t}=\bA_{t+1}y_{t}-A_{t}x_{t}=\ba_{t+1}\tx_{t+1}+\bA_{t+1}(y_{t}-\tx_{t+1})-A_{t}(x_{t}-\tx_{t+1})\,.
\]
Subtracting $a_{t+1}' \xopt$ from both sides and considering the inner product with $\nabla f(\tx_{t+1})$ then yields 
\begin{flalign*}
& \ba_{t+1}\inner{\grad f(\tx_{t+1})}{\xopt - v_{t}}\\
& \hspace{1.5em} =\grad f(\tx_{t+1})^{\T}\left[\ba_{t+1}(\xopt - \tx_{t+1})+\bA_{t+1}(\tx_{t+1}-y_{t})+A_{t}(x_{t}-\tx_{t+1})\right]\\
 &\hspace{1.5em} \stackrel{(i)}{\leq} \ba_{t+1}[f(\xopt)-f(\tx_{t+1})]+\bA_{t+1}\inner{\grad f(\tx_{t+1})}{\tx_{t+1} - y_{t}}+A_{t}[f(x_{t})-f(\tx_{t+1})]\\
 &\hspace{1.5em} \stackrel{(ii)}{\leq} A_{t}E_{t}-\bA_{t+1}[f(\tx_{t+1})-f(\xopt)]+\bA_{t+1}\inner{\grad f(\tx_{t+1})}{\tx_{t+1} - y_{t}}\,.
\end{flalign*}
where we used $(i)$ convexity of $f$ and $(ii)$ again that $A'_{t+1}=A_t+a'_{t+1}$ (\cref{line:A-update}).

Next, note that we can upper bound $\inner{\grad f(\tx_{t+1})}{\xopt - v_{t}}$ as 
\begin{equation*}
\begin{aligned}
&\lambda_{t+1}\inner{\grad f(\tx_{t+1})}{\tx_{t+1} - y_{t}} \\
&  \hspace{1.5em} =\frac{1}{2}\norm{\grad f(\tx_{t+1})+\lambda_{t+1}(\tx_{t+1}-y_{t})}^{2}-\frac{1}{2}\norm{\grad f(\tx_{t+1})}^{2}-\frac{\lambda_{t+1}^{2}}{2}\norm{\tx_{t+1}-y_{t}}^{2}\\
&  \hspace{1.5em} \leq-\lambda_{t+1}^{2}(1-\sigma^{2})N_{t+1}-\frac{1}{2}\norm{\grad f(\tx_{t+1})}^{2}\,,
\end{aligned}
\end{equation*}
where for the inequality we used that $\oracle$ is a $\sigma$-MS oracle for $f$ (\cref{def:ms-condition}) and the definition of $N_t$.
Substituting back gives
\begin{flalign*}
	&\bA_{t+1}[f(\tx_{t+1})-f(\xopt)] \le 
	\\& \hspace{1.5em} A_t E_t + \ba_{t+1} \inner{\grad f(\tx_{t+1})}{v_{t}-\xopt} - (1-\sigma^2)\bA_{t+1}\lambda_{t+1}N_{t+1} - \frac{\bA_{t+1}}{2\lambda_{t+1}} \norm{\grad f(\tx_{t+1})}^2.\numberthis\label{eq:potential-sandwitch-bottom}
\end{flalign*}

We separately consider the cases $\lambda_{t+1}\le \lambda_{t+1}'$ and $\lambda_{t+1}> \lambda_{t+1}'$. First, when $\lambda_{t+1}\le \lambda_{t+1}'$, by definition in the algorithm $x_{t+1}=\tx_{t+1}$, $a_{t+1}=a'_{t+1}$, $A_{t+1} = A'_{t+1}$ and by~\cref{line:a-quandratic} we have $A_{t+1} = \lambda_{t+1}' a_{t+1}^2$. Consequently, we can combine~\eqref{eq:potential-sandwitch-top} and~\eqref{eq:potential-sandwitch-bottom} to conclude that
\begin{equation}\label{eq:potential-1}
\begin{aligned}
	A_{t+1}E_{t+1}+D_{t+1}+\lambda_{t+1}A_{t+1}(1-\sigma^2)N_{t+1} & \le A_tE_t+D_t+\left(\frac{a_{t+1}^{2}}{2}-\frac{A_{t+1}}{2\lambda_{t+1}}\right)\norm{\grad f(\tx_{t+1})}^{2}\\
	&\leq A_tE_t+D_t.
	\end{aligned}
\end{equation}

On the other hand, when $\lambda_{t+1}>\blambda_{t+1}$, by the definition of $\gamma_{t+1}=\blambda_{t+1}/\lambda_{t+1}$ in~\cref{line:gamma-define}, $a_{t+1}, A_{t+1}$ in~\cref{line:a-update-2}, and $x_{t+1}$ in~\cref{line:x-update-2}, we have $A_{t+1}=(1-\gamma_{t+1})A_t + \gamma_{t+1}\bA_{t+1}$, and therefore convexity of $f$ implies that 
\begin{equation*}
	f(x_{t+1})\le \frac{(1-\gamma_{t+1})A_t}{A_{t+1}} f(x_t)+\frac{\gamma_{t+1}A'_{t+1}}{A_{t+1}}f(\tx_{t+1}).
\end{equation*}

Subtracting $f(\xopt)$, multiplying by $A_{t+1}$, combining with~\eqref{eq:potential-sandwitch-bottom} to bound $f(\tx_{t+1})$ and noting that $\gamma_{t+1}\ba_{t+1} = a_{t+1}$ yields
\begin{flalign*}
	& A_{t+1}E_{t+1} \le (1-\gamma_{t+1})A_t E_t + \gamma_{t+1}\bA_{t+1} [f(\tx_{t+1})-f(\xopt)]
	\\ & \hspace{1.5em} \le A_t E_t + a_{t+1} \inner{\grad f(\tx_{t+1})}{\xopt - v_{t}} - (1-\sigma^2)\bA_{t+1}\blambda_{t+1}N_{t+1} - \frac{\gamma_{t+1}\bA_{t+1}}{2\lambda_{t+1}} \norm{\grad f(\tx_{t+1})}^2.
\end{flalign*}
Noting that  $A'_{t+1} = \blambda_{t+1}(a'_{t+1})^2=\frac{\lambda_{t+1}}{\gamma_{t+1}}a_{t+1}^2$ by definition and further substituting~\eqref{eq:potential-sandwitch-top} into the above display yields
\begin{equation}\label{eq:potential-2}
	A_{t+1} E_{t+1} \le A_t E_t + D_t - D_{t+1} - (1-\sigma^2)\bA_{t+1}\blambda_{t+1}N_{t+1},
\end{equation}
which, when combined with~\eqref{eq:potential-1} yields~\eqref{eq:main-potential}.

The bound on $A_T$ follows from standard argument for Monteiro-Svaiter acceleration restricting to the proper set, i.e. $\SdownT$, see e.g.\ Lemma 27 in~\cite{carmon2020acceleration}, we include here for completeness.
\begin{align*}
\sqrt{A_T} = \sqrt{A_T}-\sqrt{A_0} & = \sum_{t=0}^{T-1}	\frac{A_{t+1}-A_{t}}{\sqrt{A_{t+1}}+\sqrt{A_{t}}} \ge \sum_{t+1\in\SdownT}\frac{a'_{t+1}}{\sqrt{A'_{t+1}}+\sqrt{A_{t}}}\\
& = \sum_{t+1\in\SdownT}\frac{\sqrt{A'_{t+1}/\lambda_{t+1}'}}{\sqrt{A'_{t+1}}+\sqrt{A_{t}}}\ge \frac{1}{2}\sum_{t\in\SdownT}\sqrt{1/\blambda_{t}}.
\end{align*}
For the second line we used that $\blambda_{t+1}(a_{t+1}')^2 = A'_{t+1}$ and that $A'_{t}$ is increasing in $t$. Finally, the conclusions in \eqref{eq:main-potential-implied} follow from inductively applying \eqref{eq:main-potential} and using $A_0 = 0$. 
\end{proof}

\subsection{Lower bounding $A_T$ using ``up'' iterates}
\label{sec:main_thm:at_lower}

Next, we provide more fine-grained bounds on the growth of $A_t$, implied by the adaptive scheme for updating $\blambda$ in~\cref{line:lam-update-1} and~\ref{line:lam-update-2}.
\begin{restatable}{lemma}{lemincrease}
	\label{lem:increase}
	In the setting of Theorem~\ref{thm:main}, for any $\widehat{T}\ge 1$, there exists a non-empty set $\QTS\subseteq \SsupT[\widehat{T}] \cup \{1\}$ and positive numbers 
	$r_{t}$ for each $t\in \QTS$ such that 
	\begin{equation}\label{eq:increase-r}
		\sum_{t\in \QTS} r_t = \frac{\widehat{T}-1}{2},
	\end{equation}
and
\begin{equation}\label{eq:increase-A}
 	\sqrt{A_{\widehat{T}}} \geq\frac{1}{4\sqrt{\alpha}}\sum_{t\in \QTS}\sqrt{\frac{\alpha^{r_{t}-1}}{\blambda_{t}}}.
 \end{equation}
 Further, the definition is consistent in the sense that for any $T\ge 1$ and defined $\QT$, for any $\widehat{T}_1,\widehat{T}_2\in\QT$, suppose $\widehat{T}_1<\widehat{T}_2$ and $r_{t, \widehat{T}}$ are the numbers when applying previous argument to $\widehat{T}$, then $\mathcal{Q}_{\widehat{T}_1}\subseteq \mathcal{Q}_{\widehat{T}_2}$ and $r_{t, \widehat{T}_1} = r_{t, \widehat{T}_2}$ for any $t\le \widehat{T}_1$. Thus, we omit the second subscript in defining $r_t$ when clear from context.
 
Furthermore, for $T\ge 1$,
\begin{equation}\label{eq:increase-A-all}
	\sqrt{A_{T}} \ge \frac{\sqrt{\alpha}-1}{4\alpha}\sum_{t\in [T]}\sqrt{\frac{1}{\blambda_t}}.
\end{equation}
\end{restatable}

\begin{proof}
We define $\QTS$ to be the set of ``up-down'' iterates, i.e., iterates $t$ for which $\lambda_{t} > \blambda_{t}$ but $\lambda_{t+1} \le \blambda_{t+1}$; we also add to $\QTS$ the first iterate and, if $\widehat{T}\in\SsupT[\widehat{T}]$, the iterate $\widehat{T}$. Formally, we have
\begin{align*}
	\QTS\defeq(\SsupT[\widehat{T}] \cap \{t\mid t+1\in\SdownT[\widehat{T}] \mbox{~~or~~} t=\widehat{T}\})\cup \{1\} .
\end{align*}
We let $1 = \tau_1< \tau_2<\cdots<\tau_S\le \widehat{T}$ denote the 
$S=|\QTS|$ distinct elements of
$\QTS$ in increasing order. For notational convenience, we also let
 $\tau_{S+1}\defeq \widehat{T}$. 

For every $i\in [S]$, we let $n_i$ be the index of the last ``down'' iterate between $\tau_{i}$ and $\tau_{i+1}$ (and $\widehat{T}$ if $i =S$), that is
\begin{equation}
	n_i \defeq \begin{cases}
 \max\{t\in \SdownT[\widehat{T}] |\tau_i \le t < \tau_{i+1}\}~&~\text{if}~i<S\\
\widehat{T} ~&~\text{otherwise}
 \end{cases}.
\end{equation}
As an immediate consequence of the definition of $\tau_i$ and $n_i$, we have for all $i<S$, $n_i\in [\tau_i,\tau_{i+1})$. We also have that the set of $n_i$ are distinct, i.e.\ $n_i \neq n_j$ for all $i,j \in [S]$ with $i \neq j$. 

Note that between any two ``up-down'' iterates $\tau_i$ and $\tau_{i+1}$ we have a sequence of ``down'' iterates (ending at $n_i$) followed by a sequence of ``up iterates'' (ending at $\tau_{i+1}$). In other words, for all $i<S$ and  $k\in(n_i,\tau_{i+1}]$ we have  $k\in \SsupT[\widehat{T}]$. Consequently, $\blambda_{k+1}=\alpha\blambda_{k}$ for all $k\in(n_i,\tau_{i+1})$ (since these are ``up'' iterates). Since $\blambda_{n_i+1}= \alpha^{-1}\blambda_{n_i}$ (because $n_i$ is a ``down'' iterate), we conclude that $\blambda_{n_i}=\alpha^{2-(\tau_{i+1}-n_i)}\blambda_{\tau_{i+1}}$. Combining this with~\Cref{prop:main-potential} implies the following lower bound on $ \sqrt{A_{\widehat{T}}}$:
\begin{align}
\label{eq:At_lower_1}
\sqrt{A_{\widehat{T}}}& \geq\frac{1}{2}\sum_{t\in\SdownT[\widehat{T}]}\frac{1}{\sqrt{\lambda_t'}} \ge \frac{1}{2}\sum_{i\in[S-1]}\frac{1}{\sqrt{\blambda_{n_i}}}\ge\frac{1}{2}\sum_{i\in[S-1]}\sqrt{\frac{\alpha^{\tau_{i+1}-n_i-2}}{\blambda_{\tau_{i+1}}}}.
\end{align}

Further, as argued above, for $k\in(\tau_i,n_i]$ we have $k\in\SdownT[\widehat{T}]$ for all and therefore $\blambda_{k+1} = \blambda_{k}/\alpha$. Consequently, when $\tau_i<\widehat{T}$ we have $\blambda_{n_i} = \alpha^{2-(n_i-\tau_i)}\blambda_{\tau_i}$. When $\tau_i=\widehat{T}$ the  inequality also holds since $\tau_i = n_i$. Together with the conclusion of~\Cref{prop:main-potential}, this implies the following lower bound on $ \sqrt{A_{\widehat{T}}}$:
\begin{align}
\label{eq:At_lower_2}	
\sqrt{A_{\widehat{T}}} & \geq \frac{1}{2}\sum_{t\in\SdownT[\widehat{T}]}\frac{1}{\sqrt{\lambda_t'}}\ge \frac{1}{2}\sum_{i\in[S]}\frac{1}{\sqrt{\blambda_{n_i}}}\ge \frac{1}{2}\sum_{i\in[S]}\sqrt{\frac{\alpha^{n_i-\tau_i-2}}{\blambda_{\tau_{i}}}}.
\end{align}

We now define $r_i$ as follows
\[r_{\tau_i}  = \begin{cases}
\half\prn*{n_1-1}~&~\text{if}~i=1\\
 	\half\prn*{n_i-n_{i-1}}~&~\text{if}~1<i\le S.
 \end{cases}
\]
Clearly we have $r_t\ge 0$ for all $t\in \QTS$  and $\sum_{t\in \QTS} r_{t} = \sum_{i\in[S]} r_{\tau_i} = \frac{\widehat{T}-1}{2}$, which proves \eqref{eq:increase-r}.

To show \eqref{eq:increase-A}, note that for any $\alpha\ge 1$, $\frac{1}{2}\sqrt{\alpha^a}+\frac{1}{2}\sqrt{\alpha^b}\ge \sqrt{\alpha^{\frac{1}{2}a+\frac{1}{2}b}}$ due to the arithmetic and geometric mean (AM-GM) inequality. Averaging our two lower bounds on $\sqrt{A_{\widehat{T}}}$, \eqref{eq:At_lower_1} and \eqref{eq:At_lower_2}, we conclude that 
\begin{align*}
\sqrt{A_{\widehat{T}}}\ge \frac{1}{4}\biggl(\sum_{i=2}^S\sqrt{\frac{\alpha^{\tau_{i}-n_{i-1}-2}}{\blambda_{\tau_{i}}}}+\sum_{i\in[S]}\sqrt{\frac{\alpha^{n_i-\tau_i-2}}{\blambda_{\tau_{i}}}}\biggr)\ge \frac{1}{4}\sum_{i\in[S]}\sqrt{\frac{\alpha^{r_{\tau_i}-2}}{\blambda_{\tau_i}}}=\frac{1}{4}\sum_{t\in \QTS}\sqrt{\frac{\alpha^{r_t-2}}{\blambda_{t}}}.
\end{align*}
Here the first term on the RHS bound comes purely from $\sqrt{\frac{\alpha^{n_i-\tau_i-2}}{\blambda_{\tau_{i}}}}$ when $i=1$ since $n_1-\tau_1>\frac{1}{2}\left(n_1-\tau_1\right) = \frac{1}{2}r_{\tau_1}$ which leads to the coefficient of $1/4$ on RHS.

Now for the consistency arguments, note by definition of $\mathcal{Q}$ and $r_{t}$ we have $\mathcal{Q}_{\widehat{T}_1}\subseteq \mathcal{Q}_{\widehat{T}_2}$ and $r_{t,\widehat{T}_1} = r_{t,\widehat{T}_2}$ for any $\widehat{T}_1<\widehat{T}_2\in\QTS$. 

To show the second inequality~\eqref{eq:increase-A-all}, we start again with $\widehat{T} = T$. From the conclusion of~\Cref{prop:main-potential} and the observation that $k\in\SdownT$ for any $k\in(\tau_i,n_i]$, giving 
\begin{align}
\label{eq:AT_lower_3}
\sqrt{A_T} & \ge \frac{1}{2}\sum_{t\in\SdownT}\frac{1}{\sqrt{\blambda_{t}}}
=
\frac{1}{2}\sum_{t=1}^{n_1}\sqrt{\frac{1}{\blambda_{t}}}+\frac{1}{2}\sum_{i=2}^{S}\sum_{t = \tau_i+1}^{n_i}\sqrt{\frac{1}{\blambda_{t}}}.
\end{align}

Moreover, since for $i<S$ and $k \in (n_i,\tau_{i+1}]$ we have $\blambda_{k} = \alpha^{(k-n_i-2)} \blambda_{n_i}$, and
\begin{align}
\label{eq:AT_lower_4}
\sum_{t\in(n_i,\tau_{i+1}]}\sqrt{\frac{1}{\blambda_t}}= \left(\sum_{j=1}^{\tau_{i+1}-n_i}\frac{\alpha}{\alpha^{j/2}}\right)\sqrt{\frac{1}{\blambda_{n_i}}}\le \frac{\alpha}{\sqrt{\alpha}-1}\sqrt{\frac{1}{\blambda_{n_i}}}.
\end{align} 
Combining \eqref{eq:AT_lower_3} and \eqref{eq:AT_lower_4} with $\sqrt{A_T} \geq \frac{1}{2}\sum_{i \in [S-1]}\sqrt{\frac{1}{\blambda_{n_i}}}$
 yields  \eqref{eq:increase-A-all} since
\begin{align*}
\sqrt{A_T} & \ge \frac{1}{4}\left(\sum_{t=1}^{n_1}\sqrt{\frac{1}{\blambda_{t}}}+\sum_{i=2}^{S}\sum_{t = \tau_i+1}^{n_i}\sqrt{\frac{1}{\blambda_{t}}}\right)+\frac{1}{4}\sum_{i=1}^{S-1}\sqrt{\frac{1}{\blambda_{n_i}}}\\
& \ge \frac{1}{4}\left(\sum_{t=1}^{n_1}\sqrt{\frac{1}{\blambda_{t}}}+\sum_{i=2}^{S}\sum_{t = \tau_i+1}^{n_i}\sqrt{\frac{1}{\blambda_{t}}}\right)+\frac{\sqrt{\alpha}-1}{4\alpha}\sum_{i=1}^{S-1}\sum_{t\in(n_i,\tau_{i+1}]}\sqrt{\frac{1}{\blambda_t}}\\
& \ge \frac{\sqrt{\alpha}-1}{4\alpha}\sum_{t\in[T]}\sqrt{\frac{1}{\blambda_t}}\,.
\end{align*}
\end{proof}

\subsection{Completing the proof of \Cref{thm:main}}
\label{sec:main_thm:proof}
We now show how to use~\Cref{prop:main-potential} and~\Cref{lem:increase} to obtain optimal  acceleration, considering the cases $s\in(1,\infty)$, $s=\infty$, and $s=1$ in turn.

\paragraph{The $s\in(1,\infty)$ case.} 

If $E_T\le 0$, the result $f(x_T) - f(\xopt) \le 0$ follows immediately. Therefore, it suffices to consider the case when $E_T>0$. 
For any $\widehat{T}\in\QT$, applying~\Cref{prop:main-potential} and~\Cref{lem:increase}
(using that movement bounds hold for all iterations in $\QT$ including the first iterate $t=1$ by assumption)
yields
\begin{equation}
	D_0\ge D_{0}-A_{\widehat{T}}E_{\widehat{T}}\geq\sum_{t\in \QTS}A_{t}\blambda_{t}(1-\sigma^{2})M_{t}\geq\frac{1-\sigma^2}{2}c^{-\frac{2s}{s-1}}\sum_{t\in \QTS}A_{t}\left(\blambda_{t}\right)^{\frac{s+1}{s-1}}\ge0.\label{eq:holder-fodder-high}
\end{equation}
This implies $E_{\widehat{T}}\le D_0/A_{\widehat{T}}$ where $\sqrt{A_{\widehat{T}}}\geq\frac{1}{4\sqrt{\alpha}}\sum_{t\in \QTS}\sqrt{\frac{\alpha^{r_{t}-1}}{\blambda_{t}}}$
for $\sum_{t\in \QTS}r_{t}=\frac{\widehat{T}-1}{2}$. 

The reverse \Holder inequality (which is a standard technique in analyzing MS acceleration~\cite{gasnikov2019optimal, bubeck2019optimal, jiang2019optimal,song2021unified,alves2021variants}) states that, for all $q>1$, and any two vectors $u,v$ with positive elements,
\begin{equation*}
	\sum_i u_i v_i \ge \prn*{ \sum_i u_i^{1/q} } ^q \prn*{ \sum_i v_i^{1/(1-q)} } ^{1-q}.
\end{equation*}
We set $q=\frac{3s+1}{2(s+1)}$ and apply the reverse \Holder inequality to obtain
\begin{align*}
4\sqrt{\alpha A_{\widehat{T}}} & \geq\sum_{t\in \QTS}\sqrt{\frac{\alpha^{r_{t}-1}}{\blambda_{t}}}=\sum_{t\in \QTS}\left(A_{t}^{q-1}\sqrt{\alpha^{r_{t}-1}}\right)\left(\frac{A_{t}^{1-q}}{\sqrt{\blambda_{t}}}\right)\\
 & \overge{(i)}\left(\sum_{t\in \QTS}A_{t}^{1-\frac{1}{q}}\alpha^{\frac{r_{t}-1}{2q}}\right)^{q}\left(\sum_{t\in \QTS}A_{t}\left(\blambda_{t}\right)^{\frac{s+1}{s-1}}\right)^{1-q}\\
 & \overge{(ii)}\left(\sum_{t\in \QTS}A_{t}^{1-\frac{1}{q}}\left(\alpha^{\frac{1}{2q}}\right)^{r_{t}-1}\right)^{q}\left(\frac{2}{1-\sigma^2}D_{0}c^{\frac{2s}{s-1}}\right)^{1-q}\\
 & \overge{(iii)}\left(\sum_{t\in \QTS}A_{t}^{1-\frac{1}{q}}r_{t}\cdot c_{\alpha,q}\right)^{q}\left(\frac{2D_{0}}{(1-\sigma^2)c^{-\frac{2s}{s-1}}}\right)^{1-q}~~\text{for}~~c_{\alpha,q}\defeq\min\left(1,\frac{1}{2q}\ln\alpha\right) \numberthis\label{eq:holder-upshot-high}
\end{align*}
where we used $(i)$ the reverse \Holder inequality with $u_t = A_t^{q-1}\sqrt{\alpha^{r_t-1}}$ and $v_t = A_t^{1-q}/\sqrt{\blambda_t}$ (for $t\in\QTS$) and $-\half \cdot \frac{1}{1-q} = \frac{s+1}{s-1}$, $(ii)$ the bound~\eqref{eq:holder-fodder-high}, and $(iii)$ the following lemma (proved in the next subsection) with $a\gets\alpha^{1/2q}$ 
and $b\gets r_t\ge0$. 

\begin{restatable}{lemma}{restateLemExpToLinear}\label{lem:exp-to-linear}
	For all $a\geq0$
	and $b\geq1$, we have $a^{b-1}\geq\min\{1,\ln a\}\cdot b$.
\end{restatable}

Substituting the definitions 
\begin{equation*}
	B_{t}\defeq  A_{t}^{1-\frac{1}{q}}~~\mbox{and}~~\beta\defeq c_{\alpha,q}\left(\frac{1}{4\sqrt{\alpha}}\right)^{\frac{1}{q}}\left(\frac{2D_{0}}{(1-\sigma^2)c^{-\frac{2s}{s-1}}}\right)^{\frac{1-q}{q}},
\end{equation*}
the bound~\eqref{eq:holder-upshot-high} can be rewritten as
\[
B_{\widehat{T}}^{\frac{s+1}{s-1}}=B_{\widehat{T}}^{\frac{1}{2(1-q)}}\geq\beta\cdot\sum_{\tau\in \QTS}B_{\tau}\cdot r_\tau~~~\text{for all}~~~\widehat{T}\in\QT. 
\]

To deduce the growth rate of $B_t$, we give the following lemma generalizing the analyses in prior work \cite{bubeck2019optimal, jiang2019optimal} (see proof in the next \Cref{sec:main_thm:helper}).
\begin{restatable}{lemma}{restateLemGrowthRateHigh}\label{lem:growth-rate-high}
	Let $B_{1},...,B_{k}\in\R_{>0}$, $r_{1},...,r_{k}\in\R_{\ge 0}$ and $\beta>0$. \yair{I am adding $\beta$ back in for consistency with the other lemma and to avoid unannounced notation overloading} Further, suppose that for some
 $m>1$\ and all $i\in[k]$ it is the case that $B_{i}^{m}\geq \beta \sum_{j\in[i]}B_{j}\cdot r_{j}$.
	Then for all $i\in[k]$ we have that $B_{i}\geq\left(\frac{m-1}{m} \beta\cdot \sum_{j\in[i]} r_{j}\right)^{1/(m-1)}$.
\end{restatable} 
Applying the lemma with $m=\frac{s+1}{s-1}$ and recalling that $ \sum_{t\in\QT} r_t = \frac{T-1}{2}$, we obtain (for $T'=\max \QT$), 
\[
B_{T}\geq B_{T'}\ge\left(\frac{\frac{s+1}{s-1}-1}{\frac{s+1}{s-1}}\cdot\beta\cdot\sum_{t\in \QT}r_{t}\right)^{\frac{1}{\frac{s+1}{s-1}-1}}=\left(\frac{2}{s+1}\cdot\beta\cdot\frac{T-1}{2}\right)^{\frac{s-1}{2}}~~\text{for any}~T\in\Z_{>0}.
\]
Since $A_{T}=B_{T}^{\frac{q}{q-1}}$and $\frac{q}{q-1}=\frac{3s+1}{s-1}$
this gives the desired growth rate of
\begin{align*}
A_T & \ge \left(\frac{2}{s+1}\cdot\beta\cdot\frac{T-1}{2}\right)^{\frac{3s+1}{2}}~~\text{for any}~T\in\Z_{>0},
\end{align*}
where, substituting back, we have %
\begin{equation*}
	\beta = \min\left(1,\frac{s+1}{3s+1}\ln\alpha\right)\cdot \left(\frac{1}{4\sqrt{\alpha}}\right)^{\frac{2s+2}{3s+1}}c^{-\frac{2s}{3s+1}}\left(\frac{2}{(1-\sigma^2)}D_{0}\right)^{-\frac{s-1}{3s+1}}.
\end{equation*}
Thus, for 
\[T= \Omega\left(\frac{\left(\frac{\frac{1}{2}\norm{x_0-x_\star}^2}{\epsilon}\right)^{\frac{2}{3s+1}}}{\frac{1}{s+1}\beta}\right) = \Omega\left(\frac{\alpha^{\frac{s+1}{3s+1}}}{\min\left(1,\frac{1}{s}\ln \alpha\right)}\cdot\prn*{\frac{c^{s} \norm{x_0 - \xopt}^{s+1}}{\epsilon}}^{\frac{2}{3s+1}}\right),\] 
we have $A_T\ge \frac{1}{2}\norm{x_0-x_\star}^2/\epsilon$, and consequently
\begin{align*}
f(x_T)-f(x_\star) = E_T \le \frac{\frac{1}{2}\norm{x_0-x_\star}^2}{A_T} \le \epsilon.
\end{align*}
The case for $s\in(1,\infty)$ follows immediately. %
	
\paragraph{The $s=\infty$ case.} Considering $s=\infty$ and $q=\frac{3}{2}$ in~\eqref{eq:holder-upshot-high} yields for any $T\in\Z_{>0}$ and $\widehat{T}\in\QT$,
\begin{equation*}
	4\sqrt{\alpha A_{\widehat{T}}} \ge 
	\left(\sum_{t\in \QTS}A_{t}^{\frac{1}{3}}\min\left(1,\frac{1}{3}\ln\alpha\right)\cdot r_{t}\right)^{\frac{3}{2}}\left(\frac{2c^2D_{0}}{1-\sigma^2}\right)^{-\frac{1}{2}}.
\end{equation*}
Defining 
\begin{equation*}
	B_t \defeq A_t^{1/3}~~\mbox{and}~~\beta \defeq \min\{1,\tfrac{1}{3}\ln\alpha\}\cdot \left(2^5c^2\alpha D_0/(1-\sigma^2)\right)^{-1/3}
\end{equation*}
we have
\[
B_{\widehat{T}} \ge \beta \sum_{\tau\in \QTS} B_\tau\cdot r_\tau,~~\text{for all}~~\widehat{T}\in \QT.
\]

We deduce an exponential rate of growth for $B_t$ using the following lemma, inspired by the analysis in~\cite{carmon2020acceleration} (and proved in the next subsection). 
\begin{restatable}{lemma}{restateLemGrowthRateBall}\label{lem:growth-rate-ball}
	Let $B_{1},...,B_{k}\in\R_{>0}$ be
	non-decreasing and let $r_{1},...,r_{k}\in\R_{\ge0}$ and $R_{i}\defeq\sum_{j\in[i]}r_{j}$
	for $i\in[k]$. Further, suppose that for some $\beta>0$, and all $i\in[k]$ it is the case that $B_{i}\geq\beta\cdot\sum_{j\in[i]}B_{j}\cdot r_{j}$.
	Then $B_{i}\geq\exp(\beta R_{i}-1)B_{1}$ for
	all $i\in[k]$.
\end{restatable}

Applying the lemma and substituting back the definition of $B_t$, we obtain,
\begin{align*}
A_T^{1/3} & \ge \exp\left(\beta \sum_{t\in \QT}r_t-1\right)A_1^{1/3} = \exp\left(\beta\cdot\frac{T-1}{2}-1\right)A_1^{1/3},
\end{align*}
where we let $\beta\defeq \min\left(1,\frac{1}{3}\ln\alpha\right)\cdot \left(\frac{2^5c^2\alpha D_0}{(1-\sigma^2)}\right)^{-\frac{1}{3}}.$

Thus, for 
\[T = \Omega\left(\frac{\log \frac{\norm{x_0-x_\star}^2}{\epsilon A_1}}{\beta}\right) = \Omega\left(\frac{\alpha^{\frac{1}{3}}}{\min\left(1,\frac{1}{3}\ln \alpha\right)}\cdot\prn*{c \norm{x_0 - \xopt}}^{\frac{2}{3}}\log \frac{\norm{x_0 - \xopt}^2}{\epsilon A_1}\right),\] 
we have $A_T\ge \frac{1}{2}\norm{x_0-x_\star}^2/\epsilon$, and consequently
\begin{align*}
f(x_T)-f(x_\star) = E_T \le \frac{\frac{1}{2}\norm{x_0-x_\star}^2}{A_T} \le \epsilon,	
\end{align*}
which proves the case for $s=\infty$.

\paragraph{The $s=1$ case.} This case corresponds to the standard analysis of Nesterov acceleration. The $(1,c)$-movement bound guarantees that $\lambda_t \le c$ for all $t$. Recalling that $\blambda_t \le \lambda_t$ for all $t\in\QT$, the bound~\eqref{eq:increase-A}, yields
\begin{equation*}
	4\sqrt{\alpha A_T} \ge \sum_{t\in \QT} \sqrt{\frac{\alpha^{r_t-1}}{c}}
	\ge \frac{\min\crl*{1,\half\ln \alpha}}{\sqrt{c}}  \sum_{t\in \QT}  r_t 
	=  \frac{\min\crl*{1, \half\ln \alpha}}{\sqrt{c}} \cdot \frac{T-1}{2},
\end{equation*}
where the final bound uses \Cref{lem:increase}. Consequently, the error bound~\eqref{eq:main-potential-implied} we have
\begin{equation*}
	E_T \le \frac{D_0}{A_T} = O\prn*{
		\frac{\alpha c D_0}{ \min\crl*{1, \ln^2 \alpha} T^2 }
	},
\end{equation*}
yielding the claimed result for $s=1$.

\subsection{Helper lemmas}
\label{sec:main_thm:helper}

\restateLemExpToLinear*
\begin{proof}
	Define the difference function $f(x)\defeq a^{x-1}-\min\left(1,\ln a\right)\cdot x$. We note that
	clearly $f(1)\geq0$ and the first-order derivative $f'(x)=(\ln a)\cdot a^{x-1}-\min\left(1,\ln a\right)\geq0$
	for all $x\geq1$. Consequently, by the integral formula $f(x) = \int_1^x f'(z) dz$ we have that $f(x) \geq 0$ for all $x \geq 1$. 	
\end{proof}

\restateLemGrowthRateHigh*

\begin{proof}
	Without loss of generality we take $\beta=1$, since otherwise we may redefine $r_i$ to be $\beta r_i$. Furthermore, we assume $r_1>0$ as otherwise we can divide into the following two cases:
	\begin{enumerate}
		\item if all $r_i=0$ the desired inequality naively holds;
		\item if there exists some $i_0$ such that $r_{i_0}>0$ and $r_{i}=0$ for all  $i<i_0$, then it suffices to consider the sequence starting from $i_0$. 
	\end{enumerate}
	
First, for $i=1$, note
	that $B_{1}^{m}\geq B_{1}\cdot r_{1}$ and consequently
	$B_{1}\geq(r_{1})^{1/(m-1)}\geq(\frac{m-1}{m}\cdot r_{1})^{1/(m-1)}$
	as desired. 
	
Next, for any $j>1$, we have
	\begin{align*}
	\sum_{j'\in[j+1]}B_{j'} r_{j'} -\sum_{j'\in[j]}B_{j'} r_{j'}= B_{j+1} r_{j+1} \ge \left(\sum_{j'\in[j+1]}B_{j'} r_{j'}\right)^{1/m}r_{j+1},
	\end{align*}
and consequently
\begin{align*}
r_{j+1} &\le \frac{\sum_{j'\in[j+1]}B_{j'} r_{j'}-\sum_{j'\in[j]}B_{j'} r_{j'}}{\left(\sum_{j'\in[j+1]}B_{j'}r_{j'}\right)^{1/m}}	 \le \int_{\sum\limits_{j'\in[j]}B_{j'} r_{j'}}^{\sum\limits_{j'\in[j+1]}B_{j'} r_{j'}}\frac{1}{t^{1/m}}dt\\
& = \frac{m}{m-1}\left(\left(\sum_{j'\in[j+1]}B_{j'} r_{j'}\right)^{\frac{m-1}{m}}-\left(\sum_{j'\in[j]}B_{j'} r_{j'}\right)^{\frac{m-1}{m}}\right).
\end{align*}

Summing the above inequality for all $j\in[i]$, noting that for $i=1$ we have $r_1\le \frac{m}{m-1}\left(B_1 r_1\right)^{\frac{m-1}{m}}$, and rearranging terms yields
\[
\left(\sum_{j\in[i+1]}B_{j} r_{j}\right)^{\frac{m-1}{m}}\ge \frac{m-1}{m}\sum_{j\in[i+1]}r_{j}.
\]
Combining this with the condition $B_{i+1}^m\ge \sum_{j\in[i+1]}B_j r_j$ concludes the proof.
\end{proof}

\restateLemGrowthRateBall*

\begin{proof}
	Let $k_{0}$ denote the largest element of $[k]$ for which $R_{k_{0}}<\beta^{-1}$
	and let $k_{0}=0$ if there is no such element. Note that $\exp(\beta (R_{k_{0}}-\beta^{-1})\leq1$.
	Since $B_{k}$ increases monotonically in $k$ this implies that $B_{i}\geq\exp\left(\beta(R_{i}-\beta^{-1})\right)B_{1}$
	for all $i\in[k_{0}]$.

	For $i =k_0+1$, we have 
	\begin{align*}
	\sum_{j'\in[i]}\beta B_{j'} r_{j'} -B_{1} \ge \sum_{j'\in[i]}\beta B_{j'} r_{j'} -\sum_{j'\in[i-1]}\beta B_{j'} r_{j'}= \beta\cdot B_{i} r_{i} & \ge \left(\sum_{j'\in[i]}\beta B_{j'} r_{j'}\right)\beta r_i,
	\end{align*}
	where the first inequality is due to the definition of $k_0$ and that $B_i$ is non-increasing.
	Consequently, 
\begin{align*}
\beta\cdot r_{k_0+1}\le \frac{\sum_{j'\in[k_0+1]}\beta\cdot B_{j'} r_{j'}-B_{1} }{\left(\beta\cdot\sum_{j'\in[k_0+1]}B_{j'}r_{j'}\right)}	& \le \int_{ B_{1} }^{\sum\limits_{j'\in[k_0+1]} \beta\cdot B_{j'} r_{j'}}\frac{1}{t}~dt\\
& = \log\left(\frac{\sum\limits_{j'\in[k_0+1]}\beta\cdot B_{j'} r_{j'}}{B_{1}}\right).
\end{align*}

	For any $i+1$ such that $2\le i+1\le k$ and $j\in[i]$, we have 
\begin{align*}
	\sum_{j'\in[j+1]}\beta B_{j'} r_{j'} -\sum_{j'\in[j]}\beta B_{j'} r_{j'}= \beta\cdot B_{j+1} r_{j+1} \ge \left(\sum_{j'\in[j+1]}\beta B_{j'} r_{j'}\right)\beta r_{j+1},
	\end{align*}
and consequently
\begin{align*}
\beta\cdot r_{j+1}\le \frac{\sum_{j'\in[j+1]}\beta\cdot B_{j'} r_{j'}-\sum_{j'\in[j]}\beta\cdot B_{j'} r_{j'}}{\left(\beta\cdot\sum_{j'\in[j+1]}B_{j'}r_{j'}\right)}	& \le \int_{\sum\limits_{j'\in[j]}\beta\cdot B_{j'} r_{j'}}^{\sum\limits_{j'\in[j+1]}\beta\cdot B_{j'} r_{j'}}\frac{1}{t}~dt\\
& = \log\left(\frac{\beta\cdot\sum\limits_{j'\in[j+1]}B_{j'} r_{j'}}{\beta\cdot\sum\limits_{j'\in[j]}B_{j'} r_{j'}}\right).
\end{align*}

Summing up above inequalities yields that for any $i \in [k_0 +1, k]$ it is the case that
\begin{align*}
\beta \sum_{j = k_0 +1}^{i} r_j \leq
\log \left(\frac{1}{B_1} \sum_{j \in [i + 1]} B_i r_i \right)
 \leq \log \left(\frac{B_i}{B_1} \right)
\end{align*}
Since $\beta \sum_{j = k_0 +1}^{i} r_j = \beta (R_i -  R_{k_0}) \geq \beta(R_i - \beta^{-1})$, where we define $R_{0} \defeq 0$, we obtain that $B_{i}\geq\exp\left(\beta R_{i}-1\right)B_{1}$ holds for all $i > k_0$, and hence for all $i\ge 1$.
\end{proof}
\section{Generalized oracle notions}\label{app:gen-framework}

In this section, we consider a setting where the convex function $f$ may be non-differentiable (i.e., $\grad f$ might not exist everywhere) and the problem may be constrained (i.e., the convex closed domain $\xset$ may be different from $\R^d$). We consider two types of oracles: a slight generalization of the MS oracle for the non-differentiable and/or constrained setting, and a fairly different notion of a ``stochastic proximal oracle'' similar to the one considered in~\cite{asi2021stochastic}.
We also provide a slight variation of \Cref{alg:optms} that makes use of these oracles and also allows more flexibility in choosing some of the iterates, and prove convergence rate bounds for this algorithm combined with either oracle.

To generalize \Cref{def:ms-condition} of a MS oracle, we consider mappings that return, in addition to $x\in\xset$ and $\lambda>0$, a vector $g\in\R^d$ that replaces $\grad f(x)$. More specifically, recall the definition of the subdifferential of $f$ at $x\in\xset$:
\begin{equation*}
	\del f(x) \defeq \{ g \in \R^d \mid \inner{g}{x'-x} \le f(x') - f(x)~~\mbox{for all}~x'\in\xset \}.
\end{equation*}
An element of $\del f(x)$ is called a subgradient of $f$ at $x$. Note that when $f$ is differentiable at $x$ we have $\grad f(x) \in \del f(x)$. However, on the boundary of $\xset$ there may additional elements in the subdifferential even when $f$ is differentiable. Our  generalized MS oracle (that originally appeared in~\cite{MonteiroS13a}) returns $g$, a subgradient of $f$ at $x$, such that the MS condition holds with $g$ instead of $\grad f(x)$. 

\begin{definition}[Generalized MS oracle]\label{def:ms-oracle-gen}
	An oracle $\MSoracle: \xset \times \R_+ \to \xset \times \R^d \times \R_+$ is a \emph{$\sigma$-Generalized MS} oracle for function $f:\xset\to\R$ if for every $y\in\xset$  and $\blambda>0$, the points  $(x,g,\lambda)=\MSoracle(y;\blambda)$ satisfy the following:
	\begin{equation}\label{eq:ms-condition-gen}
		g \in \del f(x) ~~\mbox{and}~~ 
		\norm*{ x - \prn*{y-\tfrac{1}{\lambda}g} } \le \sigma \norm{x-y}.%
	\end{equation}
\end{definition}

We remark that considering subgradients instead of gradients is essential for handling constrained optimization even when $f$ is differentiable, because even exact proximal point do not necessarily satisfy the simple MS condition~\eqref{eq:ms-condition}. That is, letting $F_\lambda(x) = f(x) + \frac{\lambda}{2}\norm{x-y}^2$, the point $x_\lambda = \argmin_{x\in\xset} F_\lambda(x)$ does not necessarily satisfy $x_\lambda = y - \frac{1}{\lambda}\grad f(x_\lambda)$. Nevertheless, the first-order optimality conditions of characterizing $x_\lambda$ guarantee that  $\lambda(y-x_\lambda) \in \del f(x_\lambda)$. Therefore, exact proximal points are $0$-Generalized MS oracles. 

We now present a different oracle, with a probabilistic approximation condition that relates directly to the exact proximal point $x_\lambda$. The advantage of approximation conditions of this kind is that they can be efficiently satisfied using stochastic first-order methods in certain non-smooth problems where certifying~\eqref{eq:ms-condition-gen} is hard~\cite{asi2021stochastic}. 

\begin{definition}[Stochastic proximal oracle]\label{def:stoch-prox-oracle}
	A (randomized) oracle $\MSoracle: \xset \times \R_+ \to \xset \times \R^d \times \R_+$ is a \emph{$\sigma$-stochastic proximal oracle}  for function $f:\R^d\to\R$ if for every $y\in\xset$  and $\blambda>0$, the points $(x,g,\lambda)=\MSoracle(y;\blambda)$ satisfy the following:
	\begin{equation}\label{eq:stoch-prox-condition}
		\E F_\lambda(x) \le \min_{x'\in\xset} F_\lambda(x') + \frac{\lambda\sigma^2}{4}\E\norm{x-y}^2~~,~~\E g = g_\lambda~~\mbox{and}~~\Var(g) \le \frac{\sigma^2}{2}\E\norm{x-y}^2,
	\end{equation}
	where $F_\lambda(x') \defeq f(x') + \frac{\lambda}{2}\norm{x'-y}^2$,  $x_\lambda = \argmin_{x'\in\xset} F_\lambda(x')$ and $g_\lambda = \lambda(y-x_\lambda)$, and all expectations are conditional on $y,\blambda$ and $\lambda$. 
\end{definition}

Three remarks are in order. First, note that exact proximal points are also $0$-stochastic proximal oracles. Second, the condition $\E g = g_\lambda = \lambda(y-x_\lambda)$ implies that if the stochastic proximal oracle outputs a deterministic $g$ then it also computes $x_\lambda$ exactly. Third, a $\sigma$-Generalized MS oracle output $x,g,\lambda$ satisfies $g + \lambda(x-y) \in \del F_\lambda(x)$ and therefore, by $\lambda$-strong convexity of $F_\lambda$, 
\begin{equation*}
	F_\lambda(x) -\min_{x'\in\xset} F_\lambda(x') \le \frac{1}{2\lambda} \norm{g + \lambda(x-y)}^2 \le \frac{\lambda\sigma^2}{2}\norm{x-y}^2.
\end{equation*}
Therefore, up to a replacing $\sigma$ with $\sigma/\sqrt{2}$, the generalized MS-condition~\eqref{eq:ms-condition-gen} implies the first part of~\eqref{eq:stoch-prox-condition}. Moreover, when $g$ is deterministic its variance is zero, giving the third part of the condition. The second part of the condition, however, is not directly implied by~\eqref{eq:ms-condition-gen}. Nevertheless, given a procedure that for any $\delta \ge 0$ outputs a point $x^\delta$ such that $F_\lambda(x^\delta) -\min_{x'\in\xset} F_\lambda(x') \le \frac{\lambda\delta^2}{2}$ (e.g., an $\sigma$-MS oracle with appropriate value of $\sigma$), it is possible to generically obtain an estimator $\hat{x}_\lambda$ that is unbiased for  $x_\lambda$ via multilevel Monte Carlo~\cite[see][]{asi2021stochastic,carmon2022distributionally}, thereby obtaining $g=\lambda(y-\hat{x}_\lambda)$ satisfying the second part of~\eqref{eq:stoch-prox-condition} as well as the variance bound in the third part of~\eqref{eq:stoch-prox-condition}.

\Cref{alg:optms-gen} uses either the generalized MS oracle (\Cref{def:ms-oracle-gen}) or the stochastic proximal oracle (\Cref{def:stoch-prox-oracle}). The differences between it and \Cref{alg:optms} are highlighted in blue. There are two differences addition to the obvious one in the oracle interface (which now returns an additional vector $g_{t+1}$). First, we use the vector $g_{t+1}$ to update $v_t$ using a projected mirror descent step (here $\proj_\xset$ denotes the Euclidean projection unto $\xset$). Second, we allow the algorithm to replace the point $\bx_t$ output from the oracle with any other point $\tx_t$ that has a lower function value. We note that such option exists also in the original proposal by \citet{MonteiroS13a} and is independent of the other generalizations studied in this section.

\begin{algorithm}[t]
	\setstretch{1.2}
	\DontPrintSemicolon
	\caption{Generalized Optimal MS Acceleration}\label{alg:optms-gen}
	\KwInput{Initial $x_0$, generalized oracle $\oracle$}
	\KwParameters{Initial $\blambdainit$, multiplicative adjustment factor $\alpha > 1$}
	
	Set $v_0 = x_0$, $A_0=0$\;
	\phantom{}{\color{blue}$\bx_1, g_1,\lambda_1 = \oracle(x_0;\blambdainit)$}~~,~~$\blambda_1 = \lambda_1$\;
	\For{$t=0,1,\ldots,$}{
		$\ba_{t+1} = \frac{1}{2\blambda_{t+1}}\prn*{1 + \sqrt{1 + 4\blambda_{t+1} A_t}}$\;
		$\bA_{t+1} = A_t + \ba_{t+1}$\;
		$y_{t} = \frac{A_{t}}{\bA_{t+1}} x_t + \frac{\ba_{t+1}}{\bA_{t+1}} v_t$\;
		\lIf{$t>0$}{
			{\color{blue}$\bx_{t+1}, g_{t+1}, \lambda_{t+1} = \oracle(y_t; \blambda_{t+1})$}}
		{\color{blue}Let $\tx_{t+1}\in\xset$ satisfy $f(\tx_{t+1})  \le f(\bx_{t+1})$}\;
		\If{$\lambda_{t+1} \le \blambda_{t+1}$}{
			$a_{t+1}=\ba_{t+1}$, ~$A_{t+1} = A_t + a_t$\;
			$x_{t+1} = \tx_{t+1}$\;
			{$\blambda_{t+2} = \frac{1}{\alpha}\blambda_{t+1}$}
		}
		\Else{
			$\gamma_{t+1} =\frac{\blambda_{t+1}}{\lambda_{t+1}}$\label{line:gen-damping-start}\;
			$a_{t+1}=\gamma_{t+1}\ba_{t+1}$, ~$A_{t+1} = A_t + a_t$\;
			$x_{t+1} = 
			\frac{(1-\gamma_{t+1})A_{t}}{A_{t+1}} x_t
			+ \frac{\gamma_{t+1} \bA_{t+1}}{A_{t+1}} \tx_{t+1}$\;
			$\blambda_{t+2} = \alpha\blambda_{t+1}$\label{line:gen-damping-end}
		}

		{\color{blue}$v_{t+1} = \argmin_{v\in\xset}\crl*{\inner{g_{t+1}}{v} + \frac{1}{2a_{t+1}}\norm{v-v_t}^2 } = \proj_\xset( v_t - a_{t+1} g_{t+1} )$}\;
	}
\end{algorithm}

\begin{theorem}\label{thm:main-gen}
	Let $f:\xset\to \R$ be convex and differentiable, let $\xset$ be closed and convex, and consider \Cref{alg:optms} with parameters $\alpha>1$, $\blambda>0$, and a $\sigma$-Generalized MS oracle (\Cref{def:ms-oracle-gen}) or $\sigma$-stochastic proximal oracle (\Cref{def:stoch-prox-oracle}) for $f$ with  $\sigma\in [0,0.99)$. Let $p\ge 1$ and $c>0$, and suppose that for all $t$ such that $\lambda_{t}>\blambda_{t}$ or $t=1$, the iterates $(\bx_{t}, y_{t-1}, \lambda_{t})$ satisfy a $(s,c)$-movement bound (\Cref{def:movement-bound}) with probability 1. There exist
$C_{\alpha,s} =  O\prn*{ \frac{s}{\min\{s,\ln \alpha\}}\alpha^{\frac{s+1}{3s+1}}}$ and $K_{\alpha} = O\prn*{\frac{1}{\ln \alpha} \alpha^{1/3}}$ such that the following holds.
	Let $\xopt\in\xset$; if $\oracle$ is a stochastic proximal oracle then let $\xopt$ be a minimizer of $f$. For any $\epsilon>0$, when
	\begin{equation*}
		T \ge \begin{cases}
			C_{\alpha,s} \prn*{\frac{c^{s} \norm{x_0 - \xopt}^{s+1}}{\epsilon}}^{\frac{2}{3s+1}}  & s< \infty \\
			K_\alpha \prn*{c \norm{x_0 - \xopt}}^{\frac{2}{3}}\log \frac{\lambda_1\norm{x_0 - \xopt}^2}{\epsilon} & s=\infty, 
		\end{cases}
	\end{equation*}
	we have $f(x_T) - f(\xopt) \le \epsilon$ with probability at least $2/3$. 
\end{theorem}

Before providing the proof, we make three more remarks. First, the success probability is relevant only for the stochastic proximal oracle; the bound for the generalized MS oracle holds with probability 1. Second, the need to assume that $\xopt$ is a minimizer of $f$ is also due to a technical issue with the analysis of the stochastic proximal oracle, pointed out in the proof below. Finally, we note that for stochastic proximal oracle we may require the movement bounds to hold on either $(\tx_t, y_{t-1},\lambda_t)$ as stated in the theorem, or on $(x^\star_{t}, y_{t-1}, \lambda_t)$, where $x^\star_{t}$ is the exact $\lambda_t$ proximal point of $y_{t-1}$. 

We recommend reading the proof of \Cref{thm:main} (in \Cref{apdx:framework}) before reading the following proof. 

\begin{proof}[Proof of \Cref{thm:main-gen}]
	The proof consists of showing that a version of \Cref{prop:main-potential} holds under the conditions of \Cref{thm:main-gen}; from there on the arguments on the analysis of the growth rate of $A_T$ is identical. For generalized MS oracles, the proof of \Cref{prop:main-potential} goes through unchanged, except for $\bx_t$ replacing $\tx_t$, the subgradient $g_t\in\del f(\bx_{t})$ replacing $\grad f(\tx_t)$, and using $f(\tx_t)\le f(\bx_t)$ to show that $A_{t+1} E_{t+1} \le (1-\gamma_{t+1})A_t E_t + \gamma_{t+1}\bA_{t+1} [f(\bx_{t+1})-f(\xopt)]$. 
	
	Next, we consider stochastic proximal oracles and adapt~\cite[Lemma 5]{asi2021stochastic}, which considers a very similar oracle, to account for our momentum damping scheme (lines \ref{line:gen-damping-start} to \ref{line:gen-damping-end}). 
	Beginning with some notation, we define the filtration
	\begin{equation*}
		\filt \defeq \sigma(\lambda_1, \bx_1,g_1, \ldots, \lambda_t,\bx_t,g_t, \lambda_{t+1})
	\end{equation*}
	so that $a_{t+1},\ba_{t+1},A_{t+1}\in\filt$. 
	In addition, we let 
	\begin{equation*}
		\hx_t = \argmin_{x\in\xset}\crl*{f(x) + \frac{\lambda_t}{2}\norm{x-y_{t-1}}^2}
	\end{equation*}
	and note that $\hx_{t+1}\in\filt$ and moreover that
	\begin{equation*}
		\hg_{t+1} \defeq \lambda_{t+1}(y_t - \hx_{t+1}) = 
		\Ex*{g_{t+1} | \filt}
	\end{equation*}
	by the second part of~\eqref{eq:stoch-prox-condition}. We also define  $E_t\defeq f(x_t)-f(\xopt)$, $D_t \defeq \frac{1}{2}\norm{v_t-\xopt}^2$, and $M_{t+1}\defeq \frac{1}{2}\norm{\bx_{t+1}-y_t}^2$ as in \Cref{prop:main-potential} (except with $\bx_t$ instead of $\tx_t$). 
	
	The update formula for $v_t$ gives us
	\begin{align*}
		D_{t+1} & 
		 = \frac{1}{2}\norm*{\proj_\xset\left(v_t-a_{t+1} g_{t+1}\right)-\xopt}^2 \nonumber \\&
		\le\frac{1}{2}\norm{(v_{t}-a_{t+1} g_{t+1})-\xopt}^{2}
		=D_{t}+a_{t+1}\inner{g_{t+1}}{\xopt-v_{t}}+\frac{a_{t+1}^{2}}{2} \norm{g_{t+1}}^{2}.
	\end{align*}
	Rearranging and taking expectation, we have
	\begin{flalign}
		a_{t+1}\inner{\hg_{t+1}}{v_{t}-\xopt} &= \Ex*{ a_{t+1}\inner{g_{t+1}}{v_{t}-\xopt} | \filt } \nonumber\\& \le D_t - \Ex*{D_{t+1} | \filt} + \frac{a_{t+1}^2}{2}\Ex*{ \norm{g_{t+1}}^2 | \filt}.\label{eq:stoch-mirror-descent}
	\end{flalign}
	Moreover, by the second and third parts of~\eqref{eq:stoch-prox-condition},
	\begin{equation}\label{eq:stoch-second-moment}
		\Ex*{ \norm{g_{t+1}}^2 | \filt} = \norm{\Ex*{g_{t+1} | \filt}}^2 + \Var\brk*{ g_{t+1} \mid \filt} \le
		\norm{\hg_{t+1}}^2 + \sigma^2\lambda_{t+1}^2\Ex*{M_{t+1}| \filt}.
	\end{equation}

	By the definition $y_t$ and $A'_{t+1} = A_t+a'_{t+1}$, we have
	\[
	\ba_{t+1}v_{t}=\bA_{t+1}y_{t}-A_{t}x_{t}=\ba_{t+1}\hx_{t+1}+\bA_{t+1}(y_{t}-\hx_{t+1})-A_{t}(x_{t}-\hx_{t+1})\,.
	\]
	Therefore,
	\begin{equation*}
		\begin{aligned}
			& \ba_{t+1}\inner{\hg_{t+1}}{\xopt - v_{t}}\\
			& \hspace{1.5em} =
			\inner{\hg_{t+1}}{\ba_{t+1}(\xopt - \hx_{t+1})+\bA_{t+1}(\hx_{t+1}-y_{t})+A_{t}(x_{t}-\hx_{t+1})}
			\\
			&\hspace{1.5em} \stackrel{(i)}{\leq} \ba_{t+1}[f(\xopt)-f(\hx_{t+1})]+\bA_{t+1}
			\inner{\hg_{t+1}}{\hx_{t+1}-y_{t}}
			+A_{t}[f(x_{t})-f(\hx_{t+1})]\\
			&\hspace{1.5em} \stackrel{(ii)}{=} A_{t}E_{t}-\bA_{t+1}[f(\hx_{t+1})-f(\xopt)]-\frac{\bA_{t+1}}{\lambda_{t+1}}\norm{\hg_{t+1}}^2\,.
		\end{aligned}
	\end{equation*}
	where we used $(i)$ the fact that $\hg_{t+1}\in\del f(\hx_{t+1})$ and $(ii)$ that $A'_{t+1}=A_t+a'_{t+1}$ and $\hx_{t+1} - y_t = -\hg_{t+1}/\lambda_{t+1}$. To connect $f(\hx_{t+1})$ to $f(\bx_{t+1})$, we use the first part of~\eqref{eq:stoch-prox-condition}, which gives
	\begin{equation*}
		\Ex*{ f(\bx_{t+1}) + \lambda_{t+1}M_{t+1} | \filt } \le f(\hx_{t+1}) + \frac{\lambda_{t+1}}{2}\norm{\hx_{t+1}-y_t}^2 + \frac{\sigma^2}{2}\lambda_{t+1}\Ex*{ M_{t+1} | \filt }.
	\end{equation*}
	Substituting back and recalling that $\norm{\hx_{t+1}-y_t}=\norm{\hg_{t+1}}/\lambda_{t+1}$ and $f(\tx_{t+1})\le f(\bx_{t+1})$ gives
	\begin{flalign}
	\bA_{t+1}\Ex*{ f(\tx_{t+1}) - f(\xopt) | \filt } \le \hspace{-4cm}&
	\nonumber \\& A_t E_t + \ba_{t+1} \inner{\hg_{t+1}}{v_{t} - \xopt} - \prn*{1-\frac{\sigma^2}{2}}\bA_{t+1}\lambda_{t+1} \Ex*{M_{t+1}|\filt}  -\frac{\bA_{t+1}}{2\lambda_{t+1}}\norm{\hg_{t+1}}^2.\label{eq:stoch-error-upper-bound}
	\end{flalign}

	Let $\gamma_{t+1} \defeq \min\crl*{1, \frac{\blambda_{t+1}}{\lambda_{t+1}}}$ and note that this definition is consistent with $\gamma_{t+1}$ as defined in the algorithm and that, for any value of $\blambda_{t+1}/\lambda_{t+1}$ we have $a_{t+1} = \gamma_{t+1}\ba_{t+1}$, $A_{t+1} = (1-\gamma_{t+1})A_t + \gamma_{t+1}\bA_{t+1}$ and $x_{t+1} = \frac{(1-\gamma_{t+1})A_{t}}{A_{t+1}} x_t + \frac{\gamma_{t+1} \bA_{t+1}}{A_{t+1}} \tx_{t+1}$. Therefore, by convexity,
	\begin{equation*}
		f(x_{t+1}) \le \frac{(1-\gamma_{t+1})A_{t}}{A_{t+1}} f(x_{t}) + \frac{\gamma_{t+1} \bA_{t+1}}{A_{t+1}} f(\tx_{t+1}).
	\end{equation*}
	Subtracting $f(\xopt)$, multiplying by $A_{t+1}$ and taking expectation, we have
	\begin{flalign}
		&\Ex*{A_{t+1}E_{t+1} | \filt} \le (1-\gamma_{t+1})E_t + \gamma_{t+1} \bA_{t+1}\Ex*{ f(\tx_{t+1}) - f(\xopt) | \filt }
		\nonumber \\ & \hspace{8pt}\le
		A_t E_t + a_{t+1} \inner{\hg_{t+1}}{v_{t} - \xopt} - \prn*{1-\frac{\sigma^2}{2}}\bA_{t+1}\gamma_{t+1}\lambda_{t+1} \Ex*{M_{t+1}|\filt}  -\frac{\gamma_{t+1}\bA_{t+1}}{2\lambda_{t+1}}\norm{\hg_{t+1}}^2.\label{eq:stoch-almost-potential-bound}
	\end{flalign}
	where in the second inequality we substituted~\eqref{eq:stoch-error-upper-bound}. Note that
	\begin{equation*}
		a_{t+1}^2 = \gamma_{t+1}^2 (\ba_{t+1})^2 = \frac{\gamma_{t+1}^2\bA_{t+1}}{\blambda_{t+1}} \le \frac{\gamma_{t+1}\bA_{t+1}}{\lambda_{t+1}}.
	\end{equation*}  
	Substituting back into~\eqref{eq:stoch-mirror-descent} and combining with~\eqref{eq:stoch-second-moment} gives
	\begin{equation*}
		a_{t+1}\inner{\hg_{t+1}}{v_{t}-\xopt} \le D_t - \Ex*{D_{t+1} | \filt} + \frac{\sigma^2\bA_{t+1}\gamma_{t+1}\lambda_{t+1}}{2}\Ex*{M_{t+1} | \filt} + \frac{\gamma_{t+1}\bA_{t+1}}{2\lambda_{t+1}}\norm{\hg_{t+1}}^2.
	\end{equation*}
	Plugging the above bound on $a_{t+1}\inner{\hg_{t+1}}{v_{t}-\xopt}$ into~\eqref{eq:stoch-almost-potential-bound}, noting that $\gamma_{t+1}\lambda_{t+1} = \min\{\blambda_{t+1},\lambda_{t+1}\}$,  and rearranging, we obtain
	\begin{equation*}
		\Ex*{A_{t+1}E_{t+1} + D_{t+1} + (1-\sigma^2)\bA_{t+1} \min\{\blambda_{t+1},\lambda_{t+1}\} M_{t+1} | \filt} \le  A_{t}E_{t} + D_{t}.
	\end{equation*}
	Iterating this bound and noting that $A_t \le \bA_t$ for all $t$, we obtain
	\begin{equation*}
		\Ex*{A_T E_T + D_T + (1-\sigma^2)\sum_{t\in\SsupT} A_t \blambda_t M_t } \le A_0E_0 + D_0.
	\end{equation*}
	By our assumption that $\xopt$ is a minimizer of $f$, we have that $E_T$ is a nonegative random variable, and consequently the above display is a bound on the expectation of a nonegative random variable. (This is the reason we require $\xopt$ to be a minimizer of $f$). Therefore, by Markov's inequality, the event 
	\begin{equation*}
		A_T E_T + D_T + (1-\sigma^2)\sum_{t\in\SsupT} A_t \blambda_t M_t \le 3(A_0E_0 + D_0)
	\end{equation*}
	holds with probability at least 2/3, implying~\eqref{eq:potential-2}, except with $A_0 E_0 +D_0$ multiplied by a factor of $3$. Moreover, the growth bounds $\sqrt{A_T} \ge \frac{1}{2}\sum_{t\in\SdownT}1/\sqrt{\blambda_{t}}$ holds deterministically as a consequence of the update rule for $A_t$. These two facts suffice to establish the growth rate of $A_T$ precisely as we do in the proof of \Cref{thm:main}, thereby obtaining the same rate of convergence (up to a constant). 
\end{proof} %
\section{Proofs for \Cref{sec:solvers}}\label{apdx:solvers}

This section contains the analysis of our adaptive oracle implementations (\Cref{alg:adaptive-msn-step,alg:adaptive-msn-cg-step}).  We begin by quickly showing how an idealized regularized Newton step adaptively yields optimal movement bounds without need to know the degree or order of the Hessian \Holder continuity (\Cref{apdx:solvers-ideal}). Then, we prove movement bound and complexity  guarantees for our second-order and first-order adaptive oracle implementations in \Cref{apdx:solvers-amsn,apdx:solvers-amsn-cg}, respectively. Finally, we provide auxiliary results used throughout the preceding proofs (\Cref{apdx:solvers-aux}).

\subsection{A movement bound for the ideal Newton step}\label{apdx:solvers-ideal}

The following proposition shows how choosing the smallest $\lambda$ for which a $\lambda$-regularized Newton step satisfies the MS condition yields movement bounds adaptive to Hessian \Holder continuity.

\begin{proposition}
	For any $y\in\R^d$ and $\sigma\in(0,1)$, let $\lambda^\star$ be the smallest $\lambda$ for which the $\lambda$-regularized Newton step $x_\lambda=y-[\hess f(y) + \lambda^\star I]^{-1} \grad f(y)$ satisfies the MS condition $\norm{x_\lambda-(y-\frac{1}{\lambda}\grad f(x))} \le \sigma \norm{x_\lambda-y}$. If $\hess f$ is $(H,\nu)$-\Holder continuous for any $\nu\in[0,1]$, the triplet $(x_{\lambda^\star},y,\lambda^\star)$ satisfies a $\prn*{1+\nu, \prn[\big]{\frac{H}{2\sigma}}^{1/(1+\nu)} }$-movement bound.
\end{proposition}
\begin{proof}
	Let 
	\begin{equation*}
		\tx = \argmin_{x'\in\R^d} \crl*{ \tilde{f}_2(x';y) + \frac{H}{2(2+\nu)\sigma}\norm{x'-y}^{2+\nu}}
		~~\mbox{and}~~\tlambda=\frac{H}{2 \sigma}\norm{x-y}^{\nu}.
	\end{equation*}
	That is $(\tx, \tlambda) = \oracleP[2,\nu](y)$ with parameter $M=H/\sigma$. Note that (a) $\tx,y$ and $\tlambda$ satisfy the MS condition as explained in \Cref{ssec:solvers-prior}, and (b) $\tx = y -[\hess f(y) + \tlambda I]^{-1} \grad f(y)$. Therefore the minimal $\lambda^\star$ satisfying the MS condition must satisfy $\lambda^\star \le \tlambda$ and 
	\begin{equation*}
		\norm{x^\star - y} \overge{(\star)} \norm{\tx -y} = \prn*{\frac{2\sigma \tlambda}{H}}^{1/\nu} \ge \prn*{\frac{2\sigma \lambda^\star}{H}}^{1/\nu},
	\end{equation*}
	where $(\star)$ is due to auxiliary \Cref{lem:reg-newton-step-properties}. This yields the movement bound. 
\end{proof}

We remark that the same proof and movement bound hold for a slightly more relaxed notion of $\lambda^\star$, namely the smallest for which the $x_\lambda$ satisfies the MS condition for all $\lambda \ge \lambda^*$. This is the actual notion of $\lambda^\star$ that we approximate in the next subsections, where we find $\lambda$ such that $x_\lambda$ satisfies the MS condition, but $x_{\lambda/2}$ does not satisfy it (and therefore $\lambda \ge \lambda^\star /2$).

\subsection{Analysis of \Cref{alg:adaptive-msn-step}}\label{apdx:solvers-amsn}

\restateThmAmsn*

\begin{proof}
	First, note that, by construction, the algorithm outputs a value of $\lambda$ for which $\mscheck(\lambda;y,\sigma)$ evaluates to True, and is therefore a $\sigma$-MS oracle as per \Cref{def:ms-condition}. 
	
	Next, let us bound the total number of linear system solutions in the algorithm, noting it is equal to the number of calls to $\mscheck$. The algorithm solves 1 linear system in~\cref{line:amsn-first-check}, and then solves $k^\star + 1$ linear systems in either the while-loop in \cref{line:amsn-decrease-loop} or the while-loop in \cref{line:amsn-increase-loop}. The algorithm then arrives at the while-loop in \cref{line:amsn-bisection-loop} with two values $\lamvalid$ and $\laminvalid$ such that (a) $\mscheck(\lamvalid;y,\sigma)$ is True and $\mscheck(\laminvalid;y,\sigma)$ is False and (b) $\lamvalid / \laminvalid = 2^{2^{k^\star}}$ . The while-loop maintains the invariant (a) while transforming  $\lamvalid / \laminvalid \to \sqrt{\lamvalid / \laminvalid}$ at each iteration. After $j$ iterations of the while-loop, we have 
	\begin{equation*}
		\frac{\lamvalid}{\laminvalid} = \prn*{2^{2^{k^\star}}}^{2^{-j}} = {2^{2^{k^\star-j}}}.
	\end{equation*}
	Therefore, after precisely $k^\star$ iterations we obtain $\lamvalid = 2\laminvalid$ and the loop terminates. Hence, the overall number of linear system solutions is $1+(k^\star+1) + k^\star=2+2k^\star$, or just 1 in case the first $\mscheck$ is True and the algorithm is lazy.
	
	To bound $k^\star$ in terms of the input $\blambda$ and output $\lambda$, consider the values of $\lamvalid$ and $\laminvalid$  before entering the while-loop at \cref{line:amsn-bisection-loop}. First, note that $\laminvalid \le \lambda \le \lamvalid$. Second, assuming that $\lambda > \blambda$ (i.e., the first $\mscheck$ fails) we have
	\begin{equation*}
		\lamvalid = \blambda \prod_{j=0}^{k^\star} 2^{2^{j}} = \blambda \cdot 2^{\prn*{2^{k^\star+1}-1}}.
	\end{equation*}
	Combining these two facts, we have
	\begin{equation*}
		2^{\prn*{2^{k^\star+1}-1}}\frac{\blambda}{\lambda} \le \frac{\lamvalid}{\laminvalid} = 2^{2^{k^\star}}.
	\end{equation*}
	Rearranging this inequality yields $k^{\star} \le \log_2\prn*{1+\log_2\frac{\lambda}{\blambda}}$ as claimed. The bound for $\lambda < \blambda$ follows analogously.
	
	We now turn to showing the movement bound assuming $\hess f$ is $(H,\nu)$-\Holder for $\nu\in[0,1]$ within a ball of radius $2\norm{x-y}$ around $y$, where $x=y-(\hess f(y) + \lambda I)^{-1} \grad f(y)$ is the output of the algorithm. Let  $\lamminus=\lambda/2$ and $\xminus = y-(\hess f(y) + \lamminus I)^{-1} \grad f(y)$. Let $\tlambda$ and $\tx=y-(\hess f(y) + \tlambda I)^{-1} \grad f(y)$ satisfy $\tlambda = \frac{H}{2\sigma} \norm{\tx- y}^\nu$, i.e.,
	\begin{equation*}
		\tx = \argmin_{x\in\R^d}\crl*{\inner{\grad f(y)}{x-y} + \half\inner{x-y}{\hess f(y) (x-y)} + \frac{H}{2(2+\nu)\sigma}\norm{x-y}^{2+\nu}}.
	\end{equation*}
	It is immediate to see that for any $\lambda>\tilde{\lambda}$, $\mscheck (\lambda; y,\sigma)$ must evaluate to True and thus the while loop in ~\Cref{line:amsn-increase-loop} is guaranteed to terminate. 
	 
	We will now argue that $\lamminus < \tlambda$. First, note that $\lamminus$ is the last value of $\laminvalid$ before the while loop at \cref{line:amsn-bisection-loop} terminates. Therefore, $\mscheck(\lamminus; y,\sigma)$ must evaluate to False. Further, note that, by \Cref{lem:reg-newton-step-properties},
	\begin{equation*}
		\norm{\xminus - y} \le \frac{\lambda}{\lamminus}\norm{x-y} = 2\norm{x-y}.
	\end{equation*}
	Using the local Hessian assumption of  \Holder continuity along with $(\hess f(y) + \lamminus I)\xminus + \grad f(y)=0$, auxiliary  \Cref{lem:ms-via-second-order-smoothness} gives
	\begin{equation*}
		\norm*{\xminus - \prn*{y- \frac{1}{\lamminus}\grad f(\xminus)}} \le 
		\frac{H}{2\lamminus}\norm{\xminus-y}^{1+\nu} = \frac{\sigma\tlambda}{\lamminus\norm{\tx-y}^\nu}\norm{\xminus-y}^{1+\nu},
	\end{equation*}
	with the final equality using the definition of $\tlambda$. 
	Assuming by contradiction that $\lamminus \ge \tlambda$, we have that
	\begin{equation*}
		\norm*{\xminus - \prn*{y- \frac{1}{\lamminus}\grad f(\xminus)}} \le 
		\frac{\sigma\norm{\xminus-y}^{1+\nu}}{\norm{\tx-y}^\nu}\le \sigma\norm{\xminus-y},
	\end{equation*}
	where the final inequality used \Cref{lem:reg-newton-step-properties} combined with $\lamminus \ge \tlambda$ to deduce that $\norm{\xminus-y}\le \norm{\tx-y}$. This implies that $\mscheck(\lamminus;y,\sigma)$ is True, giving a contradiction. The $(1+\nu,(2H/\sigma)^{1/(1+\nu)})$-movement bound follows from
	\begin{equation*}
		\lambda = 2\lamminus <  2\tlambda = 
		\frac{H \norm{\tx-y}^\nu}{\sigma} \le \frac{2^{\nu}H \norm{x-y}^\nu}{\sigma^{1}},
	\end{equation*}
	with the final inequality using \Cref{lem:reg-newton-step-properties} combined with $\lambda / \tlambda \le 2$, which implies $\norm{\tx-y} \le 2\norm{x-y}$.
\end{proof}

We remark that the termination of the while loop in~\cref{line:amsn-first-check} (when $\lazyflag$ is False) is, strictly speaking, not guaranteed. For example, if the function $f$ is quadratic, $\mscheck$ will always evaluate to True. 
However, it is also straightforward to verify that as long as $\mscheck$ is True for a given $\lambda$, the corresponding regularized Newton step $x=y - (\nabla^2 f(y) + \lambda I)^{-1} \nabla f(y)$ has optimality gap bounded by $\lambda \norm{x-y}^2$. Therefore, since the loop in~\cref{line:amsn-first-check} decreases $\lambda$ at a double-exponential rate, if it fails to terminate after a small number of iterations then it means we have found an essentially optimal point. Put differently, if we seek an $\epsilon$ suboptimal point, we may stop the loop in~\cref{line:amsn-first-check} after $O(\log \log( \frac{\lambda_0’ R^2}{\epsilon}))$ iterations. We account for this possibility in the complexity bound below.

\restateCorrAmsn*

\begin{proof}
	Throughout the proof, we let $T$ denote the index of the first iteration of \Cref{alg:optms} for which $f(x_T) \le f(\xopt) + \epsilon$.
	The bound on Hessian evaluation complexity follows immediately from the validity of the MS-approximate proximal oracle and movement bounds guaranteed in \Cref{thm:amsn}, and the iteration bound given by \Cref{thm:main}, noting that each call to \Cref{alg:adaptive-msn-step} requires only one Hessian computation. 
	
	To bound the total number of linear system solutions, we consider separately $(i)$ the iteration where $\lambda_{t} = \blambda_{t}$, $(ii)$ the iterations $t\in [2,T-1]$ where $\lambda_{t} > \blambda_{t}$ (i.e, those in $\SsupT$), $(iii)$ the first iteration (note that $\lambda_{t+1} < \blambda_{t+1}$ can only happen in the first iteration since we are using a lazy oracle), and $(iv)$ the last iteration.
	
	Case $(i)$ is easy, because at iterations where $\lambda_{t+1} = \blambda_{t+1}$ \Cref{alg:adaptive-msn-step} requires at most 2 linear system solves by \Cref{thm:amsn}, and thereofore all such iterations combined require $N_{(i)}=O(T) = O\prn*{ \prn*{\frac{H \norm{x_0-\xopt}^{2+\nu}}{\epsilon}}^{2/(4+3\nu)}}$ linear system solves.
	
	To handle case $(ii)$, which requires the most work, we use \Cref{thm:amsn} to bound the total number of linear system solves contributed by these iterates by
	\begin{flalign*}
		N_{(ii)} & \overeq{(a)} O\prn*{\sum_{t \in \SsupT[T-1]} \log_2 \prn*{ 1+\log_2 \frac{\lambda_{t}}{\blambda_{t}}}}
		\overeq{(b)}  O\prn*{\sum_{t \in \SsupT[T-1]} \log_2 \prn*{ 1+\log_2 \frac{H \norm{\tx_{t}-y_{t-1}}^\nu}{\blambda_{t}}}}
		\\ &
		 \overeq{(c)} O\prn*{\sum_{t \in \SsupT[T-1]} \prn*{H\frac{\norm{\tx_{t} - y_{t-1}}^\nu }{\blambda_{t}}}^{2/(4+3\nu)} }
		\overeq{(d)} O\prn[\Big]{ (H \norm{x_0-\xopt}^\nu A_{T-1} )^{2/(4+3\nu)}},
	\end{flalign*}
	which follows from $(a)$ the complexity bound in \Cref{thm:amsn}; $(b)$ the fact that a $(1+\nu,O(H^{1/(1+\nu)}))$-movement bound holds for every $t\in\SsupT$, meaning that $\lambda = O(H)\norm{\tx_{t}-y_t}^\nu$; $(c)$ by the inequalities 
	\[
	\log_2 \left( 1 + \log_2 z \right) \leq \log_2 z = c \log_2 z^{1/c} = O(z^{1/c})
	\]
	for any $z \geq 1$ and fixed $c \geq 0$; and $(d)$  \Cref{lem:yet-another-reverse-holder}, using the assumption that $f(x_{T-1}) > f(\xopt)$.  Moreover, note that 
	\begin{equation*}
		\epsilon < f(x_{T-1})-f(\xopt) \le  \frac{ \norm{x_0 - \xopt}^2 }{2 A_{T-1}}
	\end{equation*} 
	by eq.~\eqref{eq:main-potential-implied} in \Cref{prop:main-potential}, which implies $A_{T-1} = O(\norm{x_0 - \xopt}^2 / \epsilon)$. Substituting back into the bound on $N_{(ii)}$, we obtain
	\begin{equation*}
		N_{(ii)} = O\prn[\Big]{ (H \norm{x_0-\xopt}^\nu A_{T-1} )^{2/{(4+3\nu)}}} = O\prn*{ \prn*{\frac{H \norm{x_0-\xopt}^{2+\nu}}{\epsilon}}^{2/(4+3\nu)} }.
	\end{equation*}
	Therefore, perhaps surprisingly, the worst-case double-logarithmic per-iteration linear system complexity amortizes to a constant.  
	
	To handle the last two edge cases, let $z_\star$ be the minimizer of $f$ closest to $x_0$, such that $\norm{x_0-z_\star}=R$. 
	In case $(iii)$, i.e., the number of linear system solves in the first iteration. We consider separately the cases $\lambda_1 \ge \blambda_0$ and $\lambda_1 < \blambda_0$. In the former case, the movement bound guaranteed at the first iteration yields (noting that $y_0=x_0$)
	\begin{equation*}
		\frac{\lambda_1}{\blambdainit} = O\prn*{\frac{H \norm{x_1-x_0}^\nu}{\blambdainit}} = O\prn*{\frac{H R^\nu}{\blambdainit}},
	\end{equation*}
	where the last transition follows from \Cref{lem:ms-sc-bounds} and the assumption $f(x_1) \ge f(z_\star)$. In the latter case ($\lambda_1 < \blambda_0$), \Cref{lem:ms-sc-bounds} and $f(x_1) - f(z_\star) \ge f(x_1)-f(\xopt) \ge \epsilon$ gives
	\begin{equation*}
		\frac{\blambdainit}{\lambda_1} = O\prn*{\frac{\blambdainit R^2}{f(x_1)-f(z_\star)}} = O\prn*{\frac{\blambdainit R^2}{\epsilon}}.
	\end{equation*}
	Therefore, the number of linear system solutions at the first iteration is \[O\prn*{\log\abs*{\log\frac{\lambda_1}{\blambdainit}}}=O\prn*{\log\log\max\crl*{\frac{H R^\nu}{\blambdainit}, \frac{\blambdainit R^2}{\epsilon} }}.\]

	Finally, we consider $(iv)$ the last iteration $t=T$. If $T\notin \SsupT$ then there is nothing to consider, since it only contributes a single linear system solutions. If $T\in \SsupT$, however, we cannot treat it as in case $(ii)$, since we are not guaranteed that $f(x_T) \ge f(\xopt)$. However, we \emph{are} guaranteed that $f(x_T) \ge f(z_\star)$. Therefore, \Cref{lem:yet-another-reverse-holder} allows us to conclude that
	\begin{flalign*}
		\frac{\lambda_T}{\blambda_T} = O\prn*{\frac{H\norm{\tx_T-y_{T-1}}^\nu}{\blambda_{T-1}}} & = 
		O\prn*{\brk*{\sum_{t \in \SsupT} \prn*{H\frac{\norm{\tx_{t} - y_{t-1}}^\nu }{\blambda_{t}}}^{\frac{2}{4+3\nu}}}^{\frac{4+3\nu}{2}} }
		\\ & = O\prn[\Big]{H R^\nu A_{T} }.
	\end{flalign*}
	Noting that $A_T = O(A_{T-1})$ (see \Cref{lem:A-mult-stability}) and $A_{T-1} = O(R^2/\epsilon)$ (since $\epsilon < f(x_{T-1}) - f(\xopt) \le f(x_{T-1}) - f(z_\star) \le \frac{R^2}{2A_{T-1}}$) we conclude that $\frac{\lambda_T}{\blambda_T} = O\prn*{\frac{HR^{2+\nu}}{\epsilon}}$. Since $\frac{HR^{2+\nu}}{\epsilon} = {\frac{H R^\nu}{\blambdainit} \cdot \frac{\blambdainit R^2}{\epsilon}} \le \max\crl*{\frac{H R^\nu}{\blambdainit}, \frac{\blambdainit R^2}{\epsilon} }^2$, we conclude that the $O\prn*{\log\log \frac{\lambda_T}{\blambda_T}}$ contribution of case $(iv)$ to the total number of linear systems is no greater than our bound for case $(iii)$. 
\end{proof}

We remark that departing from a normal bisection or doubling scheme we have used a ``double-logarithmic scale'' in the while loops starting at~\Cref{line:amsn-decrease-loop} and~\Cref{line:amsn-increase-loop} in~\Cref{alg:adaptive-msn-step}. As shown in the proof above, this doesn't affect the complexity bounds shown for case $(ii)$, but gives better complexity bounds in the analysis of case $(iii)$ and $(iv)$ . Eventually this allows us to only have an additive double-logarithmic term in the final complexity of linear system solves as stated in~\Cref{coro:adaptive-msn-step}.

\subsection{Analysis of \Cref{alg:adaptive-msn-cg-step}}\label{apdx:solvers-amsn-cg}

\restateThmAmsnFO*

\begin{proof}

	Clearly, \Cref{alg:adaptive-msn-cg-step} can only return a pair $x,\lambda$ satisfying \Cref{def:ms-condition}. (At this point, we are not guaranteed that the algorithm ever returns. However, below we prove that the check in \Cref{line:newton-cr-ms-check} succeeds for finite $\lambda$ as long as the Hessian is continuous).
	
	To analyze the complexity of the algorithm, we begin by noting that it approximates a Newton step with MinRes at most $\abs*{\log_2 \frac{\lambda}{\blambda}}+1$ times: when $\lambda \ge \blambda$ we approximate Newton steps for values for regularization parameters of the form $2^k\blambda$ for $k=0,1,\ldots,\log_2 \frac{\lambda}{\blambda}$; when $\lambda < \blambda$ we instead consider regularization parameters of the form $2^{-k}\blambda$ for $k=0,1,\ldots,\log_2 \frac{\lambda}{\blambda}+1$. These considerations immediately yield our claimed $O\prn*{ \abs*{\log_2 \frac{\lambda}{\blambda}} }$ bound on the number of gradients evaluated by the algorithm. 
	
	To bound the total Hessian-vector product complexity, suppose that \Cref{alg:adaptive-msn-cg-step} attempts to approximate a Newton step with regularization parameter $\lambda_k$; we argue that the corresponding terminates in $O\prn[\big]{\sqrt{\frac{\opnorm{\hess f(y)} + \lambda_k}{\lambda_k \sigma}}}$ Hessian-vector products. Let $A=\hess f(y) + \lambda_k I$, $b=-\grad f (y)$ and $w^\star = A^{-1}b$. \Cref{lem:cr-properties} guarantees that $\norm{r_t} = O\prn*{\opnorm{A}\norm{w^\star} / t^2}$. Consequently, after $T_\lambda= O\prn[\big]{\sqrt{\frac{\opnorm{A}}{\lambda_k \sigma}}}$ steps we have $\norm{r_{T_\lambda}} \le \frac{\lambda_k \sigma}{4}\norm{w^\star}$. Since $h(x) = \half x^{\T} A x - b^{\T} x$ is $\lambda_k$-strongly-convex and $r_i = \grad h(w_i)$, we have $\lambda_k \norm{w_{T_\lambda} - w^\star} \le \norm{r_{T_\lambda}}$, and consequently $\norm{w_{T_\lambda} - w^\star} \le \frac{\sigma}{4} \norm{w^\star} \le \frac{1}{4}\norm{w^\star}$. Since $\norm{w^\star} - \norm{w_{T_\lambda}} \le \norm{w_{T_\lambda} - w^\star}$ by the triangle inequality, we conclude that $\norm{w^\star} \le \frac{4}{3} \norm{w_{T_\lambda}}$. Substituting back yields $\norm{r_{T_\lambda}} \le \frac{\lambda_k \sigma}{3}\norm{w_{T_\lambda}}$, and consequently the while-loop must terminate in  $T_\lambda = O\prn[\big]{\sqrt{\frac{\opnorm{\hess f(y)}+\lambda_k}{\lambda_k \sigma}}}$ steps, with each step corresponding to a single Hessian-vector product.
	
	Next, we argue that for very large $\lambda_k$ we do not need to compute any (new) Hessian-vector product, since the while-loop terminates in one step. More specifically that, $\lambda_k \ge \frac{4\opnorm{\hess f(y)}}{\sigma}$ the while-loop terminates after one step, i.e., $\norm{r_1} \le \frac{\lambda_k\sigma}{2}\norm{w_1}$. To see this, first observe that (since $w_1$ has the smallest residual among all vectors $w$ proportional to $b$), we have $\norm{r_1} \le \norm{A b / \lambda_k - b} = \frac{1}{\lambda_k}\norm{\hess f(y)b} \le \frac{\opnorm{\hess f(y)}}{\lambda_k} \norm{b} \le \frac{\sigma}{4} \norm{b}$. Moreover, since $w_1 = \frac{b^{\T} A b}{b^{\T} A^2 b} b$, we have $\norm{w_1} \ge \frac{1}{\lambda_k + \opnorm{\hess f(y)}}\norm{b} \ge \frac{4}{5\lambda_k} \norm{b}$. Consequently,  $\norm{r_1} \le \frac{5\lambda_k\sigma}{16}\norm{w_1}$, meaning that the while-loop terminates after the first iterate. Therefore, for $\lambda_k \ge \frac{4\opnorm{\hess f(y)}}{\sigma}$ no Hessian-vector product computations are necessary (since the ones before the while-loop can be computed once for all $k$).

	\newcommand{\lambdamin}{\lambda_{\mathrm{min}}}
	
	Let $\lambdamin$ be the smallest value of $\lambda_k$ encountered by the algorithm, and note that $\lambdamin \ge \min\{\lambda/2, \blambda\}$, and that $\lambda_k = 2^k \lambdamin$ for $k=0,\ldots, O\prn*{\abs*{\log\frac{\lambda}{\blambda}}}$. By the discussion above, we require Hessian-vector product computations only at the first $K = O(\log \frac{\opnorm{\hess f(y)}}{\sigma \lambdamin})$ iterations. Therefore the total number of Hessian-vector products is $O\prn*{\sum_{k=0}^K \sqrt{1 + \frac{\opnorm{\hess f(y)}}{\sigma 2^k \lambdamin}}} = O\prn*{\sqrt{1 + \frac{\opnorm{\hess f(y)}}{\sigma \min\{\lambda,\blambda\}}}}$, giving the claimed bound on Hessian-vector product count.

	Next, we assume that $f$ is has an $(H,\nu)$-\Holder Hessian and  argue that the algorithm's output $x,\lambda$ satisfies a movement bound (unless $\lazyflag$ is True and $\lambda = \blambda$).  To do so, we first establish an upper bound on the returned $\lambda$. 
	Let 
\begin{equation*}
		\tw = \argmin_{v}\crl*{v^{\T}\grad f(y) + \half v^{\T} \hess f(y) v + \frac{H}{(2+\nu)}\norm{v}^{(2+\nu)\sigma}}
\end{equation*}
	and note that $\tw = -(\hess f(y) + \tlambda I)^{-1} \grad f(y)$ for $\tlambda = \frac{H}{\sigma} \norm{\tw}^\nu$. Let us show that the MS condition check in \cref{line:newton-cr-ms-check} must succeed when for   $\lambda_k \ge \tlambda$. Let $w^\star = -(\hess f(y) + \lambda_k I)^{-1} \grad f(y)$ denote the exact regularized Newton step corresponding to $\lambda_k$, let $w_i$ be the output of the corresponding MinRes run, and let 
	\begin{equation*}
		r_i = (\hess f(y) + \lambda_k I)w_i + \grad f(y) 
	\end{equation*}
	be the corresponding residual. Note that, by \Cref{lem:ms-via-second-order-smoothness} we have
	\begin{equation}\label{eq:ms-condition-equivalent-form}
		\norm*{x - \prn*{y-\frac{1}{\lambda_k}\grad f(x)}} \le \frac{1}{\lambda_k} \prn*{ \norm{r_i} + \frac{H}{2}\norm{w_i}^{1+\nu} }.
	\end{equation}
	Moreover
	\begin{equation*}
		\frac{H}{2}\norm{w_i}^{1+\nu} = \frac{\sigma\tlambda}{2\norm{\tw}^\nu} \norm{w_i}^{1+\nu} \le \frac{\lambda_k \sigma}{2\norm{\tw}^\nu} \norm{w_i}^{1+\nu},
	\end{equation*}
	and
	\begin{equation*}
		\norm{w_i} \overle{(i)} \norm{w^\star} \overle{(ii)} \norm{\tw}
	\end{equation*}
	due to $(i)$ \Cref{lem:cr-properties} and $(ii)$ \Cref{lem:reg-newton-step-properties} and the fact that $\lambda_k \ge \tlambda$. 
	Therefore, for $\nu\in[0,1]$ we have $\frac{H}{2}\norm{w_i}^{1+\nu}  \le \frac{\lambda_k \sigma}{2} \norm{w_i}$, and $\norm{r_i}\le \frac{\lambda_k\sigma}{2}\norm{w_i}$ admits an identical bound by the MinRes termination condition. 
	Substituting back into~\eqref{eq:ms-condition-equivalent-form}, we conclude that
	\begin{equation*}
		\norm*{x - \prn*{y-\frac{1}{\lambda_k}\grad f(x)}} \le \frac{1}{\lambda_k} \prn*{ \norm{r_i} + \frac{H}{2}\norm{w_i}^{1+\nu}  } \le \sigma \norm{w_i} = \sigma \norm{x-y}
	\end{equation*}
	and consequently reaching $\lambda_k \ge \tlambda$ ensures that the MS condition holds. However, the algorithm returns a value of $\lambda$ such that for $\lambda_k = \lambda/2$ the MS condition check fails, meaning that $\lambda/2 < \tlambda$. 
	
	It remains to argue that when \Cref{alg:adaptive-msn-cg-step} returns with $\lambda \le 2\tlambda$, an appropriate movement bound holds. To that end, recall that (by $\lambda$-strong convexity of the quadratic subproblem)
	\begin{equation*}
		\norm{w^\star} - \norm{w_i} \le \norm{w_i - w^\star} \le \frac{1}{\lambda}\norm{r_i} \le \frac{\sigma}{2}\norm{w_i}
	\end{equation*}
	and consequently
	\begin{equation*}
		\norm{w_i} \ge \frac{2}{3}\norm{w^\star}.
	\end{equation*}
	Moreover, since $\lambda \le 2\tlambda $, we have
	\begin{equation*}
		\norm{w^\star} \overge{(i)} \frac{1}{2} \norm{\tw} \overeq{(ii)} \frac{1}{2}\left(\frac{\sigma\tlambda}{H}\right)^{\frac{1}{\nu}} \overge{(iii)} \frac{1}{2}\left(\frac{\sigma\lambda}{2H}\right)^{\frac{1}{\nu}},
	\end{equation*}
	due to $(i)$ \Cref{lem:reg-newton-step-properties} and $\lambda \le 2\tlambda$, $(ii)$ the definition of $\tlambda$ and $(iii)$ $\lambda \le 2\tlambda$ again. Combining the last two displays, we have
	\begin{equation*}
		\norm{x-y} = \norm{w_i} \ge \frac{1}{3}\left(\frac{\sigma\lambda}{2H}\right)^{\frac{1}{\nu}} \ge  \left(\frac{\lambda}{c^{1+\nu}}\right)^{1/\nu}
		~~\mbox{for}~~c=(6H/\sigma)^{\frac{1}{1+\nu}},
	\end{equation*}
	as required.
\end{proof}

\restateCorrAmsnFO*

\begin{proof}
	Throughout the proof, we let $T$ denote the index of the first iteration of \Cref{alg:optms} for which $f(x_T) \le f(\xopt) + \epsilon$; the claimed bound on total number of iterations is an immediate corollary of both~\Cref{thm:main} and~\Cref{thm:amsn-fo}. 
	
	We now bound the complexity of~\Cref{alg:adaptive-msn-cg-step} with an approach similar to the proof of~\Cref{coro:adaptive-msn-step}. 
	To do so, we categorize all iterations into the following two cases: $(i)$ $t>1$ and (since $\lazyflag$ is true) $\lambda_t \ge \blambda_t$ or $(ii)$  the first iteration $t=1$. Now using~\Cref{thm:amsn-fo}, we know that in case $(i)$ the number of Hessian-vector product and gradient evaluations  is bounded by 
	\begin{align*}
		N_{(i)}& \stackrel{}{=}O\prn*{\sum_{t=2}^{\T}\prn*{\sqrt{1 + \frac{\opnorm{\hess f(y_t)}}{\sigma \blambda_t}}+\log\frac{\lambda_t}{\blambda_t}}} \stackrel{(a)}{=}O\prn*{\sum_{t=2}^{\T}\prn*{\sqrt{1 + \frac{L}{\blambda_t}}+\log\frac{L}{\blambda_t}}}\\
	& =O\prn*{\sum_{t=2}^{\T}\prn*{1+\sqrt{\frac{L}{\blambda_t}}}}
	\stackrel{(b)}{=} O\prn*{T+\sqrt{L A_T}},
	\end{align*}
where we use $(a)$ that $\opnorm{\hess f(y_t)}\le L$ by the assumption of $L$-Lipschitz gradient, and that either $\lambda_t = \blambda_t$ or $\lambda_t = O(L)$ by the movement bound guaranteed from \Cref{thm:amsn-fo} (since $L$-Lipschitz gradient means $(L,0)$-Lipschitz Hessian), and $(b)$ the bound~\eqref{eq:increase-A-all} from~\Cref{lem:increase}. Note that $T=O\prn*{\sqrt{\frac{L\norm{x_0-x_\star}^2}{\epsilon}}}$ due to the iteration count bound for $(L,0)$-\Holder Hessian. Moreover, noting that $A_T = O(A_{T-1})$ by \Cref{lem:A-mult-stability} and that $A_{T-1} = O(\norm{x_0-\xopt}^2/\epsilon)$ as argued in the proof of \Cref{coro:adaptive-msn-step}, we have $\sqrt{L A_T}=O\prn*{\sqrt{\frac{L\norm{x_0-x_\star}^2}{\epsilon}}}$ as well. Therefore,
\begin{equation*}
	N_{(i)} = O\prn*{\sqrt{\frac{L\norm{x_0-x_\star}^2}{\epsilon}}}.
\end{equation*}

For case $(ii)$, \Cref{thm:amsn-fo} gives allows us to bound the number of first-order operations by
\begin{align*}
		N_{(ii)}& =O\left(\sqrt{1+\frac{\opnorm{\hess f(x_0)}}{\sigma\min(\blambda_0,\lambda_1)}}+\abs*{\log\frac{\lambda_1}{\blambda_0}}\right)\\
	&  \stackrel{}{=} O\left(\sqrt{1+\frac{\opnorm{\hess f(x_0)}}{\blambda_0}}+\log\frac{\lambda_1}{\blambda_0}+\sqrt{1+\frac{\opnorm{\hess f(x_0)}}{\lambda_1}}+\log\frac{\blambda_0}{\lambda_1}\right)\\
	&  \stackrel{}{=} O\left(\sqrt{\frac{L}{\blambda_0}}+\sqrt{\frac{L\norm{x_0-x_\star}^2}{\epsilon}}+\log\frac{\blambda_0}{L}\right),
\end{align*}
 where for the last equality we use $\log(\blambda_0/\lambda_1)\le \log(\blambda_0/L)+\log(L/\lambda_1)$ and the bounds $\sqrt{\frac{L}{\lambda_1}} = O(\sqrt{L A_1}) =  O\left(\sqrt{\frac{L\norm{x_0-x_\star}^2}{\epsilon}}\right)$ and $\lambda_1=O(L)$ as in previous case. Summing up the $N_{(i)}$ and $N_{(ii)}$ yields the the claimed bound.
\end{proof}

\subsection{Auxiliary results}\label{apdx:solvers-aux}

Here, we list technical results invoked throughout the section.

\begin{lemma}\label{lem:ms-sc-bounds}
	Let $f:\R^d\to\R^d$ be convex, and suppose that for some $\sigma\in(0,1)$, $x,y\in\R^d$ and $\lambda>0$ the MS condition
	\begin{equation*}
		\norm*{ x - \prn*{y-\tfrac{1}{\lambda}\grad f(x)} } \le \sigma \norm{x-y}
	\end{equation*}
	holds. 
	Then, for any $\xopt\in\R^d$, we have
	\begin{equation*}
		f(x) \le f(\xopt) + \frac{\lambda}{2}\norm{\xopt-y}^2 - \frac{\lambda(1-\sigma^2)}{2}\norm{x-y}^2.
	\end{equation*}
	Therefore
	\begin{equation*}
		\lambda \ge \frac{2(f(x)-f(\xopt))}{\norm{\xopt-y}^2}
		~~\mbox{and, if $f(x) \ge f(\xopt)$,}~~\norm{x-y} \le \frac{1}{\sqrt{1-\sigma^2}}\norm{\xopt-y}.
	\end{equation*}
\end{lemma}
\begin{proof}
	Let $F(x') = f(x') + \frac{\lambda}{2}\norm{x'-y}^2$. Since $F$ is $\lambda$-strongly convex, we have
	\begin{equation*}
		f(x) + \frac{\lambda}{2}\norm{x-y}^2  - f(\xopt) - \frac{\lambda}{2}\norm{\xopt-y}^2
		=
		F(x) - F(\xopt) \le \frac{\norm{\grad F(x)}^2}{2\lambda}
	\end{equation*}
	for every $\xopt\in\R^d$. 
	Moreover, the MS condition yields
	\begin{equation*}
		\frac{\norm{\grad F(x)}^2}{2\lambda} = \frac{\lambda}{2} \norm*{ x - \prn*{y-\tfrac{1}{\lambda}\grad f(x)} }^2 \le \frac{\lambda\sigma^2}{2} \norm{x-y}^2.
	\end{equation*}
	The lemma follows by substituting back and rearranging.
\end{proof}

\begin{lemma}\label{lem:cr-properties}
	Let $A\in\R^{d\times d}$ be a positive definite symmetric matrix, let $b\in\R^d$, and let $w^\star = A^{-1} b$. The iterates $\{w_t\}$ and residuals $\{r_t=Aw_t - b\}$ of the Conjugate Residuals/MinRes  algorithm~\cite{stiefel1955relaxationsmethoden,fong2012cg} for minimizing $\norm{Aw-b}$ satisfy
	\begin{enumerate}
		\item $\norm{w_t}$ is non-decreasing in $t$ with $\norm{w_\infty} = \norm{w^\star}$,
		\item $\norm{r_t}$ is non-increasing in $t$ with $\norm{r_\infty}= 0$, 
		\item $\norm{r_t} = O\prn*{ \frac{\opnorm{A}\norm{w^\star}}{t^2} }$.
	\end{enumerate}
\end{lemma}
\begin{proof}
	The first two parts of the lemma are Theorems 2.3 and 2.4 of \citet{fong2012cg}, respectively. To show the third part, we cite \citet{lee2021geometric} which give a gradient method that, for any $L$-smooth convex function $h$ with minimizer $\xopt$, produces iterates $x_t$ such that $\norm{\grad h(x_t)} = O( L\norm{\xopt-x_0} / t^2)$~\cite[Corollary 1]{lee2021geometric}. Applying this method to $h(x) = \half x^{\T} A x - b^{\T} x$ with $x_0=0$, which is convex and $\opnorm{A}$-smooth with minimizer $w^\star$, guarantees $\norm{A x_t - b} = \norm{\grad h(x_t)} = O\prn*{ {\opnorm{A}\norm{w^\star}}/{t^2} }$. Moreover, we note that $x_t$ is in the linear span of $\grad h(0), \grad h(x_1), \ldots \grad h(x_{t-1})$ and consequently in the Krylov subspace $\mathrm{span}(b, Ab, \ldots, A^{t-1}b)$. Therefore $\norm{r_t} \le \norm{A x_t - b}$ by definition of the Conjugate Residuals/MinRes method. 
\end{proof}

\begin{lemma}\label{lem:reg-newton-step-properties}
	Let $A\in\R^{d\times d}$ be positive semidefinite, let $b\in\R^d$ and let $\Delta(\lambda) = \norm{(A+\lambda I)^{-1}b}$. Then, for any $\lambda_1 \le \lambda_2$ we have
	\begin{equation*}
		\frac{\lambda_1}{\lambda_2} \Delta(\lambda_1) \le \Delta(\lambda_2) \le \Delta(\lambda_1).
	\end{equation*}
\end{lemma}
\begin{proof}
	This is an immediate consequence of
	\begin{equation*}
		\frac{\lambda_1}{\lambda_2}( A + \lambda_2 I) \preceq
		A + \lambda_1 I \preceq 
		A + \lambda_2 I
	\end{equation*}
	and the fact that if $0 \prec M_1 \preceq M_2$ and $M_1, M_2$ have the same eigenvectors, then $\norm{M_1^{-1}b} \ge \norm{M_2^{-1}b}$ for all $b$, since $M_1 \preceq M_2$ implies that the eigenvalues of $M_2$ majorize the eigenvalues of $M_1$. 
\end{proof}

\begin{lemma}\label{lem:ms-via-second-order-smoothness}
	Let $f:\R^d\to\R$ and $x,y\in\R^d$, and suppose that $\hess f$ is $(H,\nu)$-\Holder continuous for some $\nu\in[0,1]$ in a ball of radius $\norm{x-y}$ around $y$. Then, for any $\lambda$,
	\begin{equation*}
		\norm*{x - \prn*{y- \frac{1}{\lambda}\grad f(x)}} \le 
		\frac{1}{\lambda}\norm{ (\hess f(y) + \lambda I) (x-y) + \grad f(y)} + \frac{H}{(1+\nu)\lambda}\norm{x-y}^{1+\nu}.
	\end{equation*}
\end{lemma}
\begin{proof}
	Let
	\begin{equation*}
		\delta = \grad f(x) - \grad f(y) - \hess f(y)(x-y).
	\end{equation*}
	The local $H$-\Holder continuity of $\hess f$ around $y$ yields
	\begin{equation*}
		\norm{\delta} \le \frac{H}{1+\nu}\norm{x-y}^{1+\nu}.
	\end{equation*}
	Moreover, for 
	\begin{equation*}
		r= (\hess f(y) + \lambda I) (x-y) + \grad f(y)
	\end{equation*} 
	algebraic manipulation yields
	\begin{equation*}%
		x - \prn*{y-\frac{1}{\lambda}\grad f(x)}
		= 
		\frac{1}{\lambda}\prn*{r+\delta},
	\end{equation*}
	and the lemma holds via the triangle inequality.
\end{proof}

For the following lemma, recall the notation $\SsupT = \{t\le T \mid \lambda_t > \blambda_t\}$.

\begin{lemma}\label{lem:yet-another-reverse-holder}
	For every $T>0$, $\nu\in[0,1]$ and $\xopt\in \R^d$ such that $f(x_T) \ge f(\xopt)$, the iterates of \Cref{alg:optms} with $\sigma \in (0.01,0.99)$ and $\alpha \in (1.01, O(1))$ satisfy  
	\begin{equation*}
		\sum_{t \in \SsupT} \prn*{\frac{\norm{\tx_{t} - y}^\nu}{\blambda_{t}}}^{2/(4+3\nu)} = O\prn*{\norm{x_0 - \xopt}^{2\nu/(4+3\nu)} A_T^{2/(4+3\nu)}}.
	\end{equation*}
\end{lemma}
\begin{proof}
	The case $\nu=0$ follows immediately from eq.~\eqref{eq:increase-A} in \Cref{lem:increase}. For $\nu\in(0,1]$, define $u_t = \frac{\norm{\tx_{t} - y}^{2}}{(\blambda_{t})^{2/\nu}} \indic{t\in\SsupT}$ and $v_t = (\blambda_{t})^{1+2/\nu} A_{t}$. The reverse H\"older inequality gives, for any $q>1$,
	\begin{equation}\label{eq:yet-another-reverse-holder}
		\prn*{\sum_{t \le T} u_t^{1/q}}^q \prn*{\sum_{t \le T} v_t^{-1/(q-1)}}^{-(q-1)} \le \sum_{t\le T} u_t v_t.
	\end{equation}
	Substituting back the definitions of $u_t$ and $v_t$ we have
	\begin{equation}\label{eq:utvt-bound}
		\sum_{t\le T} u_t v_t = \sum_{t \in \SsupT} \norm{\tx_{t} - y}^{2} A_{t} \blambda_{t} = O(\norm{x_0 - \xopt}^2),
	\end{equation}
	with the last transition due to eq.~\eqref{eq:main-potential-implied} in \Cref{prop:main-potential} and the fact that $f(x_T) \ge f(\xopt)$. 
	Next, we substitute $q=\frac{4+3\nu}{\nu}$ and note that
	\begin{equation*}
		\sum_{t \le T} v_t^{-\nu/(4+2\nu)} = \sum_{t \le T}  \frac{1}{\sqrt{\blambda_{t}} A_{t}^{\nu/(4+2\nu)}}
		= O\prn*{ \sum_{t \le T} \frac{ (\blambda_{t})^{-1/2} }{ \prn*{\sum_{j\le t}  (\blambda_{j})^{-1/2}}^{\nu/(2+\nu)}}}
	\end{equation*}
	with the last transition due to 
	\begin{equation}\label{eq:At-as-sum-of-sqrt-lambda}
		\sqrt{A_{t'}} = \Omega\prn*{\sum_{t\le t'} \frac{1}{\sqrt{\blambda_{t}}}}
	\end{equation}
	for all $t'$ by eq.~\eqref{eq:increase-A-all} in \Cref{lem:increase}.
	Note that for every non-decreasing sequence $0 = B_0 \le B_1 \le B_2 \le \cdots \le B_T$ we have 
	\begin{equation*}
		\begin{aligned}
			\sum_{t\le T} \frac{B_t - B_{t-1}}{B_t^{\nu/(2+\nu)}}
			& \le \sum_{t\le T} \frac{(B_t^{2/(2+\nu)} - B_{t-1}^{2/(2+\nu)})(B_t^{\nu/(2+\nu)}+B_{t-1}^{\nu/(2+\nu)}) }{B_t^{\nu/(2+\nu)}}\\ 
			& \le 2 \sum_{t\le T}(B_t^{2/(2+\nu)} - B_{t-1}^{2/(2+\nu)}) = 2B_T^{2/(2+\nu)}. 
		\end{aligned}
	\end{equation*}
	Substituting $B_t = \sum_{j\le t}  (\blambda_{j})^{-1/2}$, we have
	\begin{equation*}
		\sum_{t \le T} v_t^{-\nu/(4+2\nu)} = O\prn*{\prn*{\sum_{t\le T} \frac{1}{\sqrt{\blambda_{t}}}}^{2/(2+\nu)}} = O(A_T^{1/(2+\nu)}),
	\end{equation*}
	with the final bound again using~\eqref{eq:At-as-sum-of-sqrt-lambda}. This  implies 
	\begin{equation}\label{eq:vt-bound}
		\left(\sum_{t \le T} v_t^{-\nu/(4+2\nu)}\right)^{(4+2\nu)/\nu}  = O(A_T^{2/\nu}).
	\end{equation}
	Substituting $q=(4+3\nu)/\nu$ and the bounds~\eqref{eq:vt-bound} and~\eqref{eq:utvt-bound} into~\eqref{eq:yet-another-reverse-holder} completes the proof.
\end{proof}

\begin{lemma}\label{lem:A-mult-stability}
	For every $T>0$, the sequence $\{A_t\}$ in \Cref{alg:optms} with $\alpha \in (1.01, O(1))$ satisfies $A_{t+1} = O(A_t)$.
\end{lemma}
\begin{proof}
	Note that
	\begin{equation*}
		\bA_{t+1} - A_t = \ba_{t+1} = \sqrt{\frac{\bA_{t+1}}{\blambda_{t+1}}} 
	\end{equation*}
	and therefore (since $\bA_{t+1} > A_t$)
	\begin{equation*}
		\sqrt{\bA_{t+1}} \le \sqrt{A_t} + \frac{1}{\sqrt{\blambda_{t+1}}}
		\overle{(i)} \sqrt{A_t} + O\prn*{\frac{1}{\sqrt{\blambda_{t}}}} \overle{(ii)} O(\sqrt{A_t})
	\end{equation*}
	due to $(i)$ $\blambda_{t+1} \ge \blambda_{t} / \alpha = \Omega(  \blambda_{t} )$ by the algorithm's construction and $(ii)$ $\frac{1}{\sqrt{\blambda_{t}}} = O(\sqrt{A_t})$ by the last inequality in \Cref{lem:increase}. The proof is complete by noting that $A_{t+1} \le \bA_{t+1}$. 
\end{proof}
\section{Experiments}\label{app:experiments}

This section provides the full details of the experiment we report in \Cref{sec:experiments} (in \Cref{app:experiments-technical-details}), as well as results of the additional experiments: algorithm comparison across additional datasets (\Cref{app:experiments-extra-experiments}), parameter sensitivity of our algorithm (\Cref{app:experiments-parameter-sensitivity}), effect of changing the parameter $M$ in $\oracle[cr]$ (\Cref{app:experiments-tuning-cr}), and the effect of momentum on the worst-case instance for Lipschitz-Hessian functions (\Cref{app:experiments-worst-case}). Finally, we also demonstrate empirically the importance of the momentum damping mechanism in \Cref{alg:optms} (\Cref{app:experiments-momentum-damping-mechanism}).

\subsection{Main experiment details}\label{app:experiments-technical-details}
We report experiments for logistic regression objectives of the form
\begin{align*}
    f(x) \defeq \frac{1}{n}\sum_{i \in [n]} \log \prn*{1+ \exp\prn*{-c_i \phi_i^{\T}x}}, 
\end{align*}
where each $\phi_i \in \R^{d}$ is a feature vector with a corresponding label $c_i \in \{-1,1\}$. 

\paragraph{Implementation details.} 
We now provide the key implementation details for all the algorithms considered in our experiments. For a complete description please refer to the Python implementations submitted with this manuscript. 

\begin{itemize}[leftmargin=*]
	\item \textbf{\Cref{alg:optms}.} Our implementation of \Cref{alg:optms} follows its pseudocode precisely. We keep $\blambdainit=0.1$ throughout and set $\alpha=2$ in all experiments except for those in \Cref{app:experiments-parameter-sensitivity}, where we test how changing it affects performance. 
	
	\newcommand{\iblambda}{\lambda^0}
	
	\item \textbf{\Cref{alg:ms}~\cite{MonteiroS13a}.} A direct implementation of the pseudocode of \Cref{alg:ms} would be quite inefficient, since stating off the bisection with a large interval $[\lambda^\ell, \lambda^h]$ at each iterations will waste many oracle calls. Instead, we implement a \emph{strong baseline} for our bisection-free algorithm by starting each bisection with a guess $\iblambda_{t+1}$ determined by the previous iterations. We construct this guess using the scheme\footnote{We also experimented with the heuristic $\iblambda_{t+1} = \half \lambda_t$, which performed slightly worse.} for updating $\blambda_t$ in \Cref{alg:optms}: if the previous final bisection output $\lambda_{t}$ and the previous initial bisection guess $\iblambda_{t}$ satisfy $\lambda_t > \iblambda_t$, we let $\iblambda_{t+1} = 2 \iblambda_{}$ and otherwise we set $\iblambda_{t+1} = \half \iblambda_{t}$. We take $\iblambda_1 = 0.1$. 
	
	To construct a bisection interval out of the initial guess $\iblambda_{t+1}$, we adopt a strategy similar to the ones used in \Cref{alg:adaptive-msn-step,alg:adaptive-msn-cg-step}. To explain it, define the following terminology. Consider some $\blambda_{t+1}$ and $\lambda_{t+1}$ computed by applying an MS oracle to $y_t$ and $\blambda_{t+1}$, with $y_t$ computed from $\blambda_{t+1}$ as in lines~\ref{line:update-standard-start} and~\ref{line:update-standard-end} of \Cref{alg:ms}. We say that $\blambda_{t+1}$ is \emph{valid} if $\lambda_{t+1} \in [\frac{1}{\rho}\blambda_{t+1}, \blambda_{t+1}]$, that $\lambda_{t+1}$ is \emph{high} if $\lambda_{t+1} < \frac{1}{\rho}\blambda_{t+1}$, and that $\blambda_{t+1}$ is \emph{low} if $\lambda_{t+1} > \blambda_{t+1}$. If $\iblambda_{t+1}$ is valid, we simply use it and there is no need for bisection. Otherwise, if it is low, we take $\lambda^\ell_{t+1} = \iblambda_{t+1}$, and repeatedly double $\iblambda_{t+1}$ until we find some $2^k \iblambda_{t+1}$ that is either valid or high. In the former case we are again done, and in the latter case we set $\lambda^h_{t+1} = 2^k \iblambda_{t+1}$ and continue with the bisection as described in \Cref{alg:ms}, except that (inspired by \Cref{alg:adaptive-msn-step}) at each iteration we take $\blambda_{t+1}$ to be the geometric mean of $\lambda^\ell_{t+1}$ and $\lambda^h_{t+1}$ rather than the arithmetic mean. The case that $\iblambda_{t+1}$ is high is treated analogously, setting $\lambda^h_{t+1} = \iblambda_{t+1}$ and repeatedly halving it until finding $2^{-k}\iblambda_{t+1}$ that is either valid or low. 
	
	Finally, we note that for $\sigma$-MS oracles with $\sigma>0$ the bisection is only guaranteed to succeed when $\rho$ is sufficiently large. The precise value of $\rho$ depends on $\sigma$ and the order of the movement bound guaranteed by the oracle. Instead of attempting a precise calculation, we set $\rho=4$.
	
	\item \textbf{Cubic-regularized Newton Method (CR)~\cite{nesterov2006cubic}.} The method consists of simply iterating $x_{t+1} = \oracle[cr](x_t)$. For when the parameter $M$ in $\oracle[cr]$ is set to $M=0$, the method reduces to the classical \textbf{Newton's method} $x_{t+1} = -[\hess f(x_t)]^{-1} \grad f(x_t)$. 
	
	\item \textbf{Accelerated CR (ACR)~\cite{nesterov2008accelerating}.} We implement \cite[Alg.~4.8]{nesterov2008accelerating} without changes.
	
	\item \textbf{Adaptive ACR~\cite{grapiglia2020tensor}.}  We implement \cite[Alg.~4]{grapiglia2020tensor} without changes.
	
	\item \textbf{\citet{song2021unified} heuristic.} Following a proposal in~\cite{song2021unified}, we consider a version of \Cref{alg:ms} that uses a single pre-specified sequence of $\blambda_t$ without checking whether the resulting $\lambda_t$ is in the interval $[\frac{1}{\rho}\blambda_t, \blambda_t]$. We compute the sequence by setting $A_t = \bA_t =  \frac{1}{2HR} (t/3)^{7/2}$, and taking $\blambda_{t+1} = \frac{A_{t+1}}{a_{t+1}^2} = \frac{A_{t+1}}{(A_{t+1}-A_t)^2}$.  Here $H$ is an estimate of the function's Hessian Lipschitz constant (see below), and $R$ is an estimate of the Euclidean distance between $x_0$ from an optimal point. We obtain $R$ by using the default scikit-learn logistic regression solver~\cite{scikit-learn}; it finds a far less accurate solution than the methods we consider but provides a reasonably accurate estimate of $R$.

	\item \textbf{Iterating $\adaOracle$.} This scheme corresponds to simply iterating $x_{t+1}, \lambda_{t+1} = \adaOracle(x_t, \lambda_{t}/2)$, with the initial $\lambda_1$ set to $0.1$. 

	\item \textbf{Iterating $\foadaOracle$.} This scheme corresponds to simply iterating $x_{t+1}, \lambda_{t+1} = \foadaOracle(x_t, \lambda_{t}/2)$, with the initial $\lambda_1$ set to $0.1$. 
	
	\item \textbf{Gradient descent (GD).} We iterate $x_{t+1} = x_t - \eta \grad f(x_t)$ and choose the best value of $\eta$ from $\crl*{3, 10, 30, 100, 1000, 3000}$, making sure the best value is never on the edge of the grid,  i.e., $3$ and $3000$ are never chosen. 
	
	\item \textbf{Accelerated gradient descent (AGD)~\cite{nesterov1983method}.} We implement the algorithm precisely as described in \cite{nesterov1983method}, and tune the step size $\eta$ as described for GD.
		
	\item \textbf{L-BFGS-B~\cite{byrd1995limited,zhu1997algorithm}.} We use the implementation available from SciPy~\cite{scipy2020}, where we set all tolerance parameter to a very small value so that the algorithm only stops after exceeding the specified maximum number of iterations. 
	
	\item \textbf{$\oracle[cr].$} To solve the problem~\eqref{eq:cr-oracle} and implement $\oracle[cr]$, we perform a bisection over $\lambda$ to solve for $\lambda$ that satisfies $\lambda \approx \frac{M}{2}\norm{[\hess f(y) + \lambda I]^{-1}\grad f(y)}$, and return $x=y-[\hess f(y) + \lambda I]^{-1}\grad f(y)$. To ensure a high-quality solution to the implicit equation for $\lambda$, we stop the bisection only when $\frac{\lambda}{\frac{M}{2}\norm{[\hess f(y) + \lambda I]^{-1}\grad f(y)}} \in [1-10^{-5}, 1+10^{-5}]$. This results in a slow implementation of $\oracle[cr]$ (requiring a lot of linear system solutions), but provides the ideal point of comparison since we measure complexity by number of Hessian evaluations. To ensure numerical stability, we also stop the bisection if the value of $\lambda$ falls below $\lambda_{\mathrm{Newton}} = 10^{-10}$. 
	
	\item \textbf{$\adaOracle$.} Our implementation follows the pseudocode of \Cref{alg:adaptive-msn-step} precisely, except that, to ensure numerical stability, we stop the procedure if  $\lambda$ falls below $\lambda_{\mathrm{Newton}} = 10^{-10}$. 
	When combining the oracle with \Cref{alg:optms}, we set the $\lazyflag$ to be True in all iterations except the first, as in \Cref{coro:adaptive-msn-step}. In all other settings we set $\lazyflag$ to be False. 
	
	\item \textbf{$\foadaOracle$.} Our implementation follows the pseudocode of \Cref{alg:adaptive-msn-cg-step}, except we only implement the case that  $\lazyflag=$True, since doing otherwise appears less practical; it is not hard to extend \Cref{thm:main} and \Cref{coro:adaptive-msn-step} to provide similar guarantees even when the first iteration is lazy.  
\end{itemize}

\paragraph{Initialization.} We initialize all algorithms at the origin, i.e., with $x_0 = 0$.

\paragraph{Estimating the Hessian Lipschitz constant.} 
For all algorithms that require an estimate for the Lipschitz constant $H$ of $\hess f$ (i.e., all the algorithms that use $\oracle[cr]$), we set $H = \frac{1}{10}\bar{H}$, where 
$\opnorm{\frac{1}{n}\sum_{i=1}^{n}\phi_i \phi_i^{\T}}\max_{i \in [n]}\norm*{\phi_i}$ is a conservative upper bound on the Lipschitz constant of $\hess f$ for logistic regression~\cite{song2021unified}.  We explore the effect of varying the estimate $H$ in \Cref{app:experiments-tuning-cr}. Note that value of $M$ given to $\oracle[cr]$ is typically $2H$: for \Cref{alg:ms,alg:optms} it is $H/\sigma$ and $\sigma=2$, while CR and ACR also use $M=2H$. We also use $H$ as the initial guess for Adaptive ACR.

\paragraph{Datasets and preprocessing.}
We compare our methods to other baselines using the following binary classification datasets: 
\begin{itemize}[leftmargin=*]
	\item \textbf{a9a} ($n=32,561$ and $d=123$) 
	\item  \textbf{w8a} ($n=49,749$ and $d=300$)
	\item \textbf{splice} ($n=1,000$ and $d=60$)
	\item \textbf{synthetic} ($n=500$ and $d=200$).
\end{itemize}
The first three datasets are from LIBSVM~\cite{chang2011libsvm}, which is available under a BSD 3-Clause "New" or "Revised" license. The synthetic dataset is generated by sampling half of the data points from $\mc{N}_1(\mu_1, I)$ and the other half from $\mc{N}_1(\mu_2, I)$, where $\mu_1, \mu_2 \in \R^d$
are independent random vectors uniformly drawn from a sphere with radius $0.5$.

For all datasets we normalize the feature vectors, such that for every $i \in [n]$, each feature vector $\phi_i$ is a unit norm.

\subsection{Replicating \Cref{fig:methods-comp} with additional datasets}\label{app:experiments-extra-experiments}

In \Cref{fig:methods-comp-app} we compare all the algorithms described above on logistic regression with the datasets: ``a9a'' (panels a-c), ``w8a'' (panels d-f), ``splice'' (panels g-i) and the synthetic dataset (panels j-l).

Among non-adaptive methods (panels a, d, g and j), \Cref{alg:optms} outperforms the other non-adaptive methods while \Cref{alg:ms} is consistently the second best-performing method.

Comparing adaptive methods (panels b, e, h and k), we see that our implementation of \Cref{alg:optms} with the adaptive oracle $\adaOracle$ converges faster than adaptive ACR and \Cref{alg:ms} with $\adaOracle$ for all datasets. However, our scheme that only iterates $\adaOracle$ without momentum converges even faster, and Newton's method outperforms all second-order methods.

For first-order methods (panels c, f, i and l),  iterating $\foadaOracle$ scheme is comparable to L-BFGS-B on 2 out of 4 datasets and is faster than tuned AGD in 3 out of 4 datasets. On the synthetic dataset it is about twice slower than L-BFGS-B but still faster than tuned AGD, while on w8a it is about 50\% slower than L-BFGS-B and tuned AGD, which perform comparably.

\begin{figure}[p]
	\centering
   \includegraphics[width=0.9\textwidth]{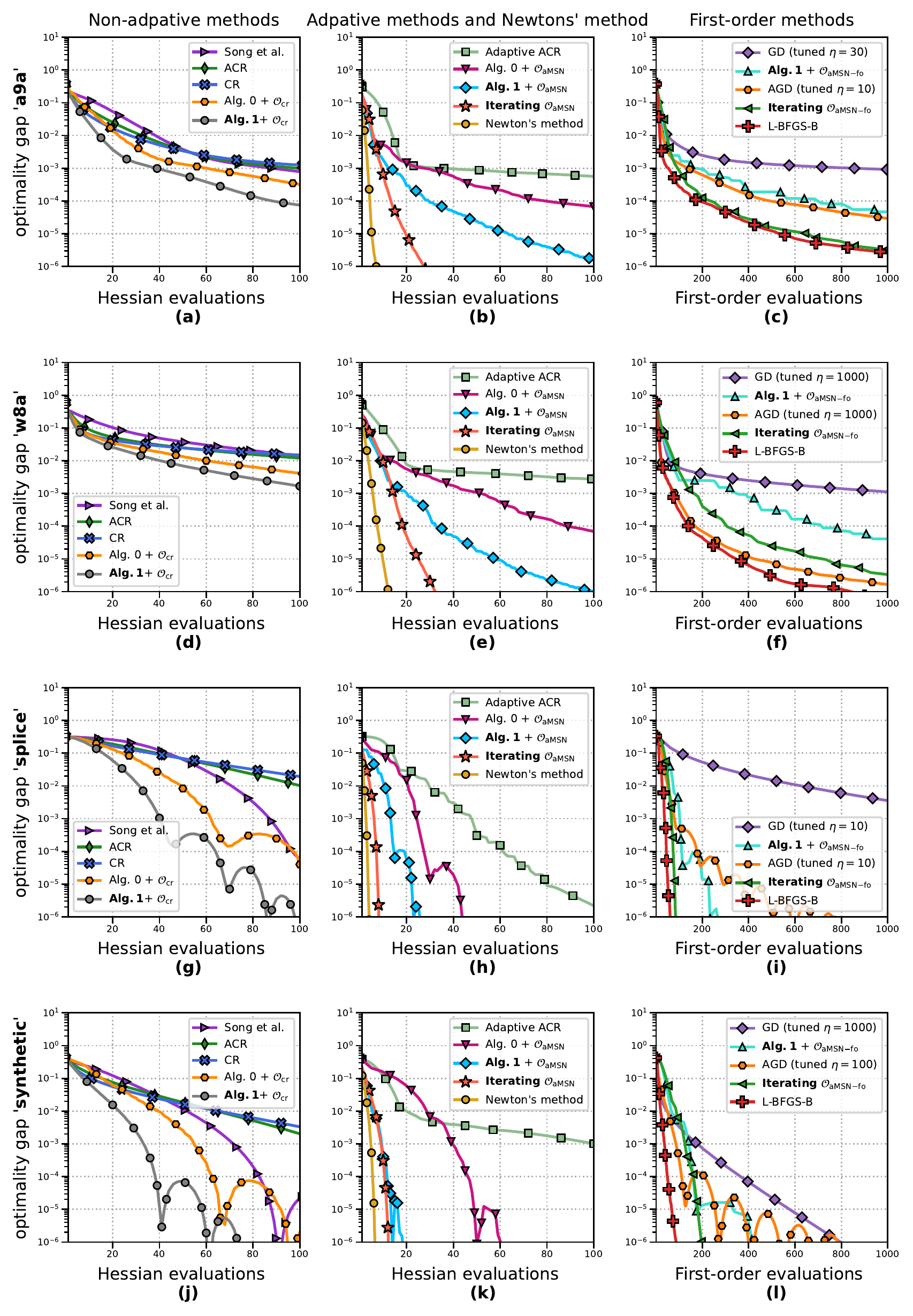}
   \caption{Empirical results on logistic regression with ``a9a'', ``w8a'', ``splice'' and a synthetic dataset. Boldface legend entries denote methods we contribute.
   }\label{fig:methods-comp-app}
\end{figure}

\subsection{Parameter sensitivity of \Cref{alg:optms}}\label{app:experiments-parameter-sensitivity}
We test the sensitivity of \Cref{alg:optms} combined with our adaptive oracle (\Cref{alg:adaptive-msn-step} or \Cref{alg:adaptive-msn-cg-step}) to the parameters $\alpha$ and $\sigma$. \Cref{fig:param-sensitivity} shows that \Cref{alg:optms} second-order oracle $\adaOracle$ performs essentially the same for all $\alpha$ in the range $1.2$ to $8$, and that the oracle's performance is similar for $\sigma=0.1$ and $\sigma=0.25$, but slightly degrades for larger and smaller $\sigma$.  \Cref{alg:optms} combined with the first-order oracle $\foadaOracle$ is a bit more sensitive to $\alpha$ (performing best for $\alpha=1.2$), but is less sensitive to $\sigma$, showing similar performance for all $\sigma$ values except the very smallest $\sigma=0.01$.

\begin{figure}[tb]
	\centering
    \begin{subfigure}[b]{0.4\textwidth}
        \includegraphics[width=1\textwidth]{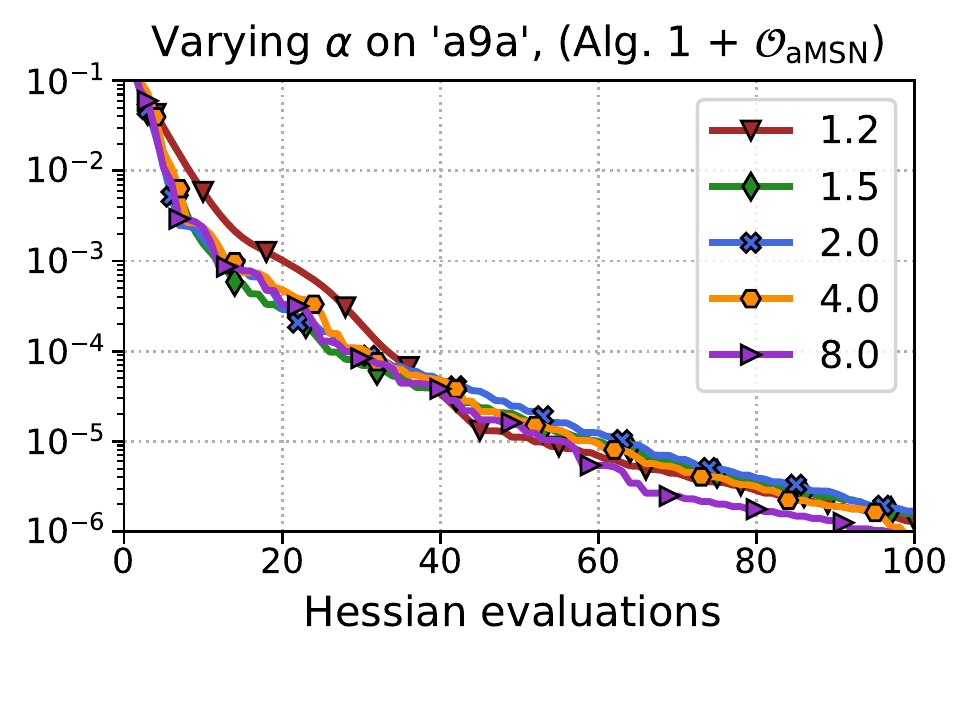}
    \end{subfigure}
    \hspace{0.1cm}
    \begin{subfigure}[b]{0.40\textwidth}
        \includegraphics[width=1\textwidth]{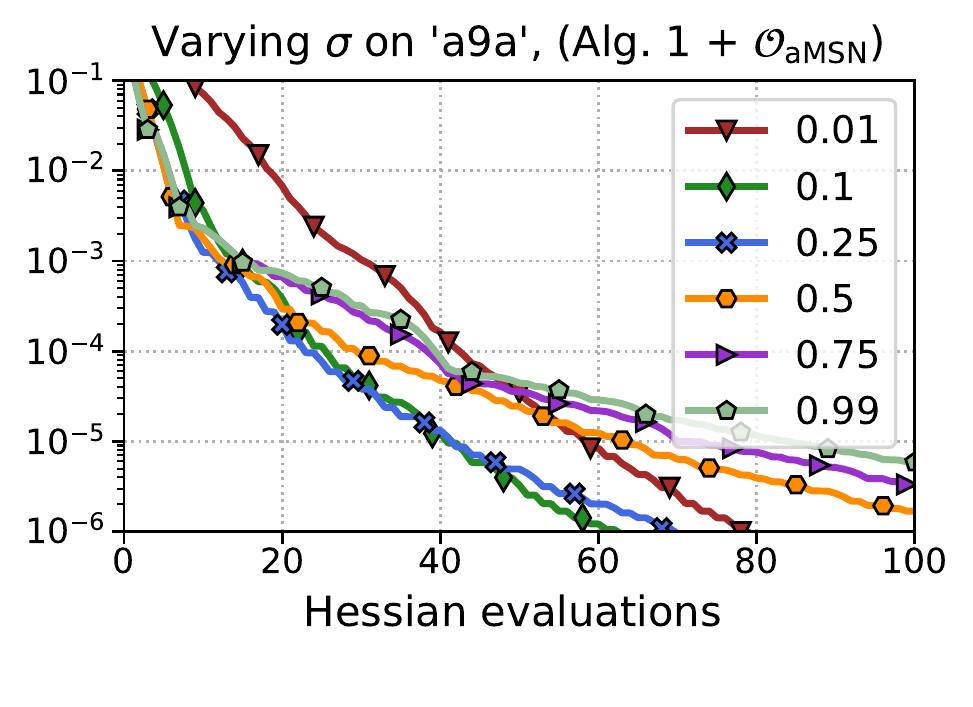}
    \end{subfigure}
    \vfill
    \begin{subfigure}[b]{0.40\textwidth}
        \includegraphics[width=1\textwidth]{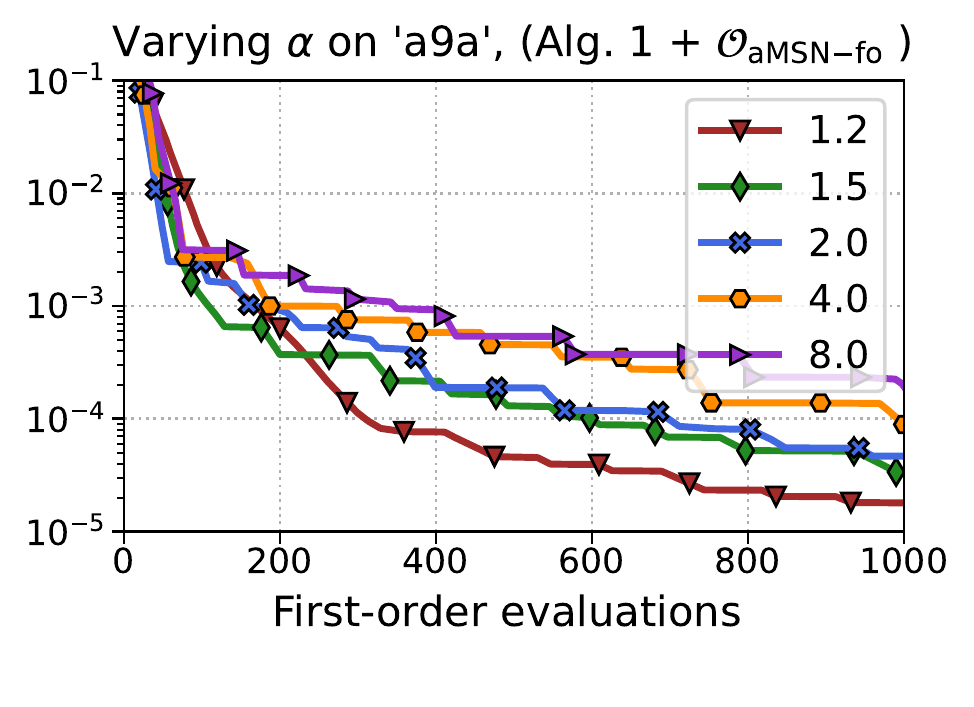}
    \end{subfigure}
    \begin{subfigure}[b]{0.40\textwidth}
    \hspace{0.1cm}
        \includegraphics[width=1\textwidth]{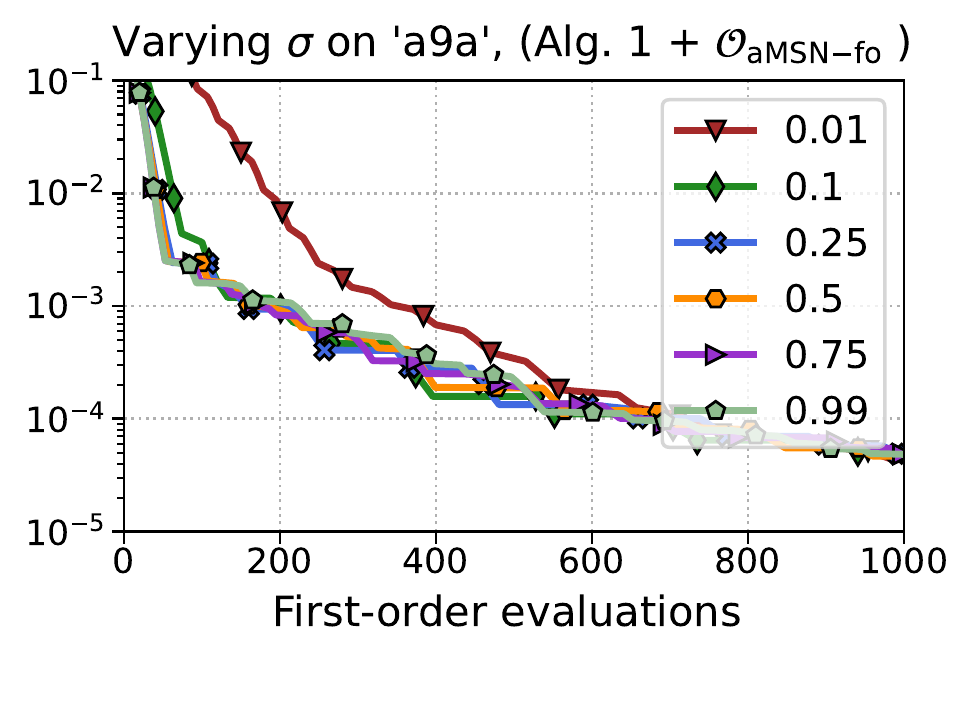}
    \end{subfigure}
    \caption{Testing the sensitivity of \Cref{alg:optms} to the parameters $\alpha$ and $\sigma$ with the``a9a'' dataset.
    }\label{fig:param-sensitivity}
\end{figure}

\subsection{Varying $M$ for $\oracle[cr]$}\label{app:experiments-tuning-cr}
In this section we test the performance of non-adaptive methods (i.e., the methods that use $\oracle[cr]$) when changing the estimate of the function's Lipschitz constant $H$. In particular, we consider values of $H$ of the form $\beta \bar{H}$, where $\bar{H}=\opnorm{\frac{1}{n}\sum_{i=1}^{n}\phi_i \phi_i^{\T}}\max_{i \in [n]}\norm*{\phi_i}$ is an upper bound on the Hessian Lipschitz constant and $\beta$ varies in $\crl*{1, 10^{-1}, 10^{-2}, 10^{-3}, 10^{-4}, 10^{-5}, 10^{-6}, 10^{-7}, 10^{-8}}$. (The experiments in \Cref{app:experiments-extra-experiments} correspond to $\beta=0.1$). \Cref{fig:tuning-cr} shows that our adaptive accelerated scheme (\Cref{alg:optms} with the $\adaOracle$) outperforms all non-adaptive methods with their optimal $H$ value, except for the CR method that has optimal $H  \approx 0$ and therefore is almost equivalent to Newton's method.

\begin{figure}[tb]
	 \centering
	\includegraphics[width=0.9\textwidth]{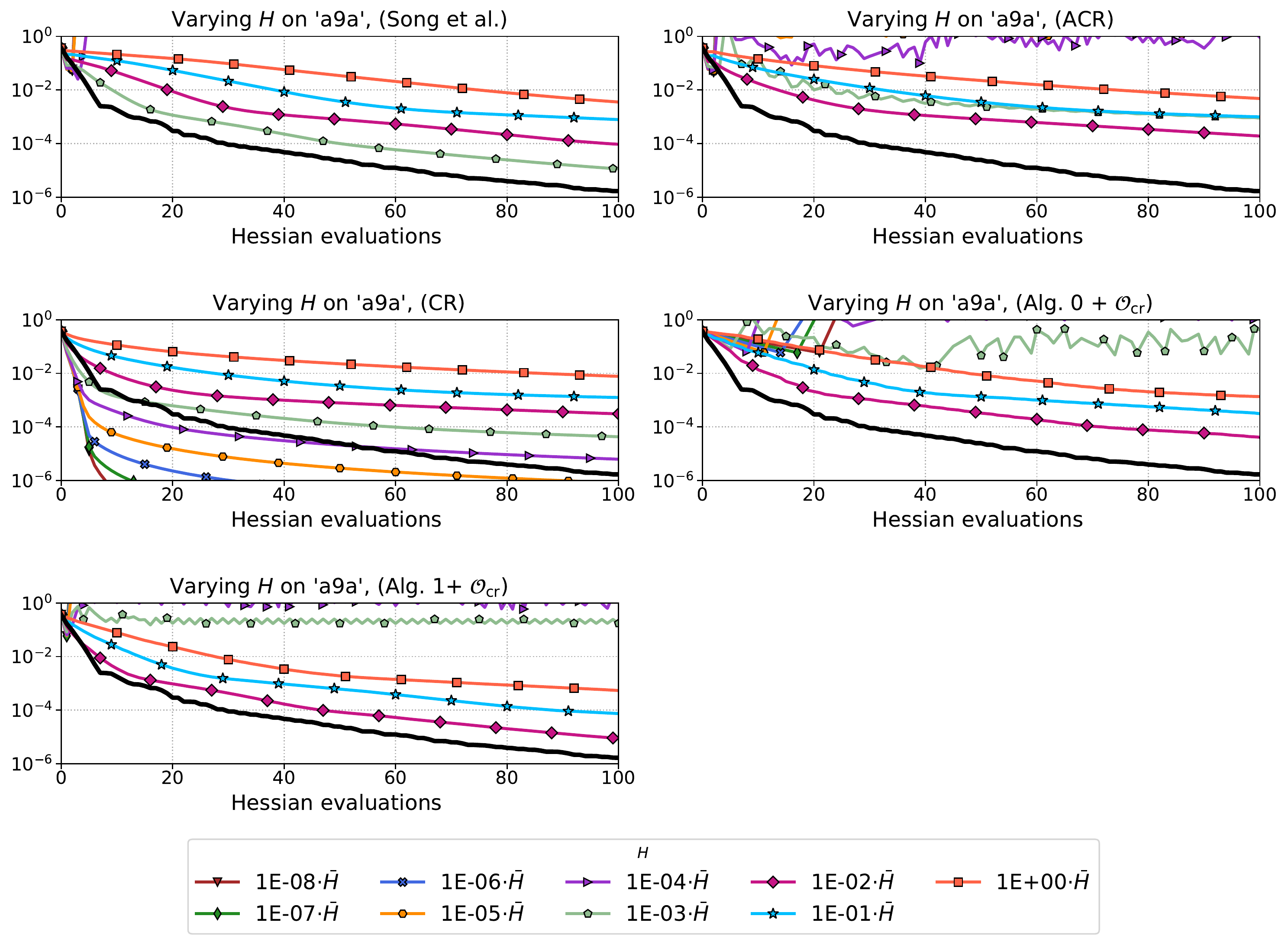}
	\caption{Varying the estimated Hessian Lipschitz constant $H$ non-adaptive methods. The thick black line corresponds our adaptive method (\Cref{alg:optms} with $\adaOracle$).}\label{fig:tuning-cr}
\end{figure}

\subsection{Performance on a worst-case instance}\label{app:experiments-worst-case}

Having observed that our adaptive oracle $\adaOracle$ performs better on logistic regression without the ``acceleration'' scheme in \Cref{alg:optms}, we now test whether \Cref{alg:optms} demonstrably accelerates $\adaOracle$ on a different, harder problem. In particular, we consider the worst case instance \cite{arjevani2019oracle,grapiglia2020tensor,doikov2020inexact} $f:\R^d \to \R$ given by
\begin{align*}
    f(x) = \abs*{x^{(1)}-1}^3 + \sum_{i=2}^{d}\abs*{x^{(i)}-x^{(i-1)}}^3.
\end{align*}
We note that, for $t<d$, the optimal rate of convergence for any of the algorithms we consider (which can only ``discover'' one coordinate of $f$ per iterations) is $O(t^{-2})$, or $O\prn*{\norm[\big]{x_0-\xopt^{(t)}}^3 t^{-3.5}}$ where  $\norm[\big]{x_0-\xopt^{(t)}}=\Theta(\sqrt{t})$ is the distance between the initial point and the best solution with only $t$ non-zero coordinates.

In our experiments, we set $d=3,000$ and compare the convergence rate of the following second-order methods: standard cubic regularized method (CR), its accelerated variant (ACR), \Cref{alg:optms} with $\oracle[cr]$, \Cref{alg:optms} with the adaptive oracle $\adaOracle$, and iterating the oracle $\adaOracle$. For methods based on $\oracle[cr]$ we estimate the Hessian Lipschitz constant to be $H=10$.  \Cref{fig:worst-case} shows that the slope of the accelerated methods using $\oracle[cr]$ (ACR and \Cref{alg:optms}) is sharper than the slope of the CR method, indicating a faster convergence rate due to the acceleration scheme.
However, the convergence rate $\adaOracle$ with and without the acceleration component is optimal. Therefore, even the worst-case instance for convex optimization with Lipschitz Hessian does not provide evidence that acceleration significantly benefits $\adaOracle$.

\begin{figure}[tb]
	 \centering
	\includegraphics[width=0.75\textwidth]{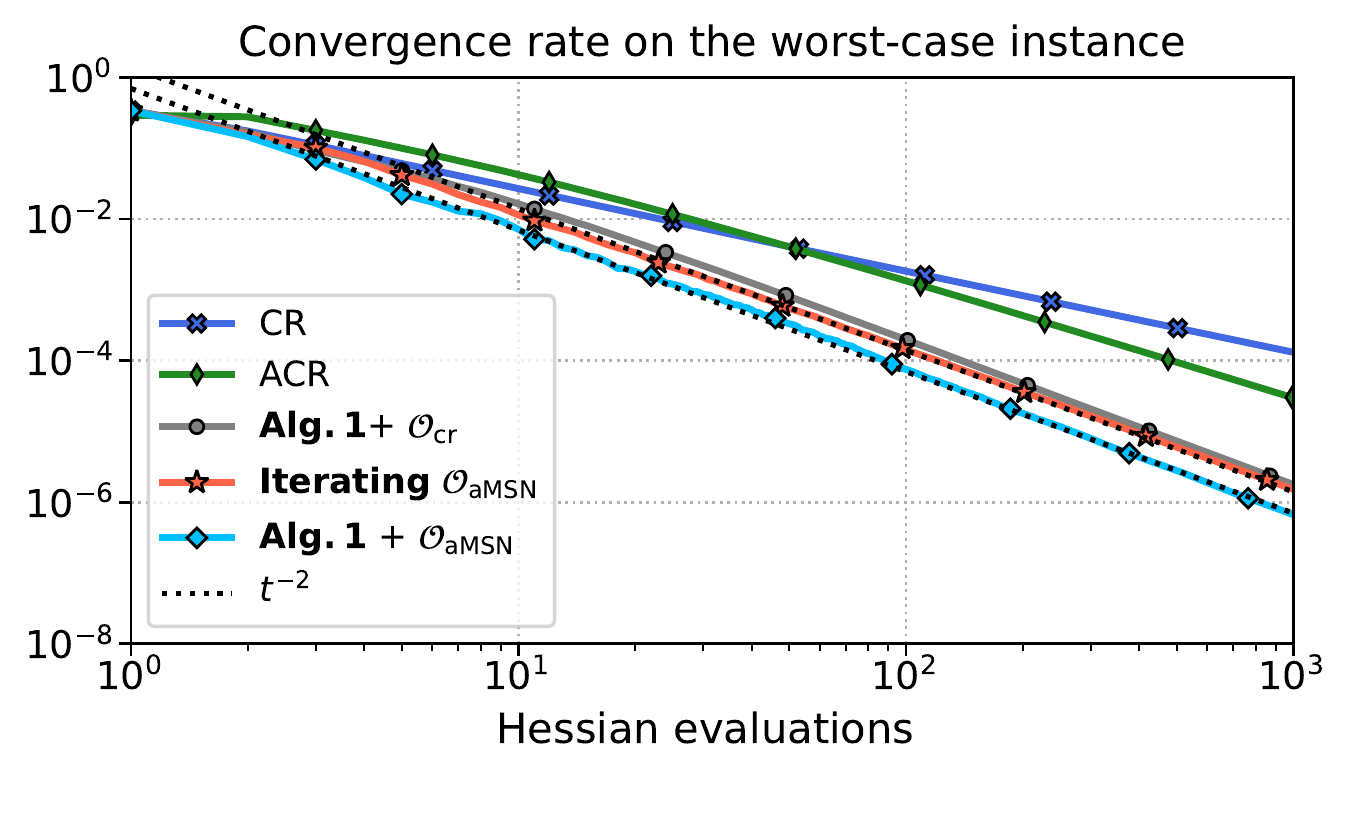}
	\caption{Empirical results on the worst case instance (the x-axis and y-axis are in logarithmic scale). Boldface legend entries denote methods we contribute.}
\label{fig:worst-case}
\end{figure}

\subsection{The importance of momentum damping in  \Cref{alg:optms}}\label{app:experiments-momentum-damping-mechanism}
We compare our method (\Cref{alg:optms} with $\adaOracle$) to a variant of it that does not use the momentum damping mechanism. That is, we set $x_{t+1}=\tx_{t+1}$ and $a_{t+1}=\ba_{t+1}$ regardless of the value of $\lambda_{t+1}$. As \Cref{fig:damping-mechanism} shows, without the momentum damping mechanism \Cref{alg:optms} fails to converge on all the datasets we test (``a9a'', ``w8a'', ``splice'', and the synthetic dataset). 

\begin{figure}[tb]
	\centering
   \includegraphics[width=0.9\textwidth]{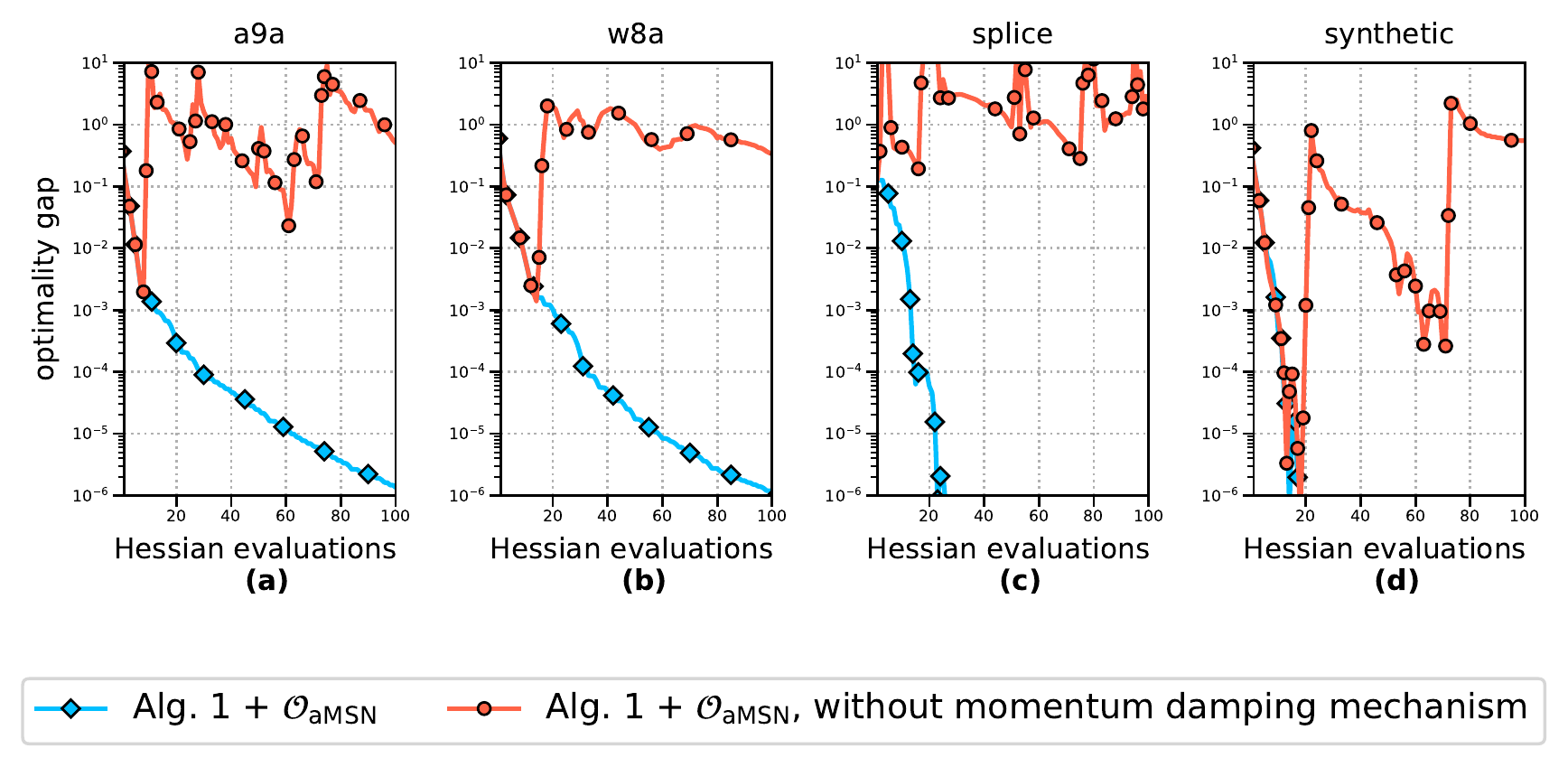}
   \caption{\Cref{alg:optms} with (light blue line) and without (red line) the momentum damping mechanism. Title denotes the dataset name.}
\label{fig:damping-mechanism}
\end{figure}

\end{document}